\documentclass[12pt]{amsart}
\usepackage{amssymb, eucal, amsfonts, amsmath, xypic,latexsym}
\newtheorem{theo}{Theorem}[section]
\newtheorem{lemm}[theo]{Lemma}
\newtheorem{cor}[theo]{Corollary}
\newtheorem{prop}[theo]{Proposition}
\newtheorem{defi}{Definition}[section]
\newtheorem{rem}[theo]{Remark}
\numberwithin{equation}{section}
\newcommand{\Der}{{\rm Der\:}}
\newcommand{\ad}{{\rm ad\:}}

\newcommand{\Z}{{\Bbb Z}}
\newcommand{\F}{{\Bbb F}}

\newcommand{\rad}{{\rm rad}}
\newcommand{\un}{\underline}
\newcommand{\ot}{\otimes}
\begin{document}
\title{Simple Lie algebras of small characteristic VI. Completion of the classification}
\author{\sc Alexander Premet and Helmut Strade}
\address
{School of Mathematics, The University of Manchester, Oxford Road,
M13 9PL, United Kingdom} \email{sashap@maths.man.ac.uk}
\address{
Fachbereich Mathematik, Universit{\"a}t Hamburg, Bundesstrasse 55,
20146 Hamburg, Germany} \email{strade@math.uni-hamburg.de}
\thanks{
Primary 17B20, 17B50}
\begin{abstract}
\noindent Let $L$ be a finite-dimensional simple Lie algebra over an
algebraically closed field of characteristic $p>3$. It is proved in
this paper that if the $p$-envelope of $\ad L$ in $\Der L$ contains
a torus of maximal dimension whose centralizer in $\ad L$ acts
nontriangulably on $L$, then $p=5$ and $L$ is isomorphic to one of
the Melikian algebras ${\mathcal M}(m,n)$. In conjunction with
\cite[Thm.~1.2]{PS5}, this implies that, up to isomorphism, any
finite-dimensional simple Lie algebra over an algebraically closed
field of characteristic $p>3$ is either classical or a filtered Lie
algebra of Cartan type or a Melikian algebra of characteristic $5$.
This result finally settles the classification problem for
finite-dimensional simple Lie algebras over algebraically closed
fields of characteristic $\ne 2,3$.
\end{abstract}
\maketitle

\section{\bf Introduction}
This paper concludes the series  \cite{PS1}, \cite{PS2}, \cite{PS3},
\cite{PS4}, \cite{PS5}. Its goal is to finish the proof of the
following theorem which was announced in \cite{St04} and \cite{PS}:
\begin{theo}[Classification Theorem]\label{1.1}
Any finite-dimensional simple Lie algebra over an algebraically
closed field of characteristic $p>3$ is of classical, Cartan or
Melikian type.
\end{theo}

For $p>7,$ the finite-dimensional simple Lie algebras were
classified by the second author in the series of papers \cite{St89},
\cite{St91}, \cite{St92}, \cite{St93}, \cite{St94}, \cite{St98}. It
should be mentioned that the Classification Theory was inspired by
the ground-breaking work of Block--Wilson \cite{BW82}, \cite{BW} who
handled the so-called restricted case (also for $p>7$).

In what follows, $F$ will denote an algebraically closed field of
characteristic $p>3$, and $L$ will always stand for a
finite-dimensional simple Lie algebra over $F$. As usual, we
identify $L$ with the subalgebra $\ad L$ of the derivation algebra
$\Der L$ and denote by $L_p$ the semisimple $p$-envelope of $L$ (it
coincides with the $p$-closure of $\ad L$ in the restricted Lie
algebra $\Der L$). Given a torus $T$ of maximal dimension in $L_p$
we let $H$ stand for the centralizer of $T$ in $L$; that is,
$$H:=\,{\mathfrak c}_{L}(T)=\{x\in L\,|\,\,[t,x]=0\ \,\,\forall
\,t\in T\}.$$ Let $\Gamma(L,T)$ be the set of roots of $L$ relative
to $T$; that is, the set of all {\it nonzero} linear functions
$\gamma\in T^*$ for which the subspace $L_\gamma:=\{x\in
L\,|\,\,[t,x]=\gamma(t)x\ \,\,\forall\, t\in T\}$ is nonzero. Then
$H$ is a nilpotent subalgebra of $L$ (possibly zero) and $L$
decomposes as
$L\,=\,H\oplus\,\bigoplus_{\gamma\in\Gamma(L,T)\,}L_\gamma.$ By
\cite[Cor.~3.7]{PS4} any root $\gamma$ in $\Gamma(L,T)$ is either
{\it solvable} or {\it classical} or {\it Witt} or {\it
Hamiltonian}. Accordingly, the semisimple quotient
$L[\gamma]=L(\gamma)/\rad\, L(\gamma)$ of the $1$-section
$L(\gamma):=H\oplus\,\bigoplus_{i\in\F_p^\times\,}L_{i\gamma}$ is
either $(0)$ or $\mathfrak{sl}(2)$ or the Witt algebra
$W(1;\underline{1})$ or contains an isomorphic copy of the
Hamiltonian algebra $H(2;\un{1})^{(2)}$ as an ideal of codimension
$\le 1$. For $\alpha,\beta\in\Gamma(L,T)$ we denote by
$L(\alpha,\beta)$ the $2$-section $\sum_{i,j\in{\Bbb
F}_p}L_{i\alpha+j\beta}$, where $L_0=H$ by convention.

We say that $T$ is {\it standard} if $H^{(1)}$ consists of nilpotent
derivations of $L$ and {\it nonstandard} otherwise. In \cite{PS4}
and \cite{PS5}, it was shown that if all tori of maximal dimension
in $L_p$ are standard, then $L$ is either classical or a filtered
Lie algebra of Cartan type. On the other hand, the main results of
\cite{P94} imply that if $L_p$ contains a nonstandard torus of
maximal dimension, say $T'$, then there are
$\alpha,\beta\in\Gamma(L,T')$ such that the factor algebra
$L(\alpha,\beta)/{\rm rad}\,L(\alpha,\beta)$ is isomorphic to the
restricted Melikian algebra ${\mathcal M}(1,1)$. In particular,
$p=5$ in this case.

\medskip

The main result of the present paper is the following:
\begin{theo}\label{1.2} If the semisimple $p$-envelope
of $L$ contains nonstandard tori of maximal dimension, then  $L$ is
isomorphic to one of the Melikian algebras ${\mathcal M}(m,n)$,
where $(m,n)\in{\mathbb N}^2$.
\end{theo}
\noindent Together with the main results of \cite{PS4} and
\cite{PS5} Theorem~\ref{1.2} implies the Classification Theorem. In
view of \cite[Cor.~7.2.3]{St04} we also obtain:
\begin{cor}\label{1.3}
Any finite-dimensional restricted simple Lie algebra over an
algebraically closed field of characteristic $p>3$ is, up to
isomorphism, either one of $W(n;\underline{1}),\, n\ge 1,\,$
$S(n;\underline{1})^{(1)},\,n\ge 3,\,$
$H(2r;\underline{1})^{(2)},\,r\ge 1,\,$
$K(2r+1;\underline{1})^{(1)},\,r\ge 1$, ${\mathcal M}(1,1)$, or has
the form $({\rm Lie}\,G)^{(1)}$, where $G$ is a simple algebraic
$F$-group of adjoint type.
\end{cor}
For the reader's convenience, we now give a brief overview of the
proof of Theorem~\ref{1.2}. Since our goal is to show that
$L\cong{\mathcal M}(m,n)$, we need to produce a subalgebra $L_{(0)}$
of codimension $5$ in $L$. As in the previous two papers of the
series, local analysis is vital here. All possible types of
$2$-sections in simple Lie algebras are described in
\cite[Sect.~4]{PS4}. The list of $2$-sections is long, but a
thorough investigation shows that most of them cannot occur in our
situation. We prove in Section~5 that if $T$ is a nonstandard torus
of maximal dimension in $L_p$ and $\alpha,\beta\in\Gamma(L,T)$ are
${\mathbb F}_p$-independent, then $\rad\,L(\alpha,\beta)\subset T$
and either $L[\alpha,\beta]\cong{\mathcal M}(1,1)$ of
$L[\alpha,\beta]^{(1)}\cong H(2;(2,1))^{(2)}$; see
Theorem~\ref{sum}. In particular, this implies that all root spaces
of $L$ with respect to $T$ are $5$-dimensional. This intermediate
result is crucial for the rest of the paper. In order to prove it we
have to refine our earlier description of $2$-sections with core of
type $H(2;(2,1))^{(2)}$; see Theorem~\ref{r1TR2}(5). The proof of
Theorem~\ref{r1TR2}(5) relies heavily on a classification of certain
toral derivations of $H(2;(2,1))$. The latter is obtained in
Section~2, the longest section of the paper.

In Section~6, we show the restricted Melikian algebra ${\mathcal
M}(1,1)$ has no nontrivial central extensions and describe the
$p$-characters of irreducible ${\mathcal M}(1,1)$-modules of
dimension $\le 125$. This gives us important new information on the
$p$-mapping of $L_p$; see Section~7. To proceed further we need a
sufficiently generic nonstandard torus of maximal dimension in
$L_p$. We show in Section~9 that there is a nonstandard torus $T$ of
maximal dimension in $L_p$ for which $H^3=[{\mathfrak
c}_L(T),[{\mathfrak c}_L(T),{\mathfrak c}_L(T)]]$ contains no
nonzero toral elements. We then use the new information on the
$p$-mapping of $L_p$ to construct for every $\alpha\in\Gamma(L,T)$ a
subalgebra $Q(\alpha)\subset L(\alpha)$ such that $L(\alpha)=H\oplus
Q(\alpha)$, and set
$L_{(0)}\,:=\,\sum_{\alpha\in\Gamma(L,T)}\,Q(\alpha)$. By
construction, $L_{(0)}$ is a subspace of $L$. In order to show that
it is a subalgebra, we need to check that $[Q(\alpha),
Q(\beta)]\subset\,Q(\alpha)\oplus\sum_{i\in{\mathbb
F}_p}\,Q(\beta+i\alpha)$ for all ${\mathbb F}_p$-independent
$\alpha,\beta\in\Gamma(L,T)$. This is carried out in Section~10. The
rest of the proof is routine.

All Lie algebras in this paper are assumed to be finite-dimensional.
We adopt the notation introduced in \cite{PS1}, \cite{PS2},
\cite{PS3}, \cite{PS4} with the following two exceptions: the
divided power algebra $A(m;\un{n})$ is denoted here by ${\mathcal
O}(m;\un{n})$, and the Melikian algebra ${\mathfrak g}(m,n)$ by
${\mathcal M}(m,n)$. Given a Lie subalgebra $M$ of $L$, we  write
$M_p$ for the $p$-envelope of $M$ in $L_p$.

\medskip

\noindent {\bf Acknowledgement.} Part of this work was done during
our stay at the Max Planck Institut f\"ur Mathematik (Bonn) in the
spring of 2007. We would like to thank the Institute for warm
hospitality and support. We are thankful to the referee for very
careful reading and helpful comments.

\section{\bf Toral elements and one-sections in  $H(2;(2,1))$}
The Lie algebra $H(2;(2,1))$ will appear quite frequently in what
follows, and to deal with it we need some refinements of
\cite[(10.1.1)]{BW}, \cite[(VI.4)]{St91} and \cite[Prop.~2.1]{PS4}.
Set $S:=H(2;(2,1))^{(2)}$, $G:=H(2;(2,1))$, and denote by $S_{(i)}$
(resp., $G_{(i)}$) the $i$th component of the standard filtration of
$S$ (resp., $G$). Recall that $S_p=H(2;(2,1))^{(2)}\oplus FD_1^p$;
see \cite[Thm.~7.2.2(5)]{St04}, for instance. By
\cite[Prop.~2.1.8(viii)]{BW}, $G=V\oplus S$ where
$$V=FD_H(x_1^{(p^2)})\oplus FD_H(x_2^{(p)})\oplus
D_H(x_1^{(p^2-1)}x_2^{(p-1)}).$$ Note that $V$ is a Lie subalgebra
of $G$, and in $\Der S$ we have $V^{[p]}=V^3=0$. We denote by
$\mathcal G$ the $p$-envelope of $G$ in $\Der S$. As $V^{[p]}=0$, it
follows from Jacobson's formula \cite[p.~17]{St04} that ${\mathcal
G}=V\oplus S_p$. We remind the reader that $G$ is a Lie subalgebra
of the Hamiltonian algebra
$H(2)\,=\,\mathrm{span}\,\{D_H(f)\,|\,\,f\in\mathcal{O}(2)\}$ and
$$[D_H(f),D_H(g)]\,=\,D_H\big(D_1(f)D_2(g)-D_2(f)D_1(g)\big)\qquad\quad\
(\forall\,f,g\in\mathcal{O}(2)).
$$ Furthermore, $D_H(f)=D_H(g)$ if and only if $f-g\in F$.
\begin{lemm}\label{conj}
Every toral element $t$ of $S_p$ contained in $S\setminus S_{(0)}$
is conjugate under the automorphism group of $S$ to an element
$$t_\mu=D_H\big(x_1+\mu x_1^{(p)}+(x_1+\mu x_1^{(p)})rx_2^{(p-1)}\big),\qquad
\ r=1+\mu x_1^{(p-1)},$$ where $\mu\in \{0,1\}$. Each such element
is toral.
\end{lemm}
\begin{proof} (a) Write $t=aD_1+bD_2+w$ with $a,b\in F$ and $w\in S_{(0)}$.
By our assumption, $t$ is a toral element of $S_p$; that is,
$t^{[p]}=t$. Since $(aD_1+bD_2)^{[p]}=a^pD_1^p$ and $w^{[p]}\in
S_{(0)}$, Jacobson's formula yields $a=0$. Since $t\not\in S_{(0)}$,
it must be that $b\ne 0$. There exists a special automorphism
$\sigma$ of the divided power algebra $\mathcal{O}(2;(2,1))$ such
that $\sigma(x_1)=b^{-1}x_1$ and $\sigma(x_2)=bx_2$. It induces an
automorphism $\Phi_\sigma$ of the Lie algebra $S$ via
$\Phi_\sigma(E)=\sigma\circ E\circ\sigma^{-1}$ for all $E\in S$; see
\cite[Thm.~7.3.6]{St04}. After adjusting $t$ by $\Phi_\sigma$ it can
be assumed that $b=1$. The description of ${\rm Aut}\,S$ given in
\cite[Thms~7.3.5 \& 7.3.2]{St04} implies that for any $\lambda\in F$
and any pair of nonnegative integers $(m,n)$ such that either
$(m,n)=(p^2,0)$ or $m+n\ge 3$, $m<p^2$, $n<p$ and $(m,n)\neq (p,1)$
 there exists $\sigma_{m,n,\lambda}\in {\rm
Aut}\,S$ with
$$\sigma_{m,n,\lambda}(u)\equiv u+\lambda\big[D_H(x_1^{(m)}x_2^{(n)}), u\big]\quad \
\big({\rm mod}\,S_{(i+m+n-1)}\big)\qquad\quad(\forall\, u\in
S_{(i)}).$$  Because
$$[D_2,D_H(x_1^{(m)}x_2^{(n)})]=D_H(x_1^{(m)}x_2^{(n-1)})\ \, \qquad
(1\le n\le p-1),$$ it is not hard to see that there is $g\in{\rm
Aut}\,S$ such that $g(t)=D_H(x_1+\mu x_1^{(p)})+D_H(fx_2^{(p-1)})$
for some $\mu\in F$ and $f=\sum_{i=1}^{p^2-1}\,\lambda_ix_1^{(i)}$
with $\lambda_i\in F$. If $\mu\ne 0$, then there exists $\alpha\in
F$ with $\alpha^{p-1}\mu=1$ and a special automorphism $\sigma'$ of
the divided power algebra $\mathcal{O}(2;(2,1))$ for which
$\sigma'(x_1)=\alpha x_1$ and $\sigma'(x_2)=x_2$. It gives rise to
an automorphism $\Phi_{\sigma'}$ of the Lie algebra $S$ such that
$\Phi_{\sigma'}(D_H(x_1^{(r)}x_2^{(s)}))=\alpha^{r-1}D_H(x_1^{(r)}x_2^{(s)})$
for all admissible $r$ and $s$; see \cite[Thm.~7.3.6]{St04}.
Adjusting $t$ by $\Phi_{\sigma'}$ we may assume without loss that
$\mu\in\{0,1\}$.

Put $r=D_1(x_1+\mu x_1^{(p)})=1+\mu x_1^{(p-1)}$, $f':=D_1(f)$, and
assume from now on that $t=D_H(x_1+\mu
x_1^{(p)})+D_H(fx_2^{(p-1)})$.

\medskip

\noindent (b) As $\big(\ad D_{H}(fx_2^{(p-1)})\big)\big(\ad
D_H(x_1+\mu x_1^{(p)})\big)^k(D_{H}(fx_2^{(p-1)}))=0$ for $0\le k\le
p-3,\,$ $D_H(x_1+\mu x_1^{(p)})^{[p]}=D_{H}(fx_2^{(p-1)})^{[p]}=0$,
and $$[D_H(x_1+\mu x_1^{(p)}),
D_H(r^ifx_2^{(j)})]\,=\,D_H(r^{i+1}fx_2^{(j-1)})\qquad\quad\,(1\le
i,j\le p-1),$$ Jacobson's formula yields
\begin{eqnarray*}
t^{[p]}&=&\big(\ad D_H(x_1+\mu x_1^{(p)})\big)^{p-1}(D_{H}(fx_2^{(p-1)}))\\
&+& \frac{1}{2}\big[D_{H}(fx_2^{(p-1)}),\big(\ad
D_H(x_1+\mu x_1^{(p)})\big)^{p-2}(D_{H}(fx_2^{(p-1)}))\big]\\
&=& D_H(r^{p-1}f)+
\frac{1}{2}\big[D_{H}(fx_2^{(p-1)}),D_{H}(r^{p-2} fx_2)]\\
&=&D_H(r^{p-1}f)+\frac{1}{2}D_H\big(f' r^{p-2} fx_2^{(p-1)}\big)
-\frac{1}{2} {p-1\choose 1} D_H\big(f
D_1(r^{p-2} f)x_2^{(p-1)}\big)\\
&=&D_H(r^{p-1} f)+D_H(f f' r^{p-2}x_2^{(p-1)})-\mu D_H(r^{p-3}
x_1^{(p-2)}f^2x_2^{(p-1)}).
\end{eqnarray*}
As $r^{p-1}=r^{-1}$, the RHS equals $t$ if and only if $f=(x_1+\mu
x_1^{(p)})r$, as claimed. \end{proof}

Denote by ${\mathcal O}(2;(2,1))_{(k)}[x_1]$ the subalgebra of
${\mathcal O}(2;(2,1))$ spanned by all $x_1^{(i)}$ with $k\le i<
p^2$ and let ${\mathcal O}(2;(2,1))[x_1]:={\mathcal
O}(2;(2,1))_{(0)}[x_1]$. For $u\in {\mathcal O}(2;(2,1))[x_1]$ put
$u':=D_1(u)$ and set $\widetilde{r}:=x_1+\mu x_1^{(p)}$, so that
$t_\mu=D_H(\widetilde{r}+r\widetilde{r}x_2^{(p-1)})$. Note that
$\widetilde{r}'=r$.

\begin{lemm}\label{centraliser} Let $t_\mu$ be as in
Lemma~\ref{conj} and put $C_\mu:={\mathfrak c}_{\mathcal G}(t_\mu)$.
\begin{itemize}
\item[(i)] The Lie algebra $C_\mu$ has an abelian ideal $C_\mu'$ of codimension $2$
spanned by  all $D_H(u+u'\widetilde{r}x_2^{(p-1)})$ with $u\in
{\mathcal O}(2;(2,1))[x_1]$ and by $D_H(x_1^{(p^2)})$. Furthermore,
$C_\mu=Fn_\mu\oplus Fh_\mu\oplus C_\mu'$, where $n_\mu=D_1^p+\mu
D_H(x_2^{(p)})$ and $h_\mu=D_H(r^{-1}x_2-x_2^{(p)})$.
\smallskip

\item[(ii)] Given $a\in F$ and $v\in{\mathcal O}(2;(2,1))[x_1]$ put
$$\varphi_a(v)\,:= \,\sum_{i=0}^{p-1}\,a^i
D_H(r^{-i}vx_2^{(i)})+a^{p-1}D_H(\widetilde{r}v'x_2^{(p-1)}).$$ Then
for every $k\in {\mathbb F}_p^\times$ the $k$-eigenspace of $\ad
t_\mu$ has dimension $p^2$ and is spanned by all $\varphi_k(u)$ with
$u\in {\mathcal O}(2;(2,1))[x_1].$

\item[(iii)] In $\mathcal G$ we have $h_\mu^{[p]}=-\mu h_\mu-n_\mu$ and $n_\mu^{[p]}=0$.

\smallskip

\item[(iv)] If $\mu = 0$, then $C_\mu$ is nilpotent and $Ft_\mu$ is a maximal
torus in $\mathcal G$.
\end{itemize}
\end{lemm}
\begin{proof} (i) It is straightforward to see that $C_\mu'$ is abelian and
$t_\mu\in C_\mu'$. Also, $$\big[D_1^p,t_\mu\big]=\mu
D_H(rx_2^{(p-1)})=-\mu\big[D_H(x_2^{(p)}),t_\mu\big],$$ implying
$n_\mu\in C_\mu$. For all $u\in\langle x_1^{(i)}\,|\,\,0\le i\le
p^2\rangle$ we have
\begin{eqnarray*}
\big[D_H(r^{-1}x_2),D_H(u+u'\widetilde{r}x_2^{(p-1)})\big]&=&-D_H(r^{-1}u')
+D_H\big(r^{-2}(r'u'\widetilde{r}-(u'\widetilde{r})'r)x_2^{(p-1)}\big)\\
\big[D_H(x_2^{(p)}),D_H(u+u'\widetilde{r}x_2^{(p-1)})\big]&=&-D_H(u'x_2^{(p-1)}).
\end{eqnarray*}
As a consequence,
\begin{equation}\label{2.1}\big[h_\mu,D_H(u+u'\widetilde{r}x_2^{(p-1)})\big]=-D_H\big(r^{-1}u'+
(r^{-1}u')'\widetilde{r}x_2^{(p-1)}\big)\end{equation} for all
$u\in{\mathcal O}(2;(2,1))[x_1]$. Putting $u=\widetilde{r}$ gives
$h_\mu\in C_\mu$.

\smallskip

\noindent (ii) We claim that for all $u\in\langle
x_1^{(i)}\,|\,\,1\le i\le p^2\rangle$ and all $k\in{\mathbb
F}_p^\times$ the following relations hold:
\begin{eqnarray}
\big[D_H(u+u'\widetilde{r}x_2^{(p-1)}),
\varphi_k(v)\big]&=&k\varphi_k(r^{-1}u'v)\label{2.2}\\
\big[D_H(r^{-1}x_2-x_2^{(p)}),
\varphi_k(v)\big]&=&[h_\mu,\varphi_k(v)]\,=\,-\varphi_k(r^{-1}v')\label{2.3}.
\end{eqnarray}
Indeed, since $k^{p-1}=1$, $r^p=1$, and $x_2^{(p-2)}\!\cdot
x_2^{(k)}=0$ for $2\le k\le p-1$, the LHS of (\ref{2.2}) equals
$D_H(w)$, where
\begin{eqnarray*}
w&=&D_1(u+u'\widetilde{r}x_2^{(p-1)})\cdot D_2(\varphi_k(v))
-D_2(u+u'\widetilde{r}x_2^{(p-1)})\cdot D_1(\varphi_k(v))\\
&=&\big(u'+u''\widetilde{r}x_2^{(p-1)}+u'rx_2^{(p-1)}\big)
\cdot\Big({\textstyle\sum}_{i=1}^{p-1}\,k^ir^{-i}vx_2^{(i-1)}+\widetilde{r}v'x_2^{(p-2)}\Big)\\
&-&u'\widetilde{r}x_2^{(p-2)}\cdot\Big(v'+k\big(r^{-1}v\big)'x_2\Big)\\
&=&u'\Big({\textstyle\sum}_{i=1}^{p-1}\,k^ir^{-i}vx_2^{(i-1)}\Big)+u'\widetilde{r}v'x_2^{(p-2)}\\
&+&ku''\widetilde{r}r^{-1}vx_2^{(p-1)}+ku'vx_2^{(p-1)}
-u'\widetilde{r}v'x_2^{(p-2)}+ku'\widetilde{r}\big(r^{-1}v\big)'x_2^{(p-1)}\\
&=&k{\textstyle\sum}_{i=0}^{p-2}\,k^ir^{-i}(r^{-1}u'v)x_2^{(i)}+
k\Big(u''\widetilde{r}r^{-1}v+u'v
+u'\widetilde{r}\big(r^{-1}v\big)'\Big)x_2^{(p-1)}\\
&=&k{\textstyle\sum}_{i=0}^{p-1}\,k^ir^{-i}(r^{-1}u'v)x_2^{(i)}
+k\widetilde{r}(r^{-1}u'v)'x_2^{(p-1)}.
\end{eqnarray*}
But then $D_H(w)=k\varphi_k(r^{-1}u'v)$ and (\ref{2.2}) follows.
Since
$$
\big(-r^{-1}v'\big)'\,=\,r^{-2}r'v'-r^{-1}v'',
$$
the LHS of (\ref{2.3}) equals $D_H(y)$, where
\begin{eqnarray*}
y&=&(r^{-1})'x_2\cdot\Big({\textstyle\sum}_{i=1}^{p-1}\,k^ir^{-i}vx_2^{(i-1)}+\widetilde{r}v'x_2^{(p-2)}\Big)\\
&-&r^{-1}\cdot\Big({\textstyle\sum}_{i=1}^{p-1}\,k^i
i(r^{-1})^{i-1}(-r^{-2}r')vx_2^{(i)}+\textstyle{\sum}_{i=0}^{p-1}\,k^ir^{-i}v'x_2^{(i)}\Big)\\
&-&r^{-1}\cdot(\widetilde{r}v')'x_2^{(p-1)}+x_2^{(p-1)}v'\\
&=&-r^{-2}r'\cdot\Big({\textstyle\sum}_{i=1}^{p-1}\,k^i
ir^{-i}vx_2^{(i)}-\widetilde{r}v'x_2^{(p-1)}\Big)\\
&+&r^{-2}r'\cdot\Big({\textstyle\sum}_{i=1}^{p-1}\,k^i
ir^{-i}vx_2^{(i)}\Big)+\textstyle{\sum}_{i=0}^{p-1}\,k^i
r^{-i}(-r^{-1}v')x_2^{(i)}\\
&-&r^{-1}\cdot(rv'+\widetilde{r}v'')x_2^{(p-1)}+x_2^{(p-1)}v'\\
&=&r^{-2}r'\widetilde{r}v'x_2^{(p-1)}+{\textstyle
\sum}_{i=0}^{p-1}\,k^i r^{-i}(-r^{-1}v')x_2^{(i)}-
r^{-1}\widetilde{r}v''x_2^{(p-1)}\\
&=&{\textstyle\sum}_{i=0}^{p-1}\,k^i
r^{-i}(-r^{-1}v')x_2^{(i)}+\big(\widetilde{r}r^{-2}r'v'-\widetilde{r}r^{-1}v''\Big)x_2^{(p-1)}\\
&=&{\textstyle \sum}_{i=0}^{p-1}\,k^i
r^{-i}(-r^{-1}v')x_2^{(i)}+\widetilde{r}\big(-r^{-1}v'\big)'x_2^{(p-1)}.
\end{eqnarray*}
This shows that $D_H(y)=D_H(-r^{-1}v')$, proving (\ref{2.3}).

Setting $u=\widetilde{r}$ in (\ref{2.2}) now gives
$[t_\mu,\varphi_k(v)]=k\varphi_k(v)$. Since $\varphi_k(v)\ne 0$ for
all nonzero $v\in{\mathcal O}(2;(2,1))[x_1]$, comparing dimensions
yields that $C_\mu$ is spanned by $h_\mu$, $n_\mu$ and $C_\mu'$ and
that for every $k\in{\mathbb F}_p^\times$ the $k$-eigenspace of $\ad
t_\mu$ has dimension $p^2$ and is spanned by all $\varphi_k(v)$ with
$v\in{\mathcal O}(2;(2,1))[x_1]$.

\smallskip

\noindent (iii) Clearly,
$n_\mu^{[p]}=D_1^{p^2}-\mu^p(x_2^{(p-1)}D_1)^p=0$. Next observe that
$$[h_\mu,n_\mu]\,=\,[D_H(r^{-1}x_2-x_2^{(p)}), D_1^p+\mu
D_H(x_2^{(p)})]\,=\,\mu D_H\big((r^{-1})'x_2\cdot
x_2^{(p-1)}\big)=\,0.$$ We claim that $h_\mu^{[p]}+\mu h_\mu
+n_\mu=0$. If $\mu=0$, then $h_\mu=D_1-x_2^{(p-1)}D_1$ and
$n_\mu=D_1^p$; hence, our claim is true in this case. Assume now
that $\mu\ne 0$ and set $q:=h_\mu+\mu^{-1}n_\mu$. Since our remarks
at the beginning of this part imply that
$q^{[p]}=(h_\mu+\mu^{-1}n_\mu)^{[p]}=h_\mu^{[p]}$, we are reduced to
showing that $q^{[p]}+\mu q=0$. As $[D_H(x_1^{(p-1)}x_2),\big(\ad
D_1\big)^i(D_H(x_1^{(p-1)}x_2))]=\,0$ for all $i\le p-2$, we see
that
\begin{eqnarray*}
q^{[p]}&=&\big(\mu^{-1}D_1^p-D_1-\mu
D_H(x_1^{(p-1)}x_2)\big)^{[p]}\,=\,\big(-D_1-\mu
D_H(x_1^{(p-1)}x_2)\big)^{[p]}\\
&=& -D_1^p-\big(\ad D_1\big)^{p-1}(\mu D_H(x_1^{(p-1)}x_2))\\
&-&\frac{1}{2}\Big[\mu D_H(x_1^{(p-1)}x_2),\big(\ad
D_1\big)^{p-2}(\mu D_H(x_1^{(p-1)}x_2))\Big]\\
&=&-D_1^p+\mu D_1+\mu^2 D_H(x_1^{(p-1)}x_2)\,=\,-\mu q,
\end{eqnarray*} and our claim follows.
\smallskip

\noindent (iv) Now suppose $\mu=0$. Then
$t_\mu=D_H(x_1(1+x_2^{(p-1)}))$,
$h_\mu=D_H(x_2-x_2^{(p)})=(x_2^{(p-1)}-1)D_1$ and $n_\mu=D_1^p$. Set
$C:=C_0$ and $ C_{(0)}:=C\cap G_{(0)}$.  By
Lemma~\ref{centraliser}(i), which we have already proved, $C$ is
spanned by $D_1^p,\,$ $(x_2^{(p-1)}-1)D_1$ and by all
$D_H(x_1^{(k+1)}+x_1^{(k)}\widetilde{r}x_2^{(p-1)})$ with $0\le k\le
p^2-1$. As a consequence, $C\,=\,FD_1^p\oplus
F(x_2^{(p-1)}-1)D_1\oplus Ft_\mu\oplus C_{(0)}$. As $G_{(0)}$ is a
restricted subalgebra of $\mathcal G$, so is $C_{(0)}$. From this it
is immediate that $C_{(0)}$ is a $p$-nilpotent subalgebra of
$\mathcal G$. Note that $C\cap S=Ft_\mu\oplus C_{(0)}$ is an ideal
of $C$. Since $\big((x_2^{(p-1)}-1)D_1\big)^{[p]}=-D_1^p$ and
$\big(D_1^p\big)^{[p]}=0$ (as derivations of $S$), Jacobson's
formula implies that $C^{[p]}\subset FD_1^p\oplus Ft_\mu\oplus
C_{(0)}$ and $C^{[p]^2}\subset Ft_\mu\oplus C_{(0)}$. Since
$C_{(0)}$ is $p$-nilpotent and $[t_\mu,C]=0$, it follows that
$C^{[p]^e} = Ft_\mu$ for all $e\gg 0$. Hence $C$ is a restricted
nilpotent subalgebra of $\mathcal G$ and $Ft_\mu$ is the unique
maximal torus of $C$.
\end{proof}

If $u$  belongs to the linear span of all $x_1^{(i)}$ with $2\le
i\le p^2$, then $r^{-1}u'\in {\mathcal O}(2;(2,1))_{(1)}$, forcing
$(r^{-1}u')^p=0$. For $k\in \mathbb{F}_p^\times$ we write $S_k$ for
the the $k$-eigenspace of $\ad t_\mu$. In view of (\ref{2.2}) we
have that $(\ad D_H(u+\widetilde{r}u'x_2^{(p-1)})\big)^p(S_k)=(0)$
for all $k\in\mathbb{F}_p^\times$. Since
$$
\big(\ad D_H(u+\widetilde{r}u'x_2^{(p-1)})\big)^p(C_\mu)\subset
\big(\ad D_H(u+\widetilde{r}u'x_2^{(p-1)})\big)^{p-1}(C_\mu')\subset
(C_\mu')^{(1)}\,=\,(0)$$ by Lemma~\ref{centraliser}(i), it follows
that $\big(\ad D_H(u+\widetilde{r}u'x_2^{(p-1)})\big)^p=0$.
Therefore, for all $u$ as above and $c\in F$ the exponential
$\exp\big(c\,\ad D_H(u+u'\widetilde{r}x_2^{(p-1)})\big)$ is
well-defined as a linear operator on $S$.

\begin{lemm}\label{stab}
Suppose $\mu\ne 0$ and let $Z(t_\mu)$ denote the stabilizer of
$t_\mu$ in ${\rm Aut}\,S$.
\begin{itemize}
\item[(i)] $\exp\big(c\,\ad D_H(x_1^{(m)}+x_1^{(m-1)}\widetilde{r}x_2^{(p-1)})\big)\in Z(t_\mu)$
for all $\,3\le m\le p^2$.

\smallskip

\item[(ii)] For every $h\in G\cap C_\mu$ with $h\not\in C_\mu'$ there
exist $z\in Z(t_\mu)$ and $a\in F^\times$ such that $z(h)=ah_\mu+b
t_\mu+s D_H(x_1^{(p^2)})$ for some $b,s\in F$.

\smallskip

\item[(iii)] If $h\in(G\cap C_\mu)\setminus C_\mu'$, then
for every $k\in{\mathbb F}_p^\times$ there is $v_k\in 1+{\mathcal
O}(2;(2,1))_{(1)}[x_1]$ such that $\varphi_k(v_k)$ is an eigenvector
for $\ad h$ and $\varphi_k(v_k)^{[p]}$ is a nonzero $p$-semisimple
element of $\mathcal G$.

\smallskip

\item[(iv)] For every $h\in(G\cap C_\mu)\setminus C_\mu'$ there
exists a nonzero $x\in \mathfrak{c}_S(t_\mu)$ such that $\ad x$ is
not nilpotent and $[h,x]=\lambda x$ for some nonzero $\lambda\in F$.
\end{itemize}
\end{lemm}
\begin{proof}
(a) For $1\le m\le p^2$ set ${\mathcal D}_m:=\ad
D_H(x_1^{(m)}+x_1^{(m-1)}\widetilde{r}x_2^{(p-1)})$. As $(\ad
{\mathcal D}_m)^p=0$ for $m\ge 3$, in order to prove (i) it suffices
to show  that
\begin{equation}\label{cocycle}
\sum_{i=1}^{p-1}\,\frac{1}{i!(p-i)!} \big[{\mathcal
D}_{m}^{i}(y_1),{\mathcal D}_{m}^{p-i}(y_2)\big]\,=\,0\qquad\ \
\big(\forall\,y_1,y_2\in S,\ \forall\,m\ge 3\big).
\end{equation}
It follows from Lemma~\ref{centraliser}(i) that ${\mathcal
D}_m^2(C_\mu)\subset (C_\mu')^{(1)}=(0)$. Therefore, we just need to
show that (\ref{cocycle}) holds for all $y_1=\varphi_k(v_1)$ and
$y_2=\varphi_l(v_2)$, where $k,l\in{\mathbb F}_p^\times$ and
$v_1,v_2\in{\mathcal O}(2;(2,1))[x_1]$.

For $3\le m\le p$ we have $\big(r^{-1}x_1^{(m-1)}\big)^{(p+1)/2}=0$,
since ${\mathcal O}(2;(1,1))$ is a subalgebra of ${\mathcal
O}(2;(2,1))$ and $\frac{(m-1)(p+1)}{2}>p$. In light of (\ref{2.2})
this gives $\big(\ad {\mathcal
D}_m\big)^{(p+1)/2}\big(\varphi_i(v)\big)=0$ for all $i\in{\mathbb
F}_p^\times$ and $v\in{\mathcal O}(2;(2,1))[x_1]$. Hence
(\ref{cocycle}) holds for $m\le p$.

If $m\ge p+2$, then (\ref{2.2}) yields that ${\mathcal
D}_m^{i}(\varphi_k(v_1))=\varphi_k(w_1)$ and ${\mathcal
D}_m^{p-i}(\varphi_l(v_2))=\varphi_l(w_2)$ for some $w_1\in
{\mathcal O}(2;(2,1))_{(i(p+1))}[x_1]$ and $w_2\in {\mathcal
O}(2;(2,1))_{((p-i)(p+1)}[x_1]$. As
$[\varphi_k(w_1),\varphi_l(w_2)]=0$ in this case, we deduce that
(\ref{cocycle}) holds for $m\ge p+2$. As ${\mathcal
O}(2;(2,1))_{(p^2)}[x_1]=0$, this argument also shows that
(\ref{cocycle}) holds if $m=p+1$ and either $v_1$ or $v_2$ belongs
to ${\mathcal O}(2;(2,1))_{(1)}[x_1]$.

Thus, in order to prove (i) it suffices to show that (\ref{cocycle})
holds for $m=p+1$ and $v_1=v_2=1$. Suppose the contrary and set
$$Y\,:=\,\,\sum_{i=1}^{p-1}\,\frac{1}{i!(p-i)!} \big[{\mathcal
D}_{p+1}^{i}(\varphi_k(1)),{\mathcal
D}_{p+1}^{p-i}(\varphi_l(1))\big].$$ Arguing as in the preceding
paragraph we now observe that $Y$ is a nonzero multiple of either
$\varphi_{k+l}(x_1^{(p^2-1)})$ (if $k+l\ne 0$) or
$D_H(x_1^{(p^2-1)}(1-x_2^{(p-1)}))$ (if $k+l=0$). In any event,
$(\ad n_\mu)^{p-1}(Y)\ne 0$.

Set $N_\mu:=\ad n_\mu$. We know from the proof of
Lemma~\ref{centraliser} that $[N_\mu, {\mathcal D}_{p+1}]={\mathcal
D}_1,\,$ $[{\mathcal D}_{1},{\mathcal D}_{p+1}]=0$ and
$N_\mu(\varphi_i(1))=0$ for all $i\in{\mathbb F}_p^\times$. From
this it follows that
\begin{eqnarray*}
N_\mu^{p-1}(Y)&=&\sum_{i=1}^{p-1}\,\frac{1}{i!(p-i)!}\Big(\sum_{j=0}^{p-1}
(-1)^j\Big[N_\mu^{j}\big({\mathcal D}_{p+1}^i(\varphi_k(1)\big),
N_\mu^{p-1-j}\big({\mathcal
D}_{p+1}^{p-i}(\varphi_l(1)\big)\Big]\Big)\\
&=&\sum_{i=1}^{p-1}\,(-1)^i\big[{\mathcal D}_{1}^i(\varphi_k(1)),
\big({\mathcal D}_{1}^{p-i-1}{\mathcal
D}_{p+1}\big)(\varphi_l(1))\big]\\
&+&\sum_{i=1}^{p-1}\,(-1)^{i-1}\big[\big({\mathcal
D}_{1}^{i-1}{\mathcal D}_{p+1}\big)(\varphi_k(1)), {\mathcal
D}_{1}^{p-i}(\varphi_l(1))\big]\\
&=&{\mathcal D}_1^{p-1}\big(\big[\varphi_k(1),{\mathcal
D}_{p+1}(\varphi_l(1))\big]\big)-\big[\varphi_k(1),\big({\mathcal
D}_1^{p-1}{\mathcal
D}_{p+1}\big)(\varphi_l(1))\big]\\
&+&{\mathcal D}_1^{p-1}\big(\big[{\mathcal
D}_{p+1}(\varphi_k(1)),(\varphi_l(1)\big]\big)-
\big[\big({\mathcal D}_1^{p-1}{\mathcal D}_{p+1}\big)(\varphi_k(1)),\varphi_l(1)\big]\\
&=&\big({\mathcal D}_1^{p-1}{\mathcal
D}_{p+1}\big)\big([\varphi_k(1),\varphi_l(1)]\big)
-l[\varphi_k(1),\varphi_l(x_1^{(p)})]-
k[\varphi_k(x_1^{(p)}),\varphi_l(1)]
\end{eqnarray*}
(we used (\ref{2.2}) and the equalities $r^p=1$, $k^p=k$ and
$l^p=l$). On the other hand, comparing components of $x_2$-degree
$0$ and $1$ one observes that
\[
\big[\varphi_k(u),\varphi_l(v)\big]=\left\{
\begin{array}{ll}
\varphi_{k+l}((lu'v-kuv')r^{-1})&\, \,\mbox{ if $k+l\ne 0,$}\\
kD_H((u'v-uv')+(u'v-uv')'\widetilde{r}x_2^{p-1})&\, \,\mbox{ if
$k+l=0$}
\end{array}
\right.
\]
for all $u,v\in{\mathcal O}(2;(2,1))[x_1]$. But then
$l[\varphi_k(1),\varphi_l(x_1^{(p)})]+
k[\varphi_k(x_1^{(p)}),\varphi_l(1)]=0$ and
$[\varphi_k(1),\varphi_l(1)]=0$, forcing $N_\mu^{p-1}(Y)=0$, a
contradiction. Statement~(i) follows.

\smallskip

\noindent (b) Observe that $C_\mu\cap G\,=\,C_\mu'\oplus Fh_\mu$.
 If $h\in C_\mu\cap G$ and $h\not\in C_\mu'$, then
 Lemma~\ref{centraliser}(i) implies that there are $a\in F^\times$,
 $b,s\in F$ such that
 $h=ah_\mu+bt_\mu+sD_H(x_1^{(p^2)})+\sum_{i=2}^{p^2-1}a_iD_H(x_1^{(i)}+x_1^{(i-1)}
 \widetilde{r}x_2^{(p-1)})$
for some $a_i\in F$. Since $C_\mu'$ is abelian, $r$ is invertible,
and
$$\big(\exp\,a_i{\mathcal D}_m\big)(h_\mu)=h_\mu +a_i
D_H(r^{-1}(x_1^{(m-1)}+x_1^{(m-2)}\widetilde{r}x_2^{(p-1)}))
\qquad\quad \ (3\le m\le p^2)$$ by (\ref{2.1}), we can clear the
$a_i$'s by applying suitable automorphisms from $Z(t_\mu)$. This
proves statement~(ii).

In dealing with (iii) we may assume that $h=h_\mu+sD_H(x_1^{(p^2)})$
where $s\in F$. In view of (\ref{2.3}) we need to find
$v_k=1+b_1x_1^{(1)}+b_2x_1^{(2)}+\cdots+b_{p^2-1}x_1^{(p^2-1)}$ and
$\eta_k\in F$ satisfying the condition
\begin{eqnarray*}
\eta_k\varphi_k(v_k)&=&[h_\mu+sD_H(x_1^{(p^2)}),\varphi_k(v_k)]\\
&=&-\varphi_k(r^{-1}v_k')+sD_H\Big(x_1^{(p^2-1)}\cdot\big({\textstyle
\sum}_{i=0}^{p-1}\,k^ir^{-i}v_kx_2^{(i-1)}+k^{p-1}\widetilde{r}v_k'x_2^{(p-2)}\big)\Big)\\
&=&-\varphi_k(r^{-1}v_k')+sk\varphi_k(x_1^{(p^2-1)}v_k).
\end{eqnarray*}
This holds if and only if
$$-b_1-b_2x_1-\cdots-b_{p^2-1}x_1^{(p^2-2)}+skx_1^{(p^2-1)}=\,\,
\eta_kr\big(1+b_1x_1^{(1)}+\cdots+b_{p^2-1}x_1^{(p^2-1)}\big).$$
Because
\begin{eqnarray*}
r\big(1+\textstyle{\sum_{i=1}^{p^2-1}}\,b_ix_1^{(i)}\big)&=&
\big(1+\textstyle{\sum_{i=1}^{p^2-1}}\,b_ix_1^{(i)}\big)+\mu
\big(x_1^{(p-1)}+
\textstyle{\sum_{i=1}^{p-1}}\,b_{ip}x_1^{(ip+p-1)}\big)\\
&=&\big(1+\textstyle{\sum_{i=1}^{p^2-1}}\,b_ix_1^{(i)}\big)+\mu\,
\textstyle{\sum_{i=0}^{p-1}}\,b_{ip}x_1^{(ip+p-1)}
\end{eqnarray*}
by Lucas' theorem, this leads to the system of equations
\begin{eqnarray*}
b_0&=&1;\\
b_i&=&-\eta_kb_{i-1}\qquad \qquad\qquad\quad\ \ \ 1\le i\le p^2-1,\ \,i\not\in p\,\Z;\\
b_{ip}&=&-\eta_k(b_{ip-1}+\mu b_{i-1})\qquad \quad\ 1\le i\le p-1;\\
\eta_kb_{p^2-1}&=&sk.
\end{eqnarray*}
Arguing recursively, one observes that there is a bijection between
the solutions to this system and the roots of a polynomial of the
form $X^{p^2}+\sum_{i=1}^{p^2-1}\lambda_iX^i-sk$, where
$\lambda_i\in F$. Since $F$ is algebraically closed, it follows that
our eigenvalue problem has at least one solution.

\smallskip

\noindent (c) In view of our discussion in part~(b),
$\varphi_k(v_k)\equiv \,D_H(x_2)+b_1D_H(x_1)\ \,({\rm mod}\,
S_{(0)})$. Since $D_H(x_2)=-D_1$ and $S_{(0)}$ is a restricted
subalgebra of $\mathcal G$, Jacobson's formula shows that
$\varphi_k(v_k)^{[p]}= -D_1^p+w_k$ for some $w_k\in S.$ In
particular, $\varphi_k(v_k)^{[p]}\ne 0$. Note that
$\varphi_k(v_k)^{[p]}\in C_\mu\cap S_p\cap\ker\ad h$. Now, using
(\ref{2.1}) it is easy to observe that $C_\mu'\cap \ker\ad
h=Ft_\mu$, whilst from Lemma~\ref{centraliser} it is immediate that
$C_\mu\cap S_p=F(\mu h_\mu+n_\mu).$ Lemma~\ref{centraliser} also
implies that $\mu h_\mu+n_\mu=-h_\mu^{[p]}$ and
$h_\mu^{[p]^2}=-\mu^{p}h_\mu^{[p]}$.

Let $h_s$ denote the $p$-semisimple part of $h$ in $\mathcal G$, an
element of $C_\mu\cap\ker\ad h\cap S_p$. Since the above discussion
shows that $C_\mu\cap S_p\cap\ker\ad h$ has dimension $\le 2$, in
order to finish the proof of (iii) we need to show that $t_\mu$ and
$h_s$ are linearly independent.

Suppose the contrary. Then $\ad h$ acts nilpotently on $C_\mu'$.
Recall that $h\in h_\mu+C_\mu'$ and $C_\mu'$ is abelian. So $\ad
h_\mu$ acts on $C_\mu'$ nilpotently, too. Since $\mu\ne 0$, our
earlier remarks and Lemma~\ref{centraliser}(iii) now show that ${\rm
ad}(h_\mu^{[p]})=-\mu\,\ad h_\mu-\ad n_\mu$ acts trivially on
$C_\mu'$. Since this violates (\ref{2.1}), we reach a contradiction.
Statement~(iii) follows.

\smallskip

\noindent (d) In proving (iv) we may assume that
$h=h_\mu+sD_H(x_1^{(p^2)})$; see part~(b).  We claim that there
exist $u=x_1+c_1x_1^{(2)}+\cdots+c_{p^2-2}x_1^{(p^2-1)}$ and
$\lambda\in F^\times$ such that
$$\big[h,D_H(u+u'\widetilde{r}x_2^{(p-1)})\big]\,=\,
\lambda D_H(u+u'\widetilde{r}x_2^{(p-1)}).$$
Since $C_\mu'$ is abelian, it follows from (\ref{2.1}) that
$$
\big[h,D_H(u+u'\widetilde{r}x_2^{(p-1)})\big]=\big[h_\mu,
D_H(u+u'\widetilde{r}x_2^{(p-1)})\big]=-D_H\big(r^{-1}u'+
(r^{-1}u')'\widetilde{r}x_2^{(p-1)}\big).$$ Thus, we seek $u$ such
that $r^{-1}u'=a-\lambda u$ for some $a\in F$. Since $r^{-1}=1-\mu
x_1^{(p-1)}$, this entails that $a=1,\,$ $c_1=-\lambda$, and
\begin{equation}\label{c_i}
(1-\mu
x_1^{(p-1)})\Big(1+{\sum}_{i=1}^{p^2-2}\,c_ix_1^{(i)}\Big)\,=\,
1+c_1\Big(x_1+{\sum}_{i=1}^{p^2-2}\,c_ix_1^{(i+1)}\Big).
\end{equation}
Since
$x_1^{(p-1)}\cdot\,\big(1+\sum_{i=1}^{p^2-2}\,c_ix^{(i)}\big)=\big(x_1^{(p-1)}+
\textstyle{\sum_{i=1}^{p-1}}\,c_{ip}x_1^{(ip+p-1)}\big)$ by Lucas'
theorem, we see that $c_{i+1}=c_1c_i$ if $p\nmid (i+2)$. Induction
on $k$ shows that $c_{kp+p-1}=c_1^k(c_1^{p-1}+\mu)^{k+1}$ for $0\le
k\le p-1$. As $c_{p^2-1}=0$, this yields
$c_1^{p-1}(c_1^{p-1}+\mu)^p=0$. As $c_1=-\lambda\ne 0$, we see that
$c_1$ must satisfy the equation $X^{p-1}+\mu=0$. Conversely, any
root of this equation gives rise to a solution of (\ref{c_i}) with
$\lambda=-c_1\ne 0$ (recall that $\mu\ne 0$ by our assumption). The
claim follows.

We now set $x:=D_H(u+u'\widetilde{r}x_2^{(p-1)})$, where $u$ is as
above. Clearly, $x\in S$. Since $r^{-1}u'-1\in
\mathcal{O}_{(1)}(2;(2,1))$, it follows from (\ref{2.2}) that $(\ad
x)^p(\varphi_k(v))=k^p\varphi\big((r^{-1}u')^pv\big)=k\varphi_k(v)$
for all $v\in \mathcal{O}(2;(2,1))[x_1]$ and all $k\in\mathbb{F}_p$.
This implies that $\ad x$ is not nilpotent, completing the proof.
\end{proof}

We now let $\mathfrak t$ be a $2$-dimensional torus in $\mathcal G$.
\begin{lemm}\label{tori}
There exist nonzero $u_1,u_2\in S$ such that ${\mathfrak
t}=F(D_1^p+u_1)\oplus Fu_2$.
\end{lemm}
\begin{proof} Since $V^{[p]}=0$, the restricted Lie
algebra ${\mathcal G}/S_p$ is $p$-nilpotent. As ${\mathfrak t}$ is a
torus, it must be that ${\mathfrak t}\subset S_p$. Then ${\mathfrak
t}\cap S\ne (0)$, for $\dim\,{\mathfrak t}=2$.

Suppose ${\mathfrak t}\subset S$. Since
$S_{(0)}/S_{(1)}\cong\mathfrak{sl}(2)$ and $S_{(-1)}/S_{(0)}$ is a
$2$-dimensional irreducible module over $S_{(0)}/S_{(1)}$, every
nonzero element of ${\mathfrak t}\cap S_{(0)}$ acts invertibly on
$S_{(-1)}/S_{(0)}$. Therefore, ${\mathfrak t}\cap S_{(0)}\ne (0)$
would force ${\mathfrak t}\subset S_{(0)}$, which is false because
$S_{(0)}$ has toral rank $1$ in $S$. On the other hand, if
${\mathfrak t}\cap S_{(0)}=(0)$ (and still ${\mathfrak t}\subset
S$), then $\mathfrak t$ would contain an element of the form $D_1+u$
with $u\in S_{(0)}$. But this would yield $D_1^p\in{\mathfrak
t}+S=S$, as $S_{(0)}$ is a restricted subalgebra of $S_p$.
Therefore, ${\mathfrak t}\not\subset S$. Since $D_1$ is nilpotent
and $S$ has codimension $1$ in $S_p$, our statement follows
immediately.
\end{proof}
\begin{lemm}\label{ham1}
Let ${\mathfrak h}={\mathfrak c}_S(\mathfrak t)$ and let
$\alpha\in\Gamma(S,{\mathfrak t})$.
\begin{itemize}
\item[(1)] If $\alpha$ vanishes on $\mathfrak h$, then $G(\alpha)$ is
solvable.

\smallskip

\item[(2)]
If $\alpha$ does not vanish on $\mathfrak h$, then $G(\alpha)\cong
H(2;\un{1})$.

\smallskip

\item[(3)]
$\dim\,G_\gamma=p+\delta_{\gamma,\,\,0}$ for all
$\gamma\in\Gamma(G,{\mathfrak t})\cup\{0\}$.

\smallskip

\item[(4)]
$\Gamma(S,{\mathfrak t})\cup\{0\}$ is a two-dimensional vector space
over ${\mathbb F}_p$.
\end{itemize}
\end{lemm}
\begin{proof} Note that ${\mathfrak c}_{S_p}({\mathfrak t})={\mathfrak
t}+\mathfrak h$ and $\mathfrak t$ is a standard torus of maximal
dimension in $S_p$. Therefore, the results of \cite[(10.1.1)]{BW}
and \cite[(VI)]{St91} apply to $\mathfrak t$.

If $\alpha$ does not vanish on $\mathfrak h$, then $G(\alpha)\cong
H(2;\un{1})$ by \cite[Prop.~2.1(2)]{PS4}. Suppose $\alpha({\mathfrak
h})=0$. As $\mathfrak t$ is a maximal torus of $S_p$, we have that
$\alpha(L_{i\alpha}^{[p]})=0$ for all $i\in{\mathbb F}_p^\times$.
Then $S(\alpha)$ is nilpotent due to the Engel--Jacobson theorem. As
$G/S$ is nilpotent too, we conclude that $G(\alpha)$ is solvable.

By \cite[(10.1.1(e))]{BW}, there is a $2$-dimensional torus
${\mathfrak t}'$ in $S_p$ such that all roots in
$\Gamma(S,{\mathfrak t}')$ are proper. Then \cite[(VI.2(2))]{St91}
applies showing that all root spaces of $G$ with respect to
${\mathfrak t}'$ are $p$-dimensional and $\dim {\mathfrak
c}_G({\mathfrak t}')=p+1$. By \cite{P89}, all root spaces of $G$
with respect to $\mathfrak t$ must have the same property, and $\dim
{\mathfrak c}_G({\mathfrak t})=p+1$ (see also
\cite[Cor.~2.11]{PS2}). As $\dim\, S=p^3-2$ and $\dim\,S_\gamma\le
p\,$ for all $\gamma\in\Gamma(S,{\mathfrak t})$, we derive that
$|\Gamma(S,{\mathfrak t})|=p^2-1$. As a consequence, the set
$\Gamma(S,{\mathfrak t})\cup\{0\}$ is $2$-dimensional vector space
over ${\mathbb F}_p$. This completes the proof. \end{proof}
\begin{lemm}\label{ham2}
Under the above assumptions on $\mathfrak t$ and $S$ the following
hold:
\begin{itemize}
\item[(1)]
If $TR({\mathfrak h},S)=2$, then all roots in $\Gamma(S,{\mathfrak
t})$ are Hamiltonian improper.

\smallskip

\item[(2)]
If $TR({\mathfrak h},S)=1$, then $\Gamma(S,{\mathfrak t})$ contains
a solvable root.

\smallskip

\item[(3)]
Suppose that $TR({\mathfrak h},S)=1$ and ${\mathfrak h}_p\cap
S_{(0)}$ contains a nonnilpotent element. Then for any solvable
$\alpha\in\Gamma(S,{\mathfrak t})$ the $1$-section $G(\alpha)$ is
nilpotent.
\end{itemize}
\end{lemm}
\begin{proof} Suppose $TR({\mathfrak h}, S)=2$. Then no root in
$\Gamma(S,{\mathfrak t})$ vanishes on $\mathfrak h$,; hence,
 all roots in
$\Gamma(S,{\mathfrak t})$ are Hamiltonian by
Proposition~\ref{ham1}(2). If ${\mathfrak h}\cap S_{(0)}$ contains a
nonnilpotent element, $x$ say, then the image of $x$ in
$S_{(0)}/S_{(1)}\cong \mathfrak{sl}(2)$ acts invertibly on
$S_{(-1)}/S_{(0)}$. As $\mathfrak h$ is nilpotent, this would force
${\mathfrak h}\subset S_{(0)}$, and hence $TR({\mathfrak h},S)=1$, a
contradiction. Consequently, ${\mathfrak t}\cap S_{(0)}=(0)$. By
\cite[(10.1.1(d))]{BW} (see the proof on p.~232-233), every
Hamiltonian root is then improper.

Now suppose  $TR({\mathfrak h},S)=1$. Then the unique maximal torus
of ${\mathfrak h}_p$ is spanned by a toral element, hence it follows
from Lemma~\ref{ham1}(4) that there is a root in
$\Gamma(S,{\mathfrak t})$ which vanishes on $\mathfrak h$. Every
such root is solvable by Proposition~\ref{ham1}(1).

Finally, suppose that $TR({\mathfrak h},S)=1$ and ${\mathfrak
h}_p\cap S_{(0)}$ contains a nonnilpotent element. Since $S_{(0)}$
is a restricted subalgebra of $S_p$, we then have
$S_{(0)}\cap\mathfrak{h}_p\cap\mathfrak{t}\ne (0)$. Since
$\mathfrak{t}\cap S=Fu_2$ for some nonzero $u_2\in S$ (see
Lemma~\ref{tori}), it must be that $u_2\in S_{(0)}$ and
$u_2^{[p]}\in Fu_2$.

If $\alpha\in\Gamma(S,{\mathfrak t})$ is solvable, then
$\alpha({\mathfrak h})=0$ by Lemma~\ref{ham1}(2). As explained in
the proof of Lemma~\ref{ham1} the Lie algebra $S(\alpha)$ is
nilpotent. There exists an element $t\in F^\times u_2$ with
$t^{[p]}=t$ such that $G(\alpha)={\mathfrak c}_G(t)$. Set
$W:=\{v-(\ad t)^{p-1}(v)\,|\,\,\,v\in V\}$. By construction,
$W\subset {\mathfrak c}_G(t)$ and $G=W\oplus S$. Since $V\subset
G_{(1)}$ and $t\in S_{(0)}$, we have the inclusion $W\subset
G_{(1)}$. In particular, all elements of $W$ act nilpotently on
${\mathfrak c}_G(t)$.

Since $S(\alpha)$ is a nilpotent ideal of $G(\alpha)$, the set
$\big({\rm ad}_{\,G(\alpha)}\,S(\alpha)\big)\cup \big({\rm
ad}_{\,G(\alpha)}\,W\big)$ is weakly closed and consists of
nilpotent endomorphisms. Since $G(\alpha)=W\oplus S(\alpha)$, the
Engel--Jacobson theorem now shows that $G(\alpha)$ is nilpotent.
\end{proof}
\begin{lemm}\label{gen}
If $t\in S_p$ is a toral element not contained in $S$, then $t$ is
conjugate to $D_1^p+D_1+D_H(x_1x_2)$ under the automorphism group of
$S$.
\end{lemm}
\begin{proof}
By our assumption, $t=aD_1^p+w$ for some $a\in F^\times$ and $w\in
S$. Choose $\alpha\in F$ satisfying $\alpha^p=a$ and let
$\sigma_\alpha$ denote the automorphism of $S$ which sends
$D_H(x_1^{(i)}x_2^{(j)})$ to $\alpha^{i-1}D_H(x_1^{(i)}x_2^{(j)})$;
see \cite[Thm.~7.3.6]{St04}. Then
$\sigma_\alpha(t)=-aD_H(\alpha^{-1}x_2)^p+w'$ for some $w'\in S$.
Hence we may assume that $a=1$. The description of ${\rm Aut}\,S$
given in \cite[Thms~7.3.5 \& 7.3.2]{St04} shows that for any pair of
nonnegative integers $(m,n)\ne (p,1)$ such that either $p\le m<p^2$
and $n<p$ or $(m,n)=(p^2,0)$ and any $\lambda\in F$ there is
$\sigma_{m,n,\lambda}\in {\rm Aut}\,S$ such that
$\sigma_{m,n,\lambda}(u)\equiv
u+\lambda\big[D_H(x_1^{(m)}x_2^{(n)}), u\big]\quad \big({\rm
mod}\,S_{i+(m+n-1)}\big)$ for all $u\in S_{(i)}$ Using Jacobson's
formula (with $u=D_1$) it is not hard to observe that
$$\sigma_{m,n,\lambda}(D_1^p)\,\equiv\, D_1^p-\lambda D_H(x_1^{(m-p)}x_2^{(n)})\quad \
\big({\rm mod}\,S_{(m+n-p-1)}\big).$$ This implies that there exists
$g\in{\rm Aut}\,S$ such that
$g(t)=D_1^p+bD_1+D_H(x_1^{(p^2-p)}\psi)$ for some $\psi\in
F[x_1,x_2]\subset {\mathcal O}(2;(1,1))$ with $\psi(0)=0$. Write
$\psi=\sum_{i=0}^{p-1}\,\psi_i x_1^{(i)}$ with $\psi_i\in F[x_2]$,
where $\psi_0(0)=0$. The element $g(t)$ being toral, it must be that
$b=1$. Note that $\big(\ad D_H(x_1^{(p^2-p)}\psi)\big) \big(\ad
(D_1^p+D_1)\big)^i(D_H(x_1^{(p^2-p)}\psi))=0$ for $0\le i\le p-3$
and $$\big(\ad D_H(x_1^{(p^2-p)}\psi)\big) \big({\rm ad}\,
(D_1^p+D_1)\big)^{p-2}(D_H(x_1^{(p^2-p)}\psi))=\big[D_H(x_1^{(p^2-p)}\psi),D_H(x_1^{(p)}\psi)\big].$$
Because $$(\ad D_1^p+\ad
D_1)^{p-1}=\,\textstyle{\sum_{i=0}^{p-1}}\,(-1)^i(\ad D_1)^{pi}(\ad
D_1)^{p-i-1}=\,{\textstyle \sum_{i=1}^p}\,(-1)^{i-1}(\ad
D_1)^{i(p-1)}$$ and $D_1^p(\psi)=0$, Jacobson's formula yields that
\begin{eqnarray*}
g(t)^{[p]}&=&(D_1^p+D_1)^{[p]}+\big({\rm ad}\,(D_1^p+D_1)\big)^{p-1}(D_H(x_1^{(p^2-p)}\psi))\\
&+&\frac{1}{2}\big[D_H(x_1^{(p^2-p)}\psi),D_H(x_1^{(p)}\psi)\big]\\
&=&D_1^p+D_H(\psi)-D_H(x_1^{(p-1)}\psi)+{\textstyle \sum_{i\ge p}}\,
\,D_H(x_1^{(i)}q_i)
\end{eqnarray*}
for some $q_i\in F[x_2]$. As the RHS equals
$D_1^p-D_H(x_2)+D_H(x_1^{(p^2-p)}\psi)$ and
$x_1^{(p-1)}\psi=x_1^{(p-1)}\psi_0$, we derive that $\psi_0=-x_2,\,$
$\psi_i=0$ for $1\le i\le p-2$, and $\psi_{p-1}=\psi_0$. In other
words, $\psi=-(1+x_1^{(p-1)})x_2$ and
$$g(t)\,=\,(D_1^p+D_1)-D_H(x_1^{(p^2-p)}x_2)-D_H(x_1^{(p^2-1)}x_2).$$
Next we show that this element is toral. Note that
$$
(D_1^p+D_1)-D_H(x_1^{(p^2-p)}x_2)-D_H(x_1^{(p^2-1)}x_2)
\,=\,(D_1^p+D_1)-[D_1^p+D_1,D_H(x_1^{(p^2)}x_2)]$$ and
${p^2-1\choose p}-{p^2-1\choose p-1}={p-1\choose 1}-1=-2$ by Lucas'
theorem. Then
\begin{eqnarray*}
\big[D_H(x_1^{(p^2-p)}(1+x_1^{(p-1)})x_2),D_H(x_1^{(p)}(1+x_1^{(p-1)})x_2)\big]&=&
\big[D_H(x_1^{(p^2-p)}x_2),D_H(x_1^{(p)}x_2)\big]\\
&=&-2D_H(x_1^{(p^2-1)}x_2).
\end{eqnarray*}
In view of the earlier computations this gives
\begin{eqnarray*}
&\big(D_1^p+D_1-D_H(x_1^{(p^2-p)}x_2)-D_H(x_1^{(p^2-1)}x_2)\big)^{[p]}\\
&\qquad\qquad\
=D_1^p-\big(\mathrm{ad}(D_1^p+D_1)\big)^p\big(D_H(x_1^{(p^2)}x_2)\big)-D_H(x_1^{(p^2-p)}x_2)\\
&\qquad\
=D_1^p-D_H(x_2)-D_H(x_1^{(p^2-p)}x_2)-D_H(x_1^{(p^2-1)}x_2).
\end{eqnarray*}
So the element $D_1^p+D_1-D_H((x_1^{(p^2-p)}+x_1^{(p^2-1)})x_2)$ is
indeed toral.

As a result, all toral elements in $S_p\setminus S$ are conjugate
under ${\rm Aut}\,S$. To finish the proof it remains to note that
the element $D_1^p+D_1+D_H(x_1x_2)\in S_p\setminus S$ is toral.
\end{proof}
\section{\bf Two-sections in simple Lie algebras}
In this section our standing hypothesis is that $L$ is a
finite-dimensional simple Lie algebra and $T$ is a torus of maximal
dimension in the semisimple $p$-envelope $L_p$ of $L$. Given
$\alpha_1,\ldots,\alpha_s\in\Gamma(L,T)$ we denote by $\rad_T
L(\alpha_1,\ldots,\alpha_s)$ the maximal $T$-{\it invariant}
solvable ideal of the $s$-section $L(\alpha_1,\ldots,\alpha_s)$ and
put
\begin{equation}\label{s-section}
L[\alpha_1,\ldots,\alpha_s]\,:=\,L(\alpha_1,\ldots,\alpha_s)/\rad_T\,
L(\alpha_1,\ldots,\alpha_s).\end{equation} We let
$\widetilde{S}=\widetilde{S}(\alpha_1,\ldots,\alpha_s)$ be the
$T$-socle of $L[\alpha_1,\ldots,\alpha_s]$, the sum of all minimal
$T$-stable ideals of the Lie algebra $L[\alpha_1,\ldots,\alpha_s]$.
Then $\widetilde{S}=\bigoplus_{i=1}^r\,\widetilde{S}_i$ where each
$\widetilde{S}_i$ is a {\it minimal} $T$-stable ideal of
$L[\alpha_1,\ldots,\alpha_s]$. It is immediate from the definition
that both $T$ and $L(\alpha_1,\ldots,\alpha_s)_p$ act on
$L[\alpha_1,\ldots,\alpha_s]$ as derivations and preserve
$\widetilde{S}$. Thus, there is a natural restricted Lie algebra
homomorphism
$T+L(\alpha_1,\ldots,\alpha_s)_p\rightarrow\Der\widetilde{S}$ which
will be denoted by $\Psi_{\alpha_1,\ldots,\,\alpha_s}$. Note that
$L(\alpha_1,\ldots, \alpha_s)\cap \ker
\Psi_{\alpha_1,\ldots,\,\alpha_s}=\rad_T\,
L(\alpha_1,\ldots,\alpha_s)$ and, moreover, the image of
$\Psi_{\alpha_1,\ldots,\,\alpha_s}$ can be identified with a
semisimple restricted Lie subalgebra of $\Der\widetilde{S}$
containing $L[\alpha_1,\ldots,\alpha_s]$ as an ideal.

We often regard the linear functions on $T$ as functions on the
nilpotent restricted Lie algebra ${\mathfrak c}_{L_p}(T)$ by using
the rule $\gamma(x):=\big(\gamma(x^{[p]^e})\big)^{p^{-e}}$ for all
$x\in {\mathfrak c}_{L_p}(T)$, where $e\gg 0$ (this makes sense
because $T$ coincides with the set of all $p$-semisimple elements of
${\mathfrak c}_{L_p}(T)$).

Let ${\rm nil}\,H_p$ denote the maximal $p$-nilpotent ideal of the
restricted Lie algebra $H_p$. According to \cite[Cor.~3.9]{PS4}, the
inclusion $H^4\subset {\rm nil}\, H_p$ holds and all roots in
$\Gamma(L,T)$ are linear functions on $H$.

\begin{lemm}\label{lemA1}
If $\delta\in\Gamma(L,T)$ has the property that $\delta(H)\ne 0$,
then $\delta\big([L_\delta,L_{-\delta}]^2\big)=0$ and
$[L_\delta,L_{-\delta}]^3\subset {\rm nil}\,H_p$.\end{lemm}
\begin{proof} This is immediate from \cite[Prop.~3.4]{PS4}.
\end{proof}
\begin{prop}\label{pro}
Let $\mathfrak t$ be a torus in $L_p$ whose centralizer in $L$ is
nilpotent, and assume further that $\mathfrak t$ contains the all
$p$-semisimple elements of the $p$-envelope of ${\mathfrak
c}_L(\mathfrak t)$ in $L_p$. Let $\eta\in\Gamma(L,{\mathfrak t})$ be
such that $L(\eta)$ is nonsolvable and denote by $S(\eta)$ the socle
of the semisimple Lie algebra $L(\eta)/rad\,L(\eta)$. Then the
following hold:
\begin{itemize}
\item[(1)]
the radical $\rad\,L(\eta)$ is $\mathfrak t$-stable;

\smallskip

\item[(2)]
the socle $S(\eta)$ is a simple Lie algebra invariant under the
action of $\mathfrak t$;

\smallskip

\item[(3)] the centralizer ${\mathfrak c}_S({\mathfrak t})$ is a
Cartan subalgebra of toral rank $1$ in $S$.
\end{itemize}
\end{prop}
\begin{proof} The torus $\mathfrak t$ satisfies the conditions of
\cite[Thm.~3.6]{PS4}. Moreover, our first statement is nothing but
\cite[Thm.~3.6(1)]{PS4}. The last two statements are immediate
consequences of \cite[Thm.~3.6(3)]{PS4} and \cite[Thm.~3.6(4)]{PS4}.
\end{proof}
\begin{theo}\label{1sec}
For every $\gamma\in\Gamma(L,T)$ the radical $\rad\,L(\gamma)$ is
$T$-stable and either $L[\gamma]$ is one of $(0)$,
$\mathfrak{sl}(2)$, $W(1;\un{1})$, $H(2;\un{1})^{(2)}$,
$H(2;\un{1})^{(1)}$ or $p=5$, $L_p$ possesses nonstandard tori of
maximal dimension, and $L[\gamma]\,\cong\, H(2;\un{1})^{(2)}\oplus
F(1+x_1)^4\partial_2$. If $\gamma$ is nonsolvable, then the derived
subalgebra $L[\gamma]^{(1)}$ is simple.
\end{theo}
\begin{proof} This is immediate from \cite[Cor.~3.7]{PS4}.
\end{proof}

\begin{lemm}\label{ham3}
Let ${\mathfrak g}=H(2;\un{1})^{(2)}\oplus F(1+x_1)^{p-1}\partial_2$
and $\mathfrak h$ a Cartan subalgebra of $\mathfrak g$. Then either
$\mathfrak h$ is abelian or ${\mathfrak h}^3$ contains a nonzero
toral element of $\mathfrak g$.
\end{lemm}
\begin{proof}
We regard $\mathfrak g$ as a restricted Lie subalgebra of
$\widetilde{\mathfrak g}:=H(2;\un{1})$. Recall that
$\widetilde{\mathfrak g}=H(2;\un{1})^{(2)}\oplus
FD_H(x_1^{(p)})\oplus FD_H(x_2^{(p)})\oplus
FD_H(x_1^{(p-1)}x_2^{(p-1)})$. Since $\widetilde{\mathfrak
g}^{[p]}\subset H(2;\un{1})^{(2)}$ by Jacobson's formula, $\mathfrak
h$ coincides with ${\mathfrak c}_{\mathfrak g}(y)$ for some nonzero
toral element $y\in H(2;\un{1})^{(2)}$. By a result of Demu{\v
s}kin, there is $\sigma\in {\rm Aut}\,H(2;\un{1})^{(2)}$ such that
either $\sigma(y)=D_H((1+x_1)x_2)$ or $\sigma(y)$ is a nonzero
multiple of $D_H(x_1x_2)$; see \cite[Thm.~7.5.8]{St04}. In the
latter case, there exist $a,b\in F$ such that $\sigma({\mathfrak
h})$ is contained in the span of $aD_H(x_1^{(p)})+bD_H(x_2^{(p)})$
and all $D_H(x_1^{(i)}x_2^{(i)})$ with $1\le i\le p-1$, hence is
abelian. Then $\mathfrak h$ is abelian, too. So assume we are in the
former case. Then there are $a,b,c\in F$ such that $\sigma(\mathfrak
h)$ coincides with the span of all $D_H((1+x_1)^{i}x_2^{(i)})$ with
$1\le i\le p-2$ and
$z:=a(1+x_1)^{p-1}D_2+bD_H(x_2^{(p)})+cD_H((1+x_1)^{p-1}x_2^{(p-1)})$.
If $a=0$, then it is easy to check that $\sigma(\mathfrak h)$ is
abelian, whilst if $a\ne 0$, then $(\ad
z)^2\big(D_H((1+x_1)^3x_2^{(3)})\big)$ is a nonzero multiple of
$\sigma(y)$. This completes the proof.
\end{proof}

\medskip

Next we recall our results on $2$-sections of $L$ with respect to
$T$. Let $\alpha,\beta\in\Gamma(L,T)$ be such that $L(\alpha,\beta)$
is nonsolvable. As explained in \cite[p.~793]{PS4}, the $T$-socle
$\widetilde{S}=\widetilde{S}(\alpha,\beta)$ is either a unique
minimal ideal of $L[\alpha,\beta]$ or
$\widetilde{S}=\widetilde{S}_1\oplus\widetilde{S}_2$, where
$TR(\widetilde{S}_i)=1$ for $i=1,2$ and each $\widetilde{S}_i$ is
$T$-stable. Moreover, in the latter case the following holds:
\begin{theo}[cf. {\cite[Thm.~4.1]{PS4}}] \label{r2}
If $\widetilde{S}=\widetilde{S}_1\oplus\widetilde{S}_2$, then there
exist $\delta_1,\delta_2\in\Gamma(L,T)$ such that
$$L[\delta_1]^{(1)}\oplus L[\delta_2]^{(1)}\subset L[\alpha,\beta]\subset
L[\delta_1]\oplus L[\delta_2].$$
\end{theo}
\noindent When the $T$-socle $\tilde{S}$ is a minimal ideal of
$L[\alpha,\beta]$, we have two possibilities: either
$TR(\widetilde{S})=2$ or $TR(\widetilde{S})=1$.
\begin{theo} \label{r1TR2}
Suppose $\widetilde{S}$ is the unique minimal ideal of
$L(\alpha,\beta)$ and $TR(\widetilde{S})=2$. Then $\widetilde{S}$ is
simple, $\Psi_{\alpha,\beta}(L_\gamma)\subset \widetilde{S}$ for all
$\gamma\in\Gamma(L,T)$, and one of the following holds:
\begin{itemize}
\item[(1)] $\widetilde{S}$ is one of $W(2;\un{1})$, $S(3;\un{1})^{(1)}$,
$H(4;\un{1})^{(1)}$, $K(3;\un{1})^{(1)}$ and
$L[\alpha,\beta]=\widetilde{S}$;

\smallskip

\item[(2)] $\widetilde{S}$ is one of $W(1;\un{2})$,
$H(2;\un{1};\Phi(\tau))^{(1)}$, $H(2;\un{1};\Delta)$ and
$$L[\alpha,\beta]= \widetilde{S}+\Psi_{\alpha,\beta}(T)\cap
L[\alpha,\beta];$$

\item[(3)] $\widetilde{S}\cong {\mathcal M}(1,1)$ and
$L[\alpha,\beta]=\widetilde{S}$;

\smallskip

\item[(4)] $\widetilde{S}$ is a classical Lie algebra of type ${\rm
A}_2$, ${\rm B}_2$ or ${\rm G}_2$ and
$L[\alpha,\beta]=\widetilde{S}$;

\smallskip

\item[(5)] $\widetilde{S}=H(2;(2,1))^{(2)}$
and $\Psi_{\alpha,\beta}(T)\subset \widetilde{S}_p$.
Moreover,
$$H(2;(2,1))^{(2)}\subset L[\alpha,\beta]\subset
H(2;(2,1))^{(2)}\oplus FD_H(x_1^{(p^2)})\oplus
FD_H(x_1^{(p^2-1)}x_2^{(p-1)}).$$
\end{itemize}
\noindent In cases (1), (3), (4) the Lie algebra $L[\alpha,\beta]$
is simple, and  $L[\alpha,\beta]^{(1)}$ is simple in all cases.
\end{theo}
\begin{proof}
If $\widetilde{S}$ is not isomorphic to $H(2;(2,1))^{(2)}$, then the
statement follows immediately from \cite[Thm.~4.2]{PS4}. So assume
$\widetilde{S}\cong H(2;(2,1))^{(2)}$. Then \cite[Thm.~4.2]{PS4}
says that $L[\alpha,\beta]\subset {\mathcal G}$ where $\mathcal G$
is the $p$-envelope of $G=H(2;(2,1))$ in $\Der \widetilde{S}$.
Recall that
$\Psi_{\alpha,\beta}\colon\,T+L(\alpha,\beta)_p\rightarrow \Der
\widetilde{S}$ is a restricted Lie algebra homomorphism. Hence
$\widetilde{S}_p$ lies in the image of $\Psi_{\alpha,\beta}$. In the
present case, $\Der \widetilde{S}={\mathcal G}\oplus
F(x_1D_1+x_2D_2)$; see \cite[Prop.~2.1.8(vii)]{BW} for instance. If
$\Psi_{\alpha,\beta}(T)\not \subset \mathcal G$, then there is a
surjective restricted Lie algebra homomorphism
$\Psi_{\alpha,\beta}(T+L(\alpha,\beta)_p)\twoheadrightarrow
F(x_1D_1+x_2D_2)$ whose kernel contains $\widetilde{S}_p$. But then
\cite[Lemma~2.4.4(2)]{StF} yields that the restricted Lie algebra
$\Psi_{\alpha,\beta}(T+L(\alpha,\beta)_p)$ contains $3$-dimensional
tori, a contradiction. Consequently,
$\Psi_{\alpha,\beta}(T+L(\alpha,\beta)_p)\subset\mathcal G$, forcing
$\Psi_{\alpha,\beta}(T)\subset {\mathcal G}^{[p]}\subset
\widetilde{S}_p$.

Let ${\mathfrak t}'$ be an optimal $2$-dimensional torus in
$\widetilde{S}_p$. By \cite[Lemma~1.7.2(b)]{BW}, there is a torus
$T'$ of maximal dimension in $T+L(\alpha,\beta)_p$ such that
$\Psi_{\alpha,\beta}(T')={\mathfrak t}'$. Let $H'$ denote the
centralizer of $T'$ in $L$. Note that
$L(\alpha,\beta)=L(\alpha',\beta')$ for some
$\alpha',\beta'\in\Gamma(L,T')$ (this follows from the main result
of \cite{P89} and \cite[Cor.~2.10]{PS2}). Each $i\alpha'+j\beta'$
with $i,j\in {\mathbb F}_p$ can be viewed as a linear function of
${\mathfrak t}'$.

Since ${\mathfrak t}'$ is optimal, ${\mathfrak t}'\cap
\widetilde{S}={\mathfrak t}'\cap \widetilde{S}_{(0)}$ is spanned by
a nonzero toral element, $t_2$ say; see \cite[(VI.1)]{St92}. Since
$\Gamma(\widetilde{S},{\mathfrak t}')\cup\{0\}$ is a $2$-dimensional
vector space over ${\mathbb F}_p$, by Lemma~\ref{ham1}(4), there is
$\delta'\in\Gamma(L(\alpha,\beta),T')$ such that $\delta'(t_2)=0$.
Since, then, $\delta'$ also vanishes on ${\mathfrak
c}_{\widetilde{S}}({\mathfrak t}')$, the Engel--Jacobson theorem
yields that $\widetilde{S}(\delta')$ is nilpotent. Since ${\mathcal
G}/\widetilde{S}$ is solvable, ${\mathcal G}(\delta')$ must be,
also. But then $L(\delta')$ is solvable, too. As explained in
\cite[(VI.4)]{St92} the union $\bigcup_{i\in{\mathbb
F}_p^\times}\,\widetilde{S}_{i\delta'}$ contains a nonnilpotent
element of $\mathcal G$. Hence $\bigcup_{i\in{\mathbb
F}_p^\times}\,L_{i\delta'}$ contains a nonnilpotent element of
$L_p$. Since $L_{i\delta'}\subset \rad\, L(\delta')$ for all
$i\in{\mathbb F}_p^\times$, it follows from \cite[Prop.~3.8]{PS4}
that $\delta'$ vanishes on $H'$.

Recall that $\widetilde{S}_p=FD_1^p\oplus \widetilde{S}$ and
${\mathcal G}=S_p\oplus V$, where $V$ is the $F$-span of
$D_H(x_1^{(p^2)}),\,$ $D_H(x_2^{(p)})$ and
$D_H(x_1^{(p^2-1)}x_2^{(p-1)})$. Hence
$\mathcal{G}^3\subset\widetilde{S}$. Pick a toral element $t_1\in
{\mathfrak t}'\setminus\widetilde{S}$ (such an element exists by
Lemma~\ref{tori}). By Lemma~\ref{gen}, we may assume that
$t_1=D_1^p+D_1+D_H(x_1x_2)$ (one should keep in mind here that
$\widetilde{S}_{(0)}$ is invariant under all automorphisms of $S$;
see \cite[Thm.~4.2.6]{St04}). Set $V':=({\rm Id}-(\ad
t_2)^{p-1})({\rm Id}-(\ad t_1)^{p-1})(V).$ Then
$${\mathfrak c}_{\widetilde{S}}({\mathfrak t}')\subset
\Psi_{\alpha,\beta}(H')\subset {\mathfrak c}_{\mathcal G}({\mathfrak
t}')={\mathfrak c}_{\widetilde{S}_p}({\mathfrak t}')\oplus
V',\qquad\, {\mathfrak c}_{\mathcal G}({\mathfrak t}')^3\subset\,
{\mathfrak c}_{\widetilde{S}}({\mathfrak t}')\subset\,
\Psi_{\alpha,\beta}(H').$$ The elements $\big({\rm Id}-(\ad
t_1)^{p-1}\big)(D_H(x_1^{(p^2)}))$ and $\big({\rm Id}-(\ad
t_1)^{p-1}\big)(D_H(x_1^{(p^2-1)}x_2^{(p-1)}))$ lie in
$G_{(p-2)}\subset G_{(1)}$ whereas $[t_1,D_H(x_2^{(p)})]=0$.
Consequently, $\big({\rm Id}-(\ad t_1)^{p-1}\big)(V)\subset
G_{(1)}$. As $\ad t_2$ preserves $G_{(1)}$ we get $V'\subset
G_{(1)}$.

We claim that $L[\alpha,\beta]\subset G$. Indeed, suppose the
contrary. Recall that $G=\widetilde{S}\oplus V'\subsetneq
L[\alpha,\beta]+V'$ and $\mathcal{G}=\widetilde{S}\oplus
FD_1^p\oplus V'$. Then ${\mathcal G}=L[\alpha,\beta]+ V'$, hence
$${\mathfrak t}'
\subset\, {\mathfrak c}_{\mathcal G}({\mathfrak t}') = \,{\mathfrak
c}_{L[\alpha,\beta]+ V'}({\mathfrak
t}')=\,\Psi_{\alpha,\beta}(H')+V'.$$ Since
$\big(\Psi_{\alpha,\beta}(H')+V'\big)^3\subset
\Psi_{\alpha,\beta}(H')$, Jacobson's formula and induction on $k$
enable us to deduce that
$\big(\Psi_{\alpha,\beta}(H')+V'\big)^{[p]^k}\subset\,(V')^{[p]^k}+
\textstyle{\sum_{i=0}^k}\,\Psi_{\alpha,\beta}(H')^{[p]^k}$ for all
$k\ge 0$. From our earlier remarks we know that $V'\subset G_{(1)}$
consists of $p$-nilpotent elements of $\mathcal G$. Therefore,
$\big(\Psi_{\alpha,\beta}(H')+V'\big)^{[p]^e}\,\subset\sum_{i=0}^e\,\Psi_{\alpha,\beta}(H'))^{[p]^i}$
for all sufficiently large $e$. Since $H'$ is nilpotent, this forces
${\mathfrak t}'=({\mathfrak
t}')^{[p]^e}\subset\big(\Psi_{\alpha,\beta}(H')\big)^{[p]^e}$ for
$e\gg 0$. But then $\delta'$ vanishes on ${\mathfrak t}'$. By
contradiction, the claim follows.

Suppose $L[\alpha,\beta]\not\subset H(2;(2,1))^{(2)}\oplus
FD_H(x_1^{(p^2)})\oplus F D_H(x_1^{(p^2-1)}x_2^{(p-1)})$ and pick
$\mu\in F^\times$. Recall the elements $t_\mu\in\widetilde{S}$ and
$h_\mu\in \mathfrak{c}_G(t_\mu)$ from Lemma~\ref{conj}. Our present
assumption on $L[\alpha,\beta]$ implies that
$\mathfrak{c}_{L[\alpha,\beta]}(t_\mu)\supsetneq C_\mu'$; see
Lemma~\ref{centraliser}(i). As $L[\alpha,\beta]\subset G$ by our
remarks earlier in the proof, $L[\alpha,\beta]$ contains an element
from $(G\cap C_\mu)\setminus C_\mu'$; call it $h$. In view of
Lemma~\ref{stab}(ii), we may assume that $h=h_\mu+s
D_H(x_1^{(p^2)})$ for some $s\in F$.

Let $h_0$ denote the $p$-semisimple part of $h$ in the $p$-envelope
of $L[\alpha,\beta]$ in $\mathcal G$. It is immediate from
Lemma~\ref{stab}(iv) that the elements $h_0$ and $t_\mu$ are
linearly independent. This implies that ${\mathfrak
t}_\mu:=Fh_0\oplus Ft_\mu$ is a torus of maximal dimension in
$\mathcal G$. Recall that the restricted Lie algebra homomorphism
$\Psi_{\alpha,\beta}$ takes $T+L(\alpha,\beta)_p$ into
$\mathcal{G}$. Hence it follows from \cite[Lemma~2.4.4(2)]{StF} that
there exists a torus of maximal dimension $T''$ in $L_p$ contained
in $T+L(\alpha,\beta)_p$ and such that
$\mathfrak{t}_\mu=\Psi_{\alpha,\beta}(T'')$ and
$T\cap\ker\alpha\cap\ker\beta\subset T\cap T''$. We denote by $H''$
the centralizer of $T''$ in $L$. By construction, there exists
$\tilde{h}\in H''$ with $\Psi_{\alpha,\beta}(\tilde{h})=h$.

Set $T_0:=T\cap\ker\alpha\cap\ker\beta$. Because
$L(\alpha,\beta)=\mathfrak{c}_L(T_0)$, it is straightforward to see
that $L(\gamma'')=L(\alpha,\beta)(\gamma'')$ for every
$\gamma''\in\Gamma(L,T'')$ with $\gamma''(T_0)=0$. Since
$\Psi_{\alpha,\beta}(T'')=\mathfrak{t}_\mu$, there exists
$\delta''\in\Gamma(L,T'')$ such that $\delta''(T_0)=0,\,$
$\delta''(t_\mu)=0$ and $\delta''(h_0)\ne 0$; see
Lemma~\ref{ham1}(4). Then $C_\mu'\subset
\Psi_{\alpha,\beta}\big((L(\alpha,\beta))(\delta'')\big)\subset
C_\mu$ and $\delta''(\tilde{h})\ne 0$. Since
$(L(\alpha,\beta))(\delta'')=L(\delta'')$ by the preceding remark,
Lemma~\ref{centraliser}(i) shows that $\delta''$ is a solvable root
which does not vanish on $H''$. In view of \cite[Prop.~3.8]{PS4},
this entails that that every root space
$L_{i\delta''}=(\rad\,L(\delta''))_{i\delta''}$, where
$i\in\mathbb{F}_p$, consists of $p$-nilpotent elements of $L_p$.
Since $\Psi_{\alpha,\beta}$ is a restricted Lie algebra
homomorphism, this means that for every $\lambda\in F^\times$ all
$\lambda$-eigenvectors of the linear operator $(\ad h)_{\vert
C_\mu'}$ must act nilpotently on $\widetilde{S}$. As this
contradicts Lemma~\ref{stab}(iv), we now derive that our present
assumption is false. Thus, $L[\alpha,\beta]\subset
H(2;(2,1))^{(2)}\oplus FD_H(x_1^{(p^2)})\oplus F
D_H(x_1^{(p^2-1)}x_2^{(p-1)})$, completing the proof.
\end{proof}

\noindent If $\widetilde{S}$ is a minimal ideal of $L[\alpha,\beta]$
and $TR(\widetilde{S})=1$, then \cite[Thm.~4.4]{PS4} implies the
following:
\begin{theo}\label{r1TR1}
Suppose $\widetilde{S}$ is a unique minimal ideal of
$L(\alpha,\beta)$ and $TR(\widetilde{S})=1$. Then there exists
$\delta\in{\mathbb F}_p\alpha+{\mathbb F}_p\beta$ such that
$\Psi_{\alpha,\beta}(L_{\gamma})\subset\widetilde{S}$ for all
$\gamma\in\Gamma(L,T)\setminus{\mathbb F}_p\delta$. Moreover, one of
the following holds:
\begin{itemize}
\item[(1)] $L[\alpha,\beta]=L[\eta]$ for some $\eta\in\Gamma(L,T)\cap
({\mathbb F}_p\alpha+{\mathbb F}_p\beta)$;

\smallskip

\item[(2)] $\widetilde{S}\cong H(2;\un{1})^{(2)}$, $\,\,L[\alpha,\beta]\subset
\Der H(2;\un{1})^{(2)}$ and $\dim\,\Psi_{\alpha,\beta}(T)=2$;

\smallskip

\item[(3)] $S\ot{\mathcal O}(m;\un{1})\subset
L[\alpha,\beta]\subset (\Der S)\ot{\mathcal O}(m;\un{1})\rtimes
({\rm Id}\ot W(m;\un{1}))$, where  $S$ is one of $\mathfrak
{sl}(2)$, $W(1;\un{1})$, $H(2;\un{1})^{(2)}$, $\,\widetilde{S}\cong
S\ot{\mathcal O}(m;\un{1})$, and $m>0$.
\end{itemize}
In cases~(1) and ~(2) one can take $\delta=0$, i.e.
$\Psi_{\alpha,\beta}(L_{\gamma})\subset\widetilde{S}$ for all
$\gamma\in\Gamma(L,T)$.
\end{theo}
More information on the two-sections of $L$ can be found in
\cite[Sect.~4]{PS4}.

\section{\bf Nonstandard tori of maximal dimension}
From now on we assume that $T$ is a nonstandard torus of maximal
dimension in the semisimple $p$-envelope $L_p$ of $L$. In light of
\cite[Thm.~1]{P94} this implies that $p=5$. As explained in Sect.~2,
the linear functions on $T$ can be regarded as functions on the
nilpotent restricted Lie algebra ${\mathfrak c}_{L_p}(T)$. Set
$H:={\mathfrak c}_L(T)$ and define
$$\Omega=\Omega(L,T):=\{\delta\in\Gamma(L,T)\,|\,\,\,\delta(H^3)\ne 0\}.$$
As $T$ is a torus of maximal dimension in $L_p$, it is immediate
from \cite[Thm.~1(ii)]{P94} that there exist ${\mathbb
F}_p$-independent roots $\alpha,\beta\in\Gamma(L,T)$ for which
$L[\alpha,\beta]\cong{\mathcal M}(1,1)$. By Lemmas~4.1 and 4.4 of
\cite{P94}, we then have $i\alpha+j\beta\in\Omega$ for all nonzero
$(i,j)\in{\mathbb F}_p^2$. In particular, $\Omega\ne\emptyset$. In
view of Schue's lemma \cite[Prop.~1.3.6(1)]{St04}, this yields
\begin{equation}\label{eq0} L_\gamma\,=\,\textstyle{\sum}_{\delta\in\,\Omega}\,\,
[L_\delta,L_{\gamma-\delta}]\qquad\qquad\,
\big(\forall\,\gamma\in\Gamma(L,T)\cup\{0\}\big).\end{equation}

Because of \cite[Thm.~1(ii)]{P94} we can also assume that $TR(L)\ge
3$. Our main goal in this section is to give a preliminary
description of the $2$-sections of $L$ relative to $T$. More
precisely, we will go through all possible types of $2$-sections
(described in Sect.~3) and eliminate some of them by using our
assumption on $T$.
\begin{lemm}\label{lemA3}
For any nonsolvable $\alpha\in\Omega$ there exists
$\beta\in\Gamma(L,T)$ such that $L[\alpha,\beta]\cong {\mathcal
M}(1,1)$ and
$\alpha\big([L_{i\alpha},L_{-i\alpha}],[L_\beta,L_{-\beta}]\big)\ne
0$ for some $i\in{\mathbb F}_p^\times$.\end{lemm} \begin{proof}
Since $\alpha$ is nonsolvable and $\alpha(H^3)\ne 0$,
Theorem~\ref{1sec} implies that $L[\alpha]\cong
H(2;\un{1})^{(2)}\oplus F(1+x_1)^4\partial_1.$ By
\cite[Thm.~3.5]{PS4}, there is $k\in{\mathbb F}_p^\times$ for which
the set
$\Omega_1:=\{\delta\in\Gamma(L,T)\,|\,\,\,\delta\big([L_{k\alpha},L_{-k\alpha}]\big)\ne
0\}$ is nonempty. Since $\Psi_\alpha(H)\cap H(2;\un{1})^{(2)}$ has
codimension one in $\Psi_\alpha(H)$, Schue's lemma
\cite[Prop.~1.3.6(1)]{St04} implies that there exists
$\beta\in\Omega_1$ with the property that
$$\Psi_\alpha(H)\,=\,\Psi_\alpha(H)\cap H(2;\un{1})^{(2)}+
\Psi_\alpha\big([L_\beta,L_{-\beta}]\big).$$ Hence there exist
$h_1\in L(\alpha)^{(\infty)}\cap H$ and $h_2\in
[L_\beta,L_{-\beta}]$ with $\alpha([h_2,[h_2,h_1]])\ne 0$. Note that
$\beta([h_2,[h_2,h_1]])\in \beta\big([L_\beta,L_{-\beta}]^2\big)=0$
by Lemma~\ref{lemA1}. In particular, $\alpha$ and $\beta$ are
linearly independent over ${\mathbb F}_p$. Since $\beta\in\Omega_1$,
we then have
\begin{eqnarray}\label{h1h2}
\beta([h_2,[h_2,h_1]])= 0;\ \quad \alpha([h_2,[h_2,h_1]])\ne 0;\
\quad \beta\big([L_{k\alpha},L_{-k\alpha}]\big)\ne 0.
\end{eqnarray}

We now look more closely at the $T$-semisimple quotient
$L[\alpha,\beta]$ of the $2$-section $L(\alpha,\beta)$. Since
$\alpha$ is nonsolvable, $L[\alpha,\beta]\ne (0)$. Let
$\widetilde{\mathcal S}$ denote the $p$-envelope of the $T$-socle
$\widetilde{S}$ of $L[\alpha,\beta]$ in $\Der\widetilde{S}$, and set
$u:=\Psi_{\alpha,\beta}\big([h_2,[h_2,h_1]]\big)$. Given
$x\in\widetilde{S}$ we write $x_s$ for the $p$-semisimple part of
$x$ in $\widetilde{\mathcal S}$. Because the roots $\alpha, \beta$
are ${\mathbb F}_p$-independent, $h_1\in L(\alpha)^{(\infty)}\cap
H=\sum_{j\in{\mathbb F}_p^\times}\,[L_{j\alpha},L_{-j\alpha}]$ and
$\,h_2\in [L_\beta,L_{-\beta}]$, it follows from
Theorems~\ref{1sec}, \ref{r2},~\ref{r1TR2} and ~\ref{r1TR1} that
$u\in \widetilde{S}$. Now relations~(\ref{h1h2}) enable us to find
$v\in
\widetilde{S}\cap\Psi_{\alpha,\beta}\big([L_{k\alpha},L_{k\alpha}]\big)$
such that the span of $u_s$ and $v_s$ is $2$-dimensional. This
yields $\Psi_{\alpha,\beta}(T)\subset\widetilde{\mathcal S}$ showing
that $TR(\widetilde{S})=2$. Since
$\beta\big([L_{k\alpha},L_{-k\alpha}]\big)\ne 0$, we also deduce
that there are ${\mathbb F}_p$-independent
$\delta_1,\delta_2\in\Gamma(L,T)$ for which
$\big[\Psi_{\alpha,\beta}(L_{\delta_1}),\Psi_{\alpha,\beta}(L_{\delta_2})\big]\ne
0$. In view of Theorem~\ref{r2}, this implies that $\widetilde{S}$
is a minimal ideal of $L[\alpha,\beta]$.

Theorem~\ref{r1TR2} now says that $\widetilde{S}$ is a simple Lie
algebra and $\Psi_{\alpha,\beta}(L_\gamma)\subset \widetilde{S}$ for
all $\gamma\in\Gamma(L,T)\cap ({\mathbb F}_p\alpha+{\mathbb
F}_p\beta)$. Since $\alpha\big(H,[L_{k\alpha},L_{-k\alpha}]\big)\ne
0$,  the torus $\Psi_{\alpha,\beta}(T)\subset \widetilde{\mathcal
S}=\widetilde{S}_p$ is nonstandard. Applying \cite[Thm.~1(ii)]{P94}
we conclude that $L[\alpha,\beta]\cong{\mathcal M}(1,1)$, finishing
the proof. \end{proof}
\begin{prop}\label{types}
If $\alpha\in\Omega$ and $\beta\in\Gamma(L,T)$, then one of the
following occurs:
\begin{itemize}
\item[1)]
$L[\alpha,\beta]=(0)$.

\medskip

\item[2)]
$L[\alpha,\beta]=L[\delta]$ for some $\delta\in\Gamma(L,T)$.

\medskip

\item[3)]
$L[\delta_1]^{(1)}\oplus L[\delta_2]^{(1)}\subset
L[\alpha,\beta]\subset L[\delta_1]\oplus L[\delta_2]$ for some
$\delta_1,\delta_2\in\Gamma(L,T)$.

\medskip

\item[4)]
$S\ot{\mathcal O}(m;\un{1})\subset L[\alpha,\beta]\subset (\Der
S)\ot{\mathcal O}(m;\un{1})\rtimes ({\rm Id}\ot W(m;\un{1}))$, where
$S$ is one of $\mathfrak{sl}(2)$, $W(1;\un{1})$,
$H(2;\un{1})^{(2)}$, $\,\widetilde{S}\cong S\ot{\mathcal
O}(m;\un{1})$, and $m>0$.

\medskip

\item[5)]
$H(2;(2,1))^{(2)}\subset L[\alpha,\beta]\subset H(2;(2,1))$ and
$\widetilde{S}=H(2;(2,1))^{(2)}=L[\alpha,\beta]^{(1)}$. Furthermore,
 each
$\eta\in\Gamma(L[\alpha,\beta],\Psi_{\alpha,\beta}(T))$ is
Hamiltonian, $\eta\big(\Psi_{\alpha,\beta}(T)\cap
\widetilde{S}\big)\ne 0$, and
$\Gamma(L[\alpha,\beta],\Psi_{\alpha,\beta}(T))\,=\,({\mathbb
F}_p\alpha\oplus{\mathbb F}_p\beta)\setminus\{0\}$.

\medskip

\item[6)]
$L[\alpha,\beta]\cong{\mathcal M}(1,1)$.
\end{itemize}
\end{prop}
\begin{proof} (a) Set $\overline{T}:=\Psi_{\alpha,\beta}(T)$ and
$\overline{H}:=\Psi_{\alpha,\beta}(H)$. If
$\Gamma(L[\alpha,\beta],\overline{T})=\emptyset$, then
$L(\alpha,\beta)$ is solvable, forcing $L[\alpha,\beta]=(0)$. If
$\emptyset\ne \Gamma(L[\alpha,\beta],\overline{T})\subset {\mathbb
F}_p\delta$ for a single root $\delta$, then for any $\delta'\in
({\mathbb F}_p\alpha\oplus{\mathbb F}_p\beta)\setminus{\mathbb
F}_p\delta$  we have that $L_{\delta'}\subset
\rad_{T}\,L(\alpha,\beta)$. Then $L[\alpha,\beta]=L[\delta]$. So we
may assume from now that $\Gamma(L[\alpha,\beta],\overline{T})$
contains two roots independent over ${\mathbb F}_p$. Then
$L[\alpha,\beta]$ is described in Theorems~\ref{r2}, ~\ref{r1TR2}
and ~\ref{r1TR1}. Let $\widetilde{S}$ be the $T$-socle of
$L[\alpha,\beta]$. If $\widetilde{S}$ is not a minimal ideal of
$L[\alpha,\beta]$, then Theorem~\ref{r2} says that we are in case~3)
of this proposition. Thus, we may assume further that
$\widetilde{S}$ is a minimal ideal of $L[\alpha,\beta]$.

\smallskip

\noindent (b) Suppose $TR(\widetilde{S})=2$. Then $L[\alpha,\beta]$
is described in Theorem~\ref{r1TR2}. Since $\alpha(H^3)\ne 0$, there
exists $\eta\in\Gamma(\widetilde{S},\overline{T})$ with
$\eta\big(\overline{H}^3\big)\ne 0$. In cases~(1)\,--\,(4) of
Theorem~\ref{r1TR2} we have $\overline{H}^3\subset
\big(\overline{T}+\overline{H}\cap\widetilde{S}\big)^3=\big(\overline{H}\cap\widetilde{S}\big)^3$,
implying that $\overline{H}'={\mathfrak
c}_{\widetilde{S}}(\overline{T})$ acts nontriangulably on
$\widetilde{S}$. But then \cite[Thm.~1(ii)]{P94} shows that
$\widetilde{S}\cong {\mathcal M}(1,1)$. This brings up case~6) of
this proposition.

\smallskip

\noindent (c) Suppose $L[\alpha,\beta]$ is as in case~(5) of
Theorem~\ref{r1TR2}. Then $\widetilde{S}\cong H(2;(2,1))^{(2)}$ and
$L[\alpha,\beta]\subset H(2;(2,1))^{(2)}\oplus
FD_H(x_1^{(p^2)})\oplus FD_H(x_1^{(p^2-1)}x_2^{(p-1)})$.
Furthermore, $\overline{T}\subset \widetilde{S}_p$. If no root in
$\Gamma(\widetilde{S},\overline{T})$ vanishes on $\overline{T}\cap
\widetilde{S}$, then Lemma~\ref{ham1}(2) shows that we are in
case~5) of this this proposition. So assume for a contradiction that
there is $\delta\in\Gamma(\widetilde{S},\overline{T})$ with
$\delta(\overline{T}\cap\widetilde{S})=0$. By Lemma~\ref{tori}, we
have $\overline{T}\cap \widetilde{S}=Fu_2\ne (0)$. Since $\delta$
vanishes on $u_2\in \overline{T}\cap\widetilde{S}$, we may assume
without loss that $u_2$ is a toral element. As before, we put
$G=H(2;(2,1))$ and ${\mathcal G}=\widetilde{S}_p\oplus V$, where
$V\subset\Der\widetilde{S}$ is defined in Sect.~2. Since
$\alpha\in\Omega$, the Lie algebra $\overline{H}^3$ acts
nonnilpotently on $S$.

\smallskip
\noindent (c1) We first suppose that $\overline{T}\cap
\widetilde{S}\not\subset S_{(0)}$. Then we can find
$\Psi_{\alpha,\beta}$ such that $\overline{T}\cap
\widetilde{S}=Ft_\mu$ where $\mu\in F$; see Lemma~\ref{conj}. Thus,
no generality will be lost by assuming that $u_2=t_\mu$. But then it
follows from Lemma~\ref{centraliser}(i) that
$$\overline{H}\subset\, C_\mu\cap\Big(H(2;(2,1))^{(2)}\oplus
FD_H(x_1^{(p^2)})\oplus
FD_H(x_1^{(p^2-1)}x_2^{(p-1)})\Big)\,=\,C_\mu'$$ and
$[\overline{H},\overline{H}]\subset [C_\mu',C_\mu']=(0)$. Since
$\overline{H}$ acts nontriangulably on $\widetilde{S}$, this is
impossible.

\smallskip

\noindent (c2) Now suppose that $\overline{T}\cap
\widetilde{S}\subset S_{(0)}$. Then $\overline{T}\cap
\overline{S}_{(0)}$ contains a nonzero $p$-semisimple element, say
$t$; see Lemma~\ref{tori}.  It follows from Lemma~\ref{tori} and our
earlier remarks that $\mathcal{G}=\overline{T}+G$. As
$\mathrm{gr}\,t\in G_{(0)}/G_{(1)}\cong \mathfrak{sl}(2)$ acts
invertibly on $G_{(-1)}=G/G_{(0)}$, this implies that
$\overline{H}\subset\overline{T}+\mathfrak{c}_G(\overline{T})=\overline{T}+
\mathfrak{c}_{G_{(1)}}(\overline{T}).$ But then
$\overline{H}^{(1)}\subset G_{(1)}$ acts nilpotently on $G$, a
contradiction.

As a result, no root in $\Gamma(\widetilde{S},\overline{T})$
vanishes on $\overline{H}\cap \widetilde{S}$ and we are in case~5)
of this proposition; see Lemma~\ref{ham1}(2).

\smallskip

\noindent (d) If $L[\alpha,\beta]$ is as in case (1) of
Theorem~\ref{r1TR1}, then it is listed in the present proposition as
type~2). If $L[\alpha,\beta]$ is as in case~(2) of
Theorem~\ref{r1TR1}, then $\widetilde{S}=H(2;\un{1})^{(2)}$,
$L[\alpha,\beta]\subset \Der H(2;\un{1})^{(2)}$, and  $\overline{T}$
is a $2$-dimensional torus in $\Der \widetilde{S}$. It is well-known
that any $2$-dimensional torus in $\Der\widetilde{S}$ is
self-centralizing; see \cite[(III.1)]{St92} for instance. But then
$\gamma(H^{(1)})=0$ for all $\gamma\in {\mathbb
F}_p\alpha\oplus{\mathbb F}_p\beta$. Thus, this case cannot occur in
our situation. Finally, case~(3) of Theorem~\ref{r1TR1} is listed as
type~4) in the present proposition.
\end{proof}
\begin{cor}\label{not4}
Let $\alpha\in\Omega$ and $\beta\in\Gamma(L,T)$. If
$L[\alpha,\beta]$ is as in cases~1), 2), 3), 5) or 6) of
Proposition~\ref{types}, then $\sum_{i\in{\mathbb
F}_p^\times}\,(\rad\,L(\gamma))_{i\gamma}\subset
\rad_T\,L(\alpha,\beta)$ for all nonzero $\gamma\in {\mathbb
F}_p\alpha+{\mathbb F}_p\beta$.
\end{cor}
\begin{proof}
If $L[\alpha,\beta]$ is of type~1) or 2), then all $1$-sections of
$L[\alpha,\beta]$ are semisimple and there is nothing to prove. If
$L[\alpha,\beta]$ is of type~3), then there are $h_i\in
\overline{H}\cap L[\delta_i]$ such that $\delta_i(h_i)\ne 0$, where
$i=1,2$ (recall that $\overline{H}=\Psi_{\alpha\,\beta}(H)$). It
follows that $\rad\,L[\alpha,\beta](\delta_i)\subset
\overline{H}+L[\delta_i]^{(1)}.$ As each $L[\delta_i]^{(1)}$ is
simple, we get $\rad\big(L[\alpha,\beta](\gamma)\big)\subset
\overline{H}$ for all nonzero $\gamma\in {\mathbb
F}_p\alpha\oplus{\mathbb F}_p\beta$. If $L[\alpha,\beta]$ is of
type~5) or 6), then all $T$-roots of $L[\alpha,\beta]$ are
Hamiltonian and the corresponding root spaces are $5$-dimensional
(see Lemma~\ref{ham1} and \cite[Lemmas~4.1 \& 4.4]{P94}). Hence in
these cases $\rad\big(L[\alpha,\beta](\gamma)\big)\subset
\overline{H}$ for all $\gamma\in ({\mathbb F}_p\alpha\oplus{\mathbb
F}_p\beta)\setminus\{0\}$.
\end{proof}
\begin{lemm}\label{omega'}
The following hold for every $\gamma\in\Gamma(L,T)$ with
$\gamma(H)\ne 0$:
\begin{itemize}
\item[(a)] All elements in $\,\textstyle{\bigcup_{i\in{\mathbb
F}_p^\times}}\,\big(H^3\cap \,[(\rad\,
L(\gamma))_{i\gamma},L_{-i\gamma}]\big)$ are $p$-nilpotent in $L_p$.

\smallskip

\item[(b)] If $\gamma\in\Omega$, then all elements in $\,\bigcup_{i\in{\mathbb
F}_p^\times}\,\big((\rad\, L(\gamma))_{i\gamma}\cup\,[(\rad\,
L(\gamma))_{i\gamma},L_{-i\gamma}]\big)$ are $p$-nilpotent in $L_p$.
\end{itemize}
\end{lemm}
\begin{proof}
We will treat both cases simultaneously. Set
\begin{eqnarray*}
\Omega'&:=&\Big\{\alpha\in\Gamma(L,T)\,|\,\,\,\,
\alpha\Big(\textstyle{\bigcup_{i\in{\mathbb
F}_p^\times}}\,\big(H^3\,\cap\,[(\rad\,
L(\gamma))_{i\gamma},L_{-i\gamma}]\Big)\ne 0 \Big\},\\
\Omega''&:=&\Big\{\alpha\in\Gamma(L,T)\,|\,\,\,\,
\alpha\Big(\textstyle{\bigcup_{i\in{\mathbb
F}_p^\times}}\,\big((\rad\,
L(\gamma))_{i\gamma}^{[p]}\,\cup\,[(\rad\,
L(\gamma))_{i\gamma},L_{-i\gamma}]\big)\Big)\ne 0 \Big\}.
\end{eqnarray*}
Assume for a contradiction that either $\Omega'\ne\emptyset$ or
$\gamma\in\Omega$ and $\Omega''\ne\emptyset$. Note that
$\Omega'\subset\Omega''\cap\Omega$. Since $\gamma(H)\ne 0$, Schue's
lemma \cite[Prop.~1.3.6(1)]{St04} shows that there exists
$\mu\in\Omega'$ or $\mu\in\Omega''$ for $\gamma\in\Omega$ such that
\begin{equation}\label{gamma-mu}
\gamma([L_\mu,L_{-\mu}])\ne 0.
\end{equation}
In both cases, the type of $L[\gamma,\mu]$ is determined by
Proposition~\ref{types}. If $L[\gamma,\mu]$ is as in cases~1), 2),
3), 5) or 6) of Proposition~\ref{types}, then $\sum_{i\in{\mathbb
F}_p^\times}\,(\rad\,
L(\gamma))_{i\gamma}\subset\,\rad_T\,L(\gamma,\mu)$ by
Corollary~\ref{not4}. Since $\mu\in\Omega''$ in both cases, this
yields $L_{\pm\mu}\subset\, \rad_T\, L(\gamma,\mu)$. Easy induction
on $n$ based on (\ref{gamma-mu}) now gives
$$\sum_{i\in{\mathbb F}_p^\times}\,(\rad\,L(\gamma))_{i\gamma}
\,\subset\,\,\bigcap_{n\ge
1}\,\big(\rad_T\,L(\gamma,\mu)\big)^{(n)}\,=\,(0).
$$
Since this contradicts our assumption that either $\Omega'$ or
$\Omega''$ is nonempty, $L[\gamma,\mu]$ must be of type~4). Then the
minimal ideal of $L[\gamma,\mu]$ has the form
$\widetilde{S}=S\ot{\mathcal O}(m;\un{1})$, where $S$ is a
restricted simple Lie algebra of absolute toral rank $1$ and $m>1$.
According to \cite[Thm.~3.2]{PS2} we can choose $\Psi_{\gamma,\mu}$
such that $\overline{T}=\Psi_{\gamma,\mu}(T)$ has the form $F(h_0\ot
1)\oplus F(d\ot 1+{\rm Id}_S\ot t_0)$ for some $d\in \Der S$ and
some nonzero toral elements $t_0\in W(m;\un{1})$ and $h_0\in S$.

Since $TR(L[\gamma,\mu])=2$, the roots $\gamma$ and $\mu$ span the
dual space of $\overline{T}$. Therefore, $\gamma(h_0\ot 1)\ne 0$ or
$\mu(h_0\ot 1)\ne 0$. It is straightforward to see that $\gamma$
vanishes on all $(\rad\, L(\gamma))_{i\gamma}^{[p]}$ and $[(\rad\,
L(\gamma))_{i\gamma},L_{-i\gamma}]$ with $i\in{\mathbb F}_p^\times$.
Because $\mu\in\Omega''$, this observation in conjunction with
(\ref{gamma-mu}) shows that
$\Psi_{\gamma,\mu}(L_{i\gamma+j\mu})\subset S\ot{\mathcal
O}(m;\un{1})$ for all nonzero $(i,j)\in ({\mathbb F}_p)^2$. There
are in both cases $$x\in \textstyle{\bigcup_{i\in{\mathbb
F}_p^\times}}\,\big((\rad\,
L(\gamma))_{i\gamma}^{[p]}\,\cup\,[(\rad\,
L(\gamma))_{i\gamma},L_{-i\gamma}]\big)\ \,\, \mbox{ and }\ \,\,
h\in [L_\mu,L_{-\mu}]$$ such that $\gamma(x^{[p]})=0,\,$
$\mu(x^{[p]})\ne 0\,$ and $\gamma(h)\ne 0$. But then $2\le
TR(S\ot{\mathcal O}(m;\un{1}))=TR(S)=1$, a contradiction.
\end{proof}
\begin{prop}\label{type4}
Let $\alpha\in\Omega$ and $\beta\in \Gamma(L,T)$ be such that
$L[\alpha,\beta]$ is as in case~4) of Proposition~\ref{types}. Then
$\widetilde{S}\cong S\ot {\mathcal O}(1;\un{1})$, where
$S=H(2;\un{1})^{(2)}$, and $\Psi_{\alpha,\beta}$ can be chosen such
that $\overline{T}:=\Psi_{\alpha,\beta}(T)=F(h_0\ot 1)\oplus F({\rm
Id}_S\ot (1+x_1)\partial_1)$ for some nonzero toral element $h_0\in
S$. Furthermore, $\Omega\ne \Gamma(L,T)$ and the following hold for
$\gamma\in\Gamma(L[\alpha,\beta],\overline{T})$:
\begin{eqnarray*}
\gamma\in\Omega&\Leftrightarrow & \gamma(h_0\ot 1)\ne 0;\\
\gamma\not\in\Omega&\,\Rightarrow\,& \alpha(L_\gamma^{[p]})\ne
0\,\, \mbox{ or }\,\, \beta(L_\gamma^{[p]})\ne 0.
\end{eqnarray*}
\end{prop}
\begin{proof}
By our assumption, $\widetilde{S}=S\ot{\mathcal O}(m;\un{1})$
where $m\ge 1,\,$ $S$ is one of $\mathfrak{sl}(2),\,$
$W(1;\un{1}),\,$ $H(2;\un{1})^{(2)}$. Recall that
$\Psi_{\alpha,\beta}$ takes $T+L(\alpha,\beta)_p$ into ${\rm
Der}(S\ot{\mathcal O}(m;\un{1}))$. Let $$\pi\colon\,{\rm
Der}(S\ot{\mathcal O}(m;\un{1}))\,=\,(\Der S)\ot{\mathcal O}(m;\un
1)\rtimes({\rm Id}_S\ot W(m;\un{1}))\twoheadrightarrow
\,W(m;\un{1})$$ denote the canonical projection. According to
\cite[Thm.~3.2]{PS2}, we can choose $\Psi_{\alpha,\beta}$ such
that
$$\overline{T}:=\Psi_{\alpha,\beta}(T)\,=\,F(h_0\ot 1)\oplus F(d\ot
1+{\rm Id}_S\ot t_0),$$ where $Fh_0$ is a maximal torus of
$S$,\,$d\in\Der S$ and $t_0$ is a toral element of $W(m;\un{1})$.
Moreover, if $t_0\in W(m;\un{1})_{(0)},$ then
$t_0=\sum_{i=1}^m\,s_ix_i\partial_i$, where $s_i\in{\mathbb F}_p$,
and if $t_0\not\in W(m;\un{1})_{(0)}$, then $d=0$ and
$t_0=(1+x_1)\partial_1$.

\smallskip

Our argument is quite long and will be split into two parts, each
part consisting of several intermediate statements. Given a subset
$X$ of $T+L(\alpha,\beta)_p$ we denote by $\overline{X}$ the set
$\{\Psi_{\alpha,\beta}(x)\,|\,\,x\in X\}$. If $\{x_1,\ldots,x_m\}$
is a generating set of the maximal ideal ${\mathcal
O}(m;\un{1})_{(1)}$, then we sometimes invoke the notation
${\mathcal O}(m;\un{1})\,=\,F[x_1,\ldots,x_m]$.

\medskip

\noindent{\em Part~A.\,\,} We first consider the case where $t_0\in
W(m;\un{1})_{(0)}$.

\medskip

\noindent {\it Claim~1.\,\,} $\pi(\overline{H})\subset
W(m;\un{1})_{(0)}.$

\smallskip

\noindent Indeed, suppose the contrary. Then Schue's lemma
\cite[Prop.~1.3.6(1)]{St04} shows that there exists
$\kappa\in\Gamma(L,T)$ with $\kappa(H)\ne 0$ such that
$\pi\big(\overline{[L_\kappa,L_{-\kappa}]}\big)\not\subset
W(m;\un{1})_{(0)}$. Then there is $E\in [L_\kappa, L_{-\kappa}]$
such that $\overline{E}=\overline{E}'+{\rm
Id}_S\ot\pi(\overline{E})$ with $\overline{E}'\in (\Der
S)\ot{\mathcal O}(m;\un{1})$ and
$\pi(\overline{E})\,\equiv\,\sum_{i=1}^m\,a_i\partial_i\,\not\equiv\,
0\ \, \big({\rm mod}\ W(m;\un{1})_{(0)}\big)$ for some $a_i\in F$.
No generality will be lost by assuming that $a_1\ne 0$. Then
$$
0\,=\,\big[t_0,\pi(\overline{E})\big]\,\,\equiv\,\,\textstyle{\sum_{i=1}^m}\,a_is_i\partial_i\quad\
\ \big({\rm mod}\ W(m;\un{1})_{(0)}\big),$$ forcing $s_1=0$. But
then $h_0\ot x_1^{p-1}\in\overline{H}$ and $$\big(\ad
\overline{E}\big)^{p-1}(h_0\ot x_1^{p-1})\in F^\times(h_0\ot
1)+S\ot{\mathcal O}(m;\un{1})_{(1)},$$ which implies that
$[L_\kappa,L_{-\kappa}]^3\not\subset{\rm nil}\, H_p$. As this
contradicts Lemma~\ref{lemA1}, the claim follows. \checkmark

\medskip

\noindent {\it Claim~2.\,\,} {\it There exists
$\nu\in\Gamma(L[\alpha,\beta],\overline{T})$ with
$\pi(\overline{L}_\nu)\not\subset W(m;\un{1})_{(0)}$ and $\nu(h_0\ot
1)=0$.}

\smallskip

\noindent Indeed, $\widetilde{S}$ is derivation simple and
$\pi(\overline{T}+\overline{H})\subset W(m;\un{1})_{(0)}$ by our
general assumption in this part and Claim~1. Hence there is
$\nu\in\Gamma(L[\alpha,\beta],\overline{T})$ with
$\pi(\overline{L}_\nu)\not\subset W(m;\un{1})_{(0)}$. Since
$\pi([h_0\ot 1,\overline{L}_\nu])=0$, it must be that $\nu(h_0\ot
1)=0$. \checkmark

\medskip

\noindent {\it Claim~3.} {\it If
$\gamma\in\Gamma(L[\alpha,\beta],\overline{T})$, then
$\gamma\in\Omega\,\Leftrightarrow\, \gamma(h_0\ot 1)\ne 0$.}

\smallskip

\noindent Let $\gamma$ be any root in
$\Gamma(L[\alpha,\beta],\overline{T})$ with $\gamma(h_0\ot 1)=0$. As
$h_0\ot 1\in\overline{T}$ is a nonzero toral element, $\gamma\in
{\mathbb F}_p^\times\,\nu$, where $\nu$ is the root from Claim~2.
Hence there is $\overline{E}\in \overline{L}_{i\gamma}$ for some
$i\in{\mathbb F}_p^\times$, such that $\pi(\overline{E})\not\in
W(m;\un{1})_{(0)}$. As before, we have that
$\pi(\overline{E})\,\equiv\,\sum_{i=1}^m\,a_i\partial_i\,\not\equiv\,
0\ \, \big({\rm mod}\ W(m;\un{1})_{(0)}\big)$, and it can be assumed
that $a_1\ne 0$. Then $h_0\ot x_1\in \widetilde{S}_{-i\gamma}$. Note
that $h_0\ot{\mathcal O}(m;\un{1})$ is an abelian ideal of the
centralizer of $h_0\ot 1$ in $\Der \widetilde{S}$. Consequently,
$h_0\ot x_1\in\, {\rm rad}(L[\alpha,\beta](\gamma))_{-i\gamma}\,$
and
$$a_1h_0\ot1\,\equiv\,\big[\overline{E},h_0\ot
x_1\big]\qquad\ \big({\rm mod}\ S\ot {\mathcal
O}(m;\un{1})_{(1)}\big).
$$
It follows that $[L_{i\gamma},(\rad\, L(\gamma))_{-i\gamma}]$
contains an element which is not $p$-nilpotent in $L_p$. Then
$\gamma\not\in\Omega$ by Lemma~\ref{omega'}. Since
$\alpha\in\Omega$, these considerations show that $\alpha(h_0\ot
1)\ne 0$. As a consequence,
$$i\alpha+j\gamma\in\Omega\Leftrightarrow (i\alpha+j\gamma)(H^3)\ne 0\Leftrightarrow
i\in\mathbb{F}_p^\times \Leftrightarrow (i\alpha+j\gamma)(h_0\ot
1)\ne 0,$$ hence the claim. \checkmark

\medskip

\noindent {\it Claim~4.\,\,} {\it The Lie algebra
$\pi(\overline{H})^3$ consists of $p$-nilpotent elements of
$W(m;\un{1})$.}

\smallskip

\noindent Otherwise, there is $y\in \overline{H}^3$ with
$y^{[p]^e}\in \overline{T}\setminus F(h_0\ot 1)$, so that
$y^{[p]^e}=b_1(h_0\ot 1)+b_2(d\ot 1+{\rm Id}_S\ot t_0)$ for some
$b_1\in F$ and $b_2\in F^\times$. Let
$\nu\in\Gamma(L[\alpha,\beta],\overline{T})$ be as in Claim~2. Then
$\nu(h_0\ot 1)= 0$ and $\nu(d\ot 1+{\rm Id}_S\ot t_0)\ne 0$, forcing
$\nu(y^{[p]^e})\ne 0$. It follows that $\nu\in\Omega$. This
contradicts Claim~3, however. \checkmark

\medskip

\noindent {\it Claim~5.\,\,} $d\in Fh_0$.

\smallskip

\noindent Claim~1 in conjunction with our standing hypothesis in
this part shows that there is a Lie algebra homomorphism
$$\Psi\colon\,(\Der S)\ot{\mathcal
O}(m;\un{1})+(\overline{H}+\overline{T})\longrightarrow\,\Der S$$
whose kernel is spanned by $(\Der S)\ot {\mathcal
O}(m;\un{1})_{(1)}$ and those elements of
$\overline{H}+\overline{T}$ which map $(\Der S)\ot{\mathcal
O}(m;\un{1})$ into $(\Der S)\ot{\mathcal O}(m;\un{1})_{(1)}$.
Suppose $d\not\in Fh_0$. Then $\Psi(\overline{T})=Fh_0\oplus Fd$.
Since $d$ is a semisimple derivation of $S$, it follows that
$S=H(2;{\un 1})^{(2)}$ and $\Psi(\overline{T})$ is a torus of
maximal dimension in $\Der S$. Since every such torus is
self-centralizing in $\Der S$, by \cite[(III.1)]{St92}, it must be
that $\overline{H}\subset \overline{T}+\ker \Psi$. Note that
$$
(\overline{H}+\overline{T})\,\subset\,(\Der S)\ot{\mathcal
O}(m;\un{1})+F({\rm Id}_S\ot t_0)+{\rm Id}_S\ot\pi(\overline{H})
$$
and $F({\rm Id}_S\ot t_0)+{\rm
Id}_S\ot\pi(\overline{H})\subset\ker\Psi$ by our assumption on
$t_0$ and Claim~1. Hence
\begin{eqnarray*}
\overline{H}&\subset&(\overline{T}+\ker \Psi)\cap
\overline{H}\,\subset\,(\ker
\Psi)\cap(\overline{H}+\overline{T})+\overline{T}\\
&\subset& (\Der S)\ot{\mathcal O}(m;\un{1})_{(1)}+F({\rm Id}_S\ot
t_0)+{\rm Id}_S\ot\pi(\overline{H})+\overline{T},
\end{eqnarray*}
forcing $\overline{H}^3\subset\, (\Der S)\ot{\mathcal
O}(m;\un{1})_{(1)}+{\rm Id}_S\ot\pi(\overline{H})^3$. In view of
Claim~4 the Lie algebra on the right acts nilpotently on
$S\ot{\mathcal O}(m;\un{1})$. But then $\overline{H}^3$ acts
nilpotently on $L[\alpha,\beta]$, a contradiction. \checkmark

\smallskip

As a consequence, $\overline{H}\cap(S\ot{\mathcal
O}(m;\un{1}))\,=\,{\mathfrak c}_S(h_0)\ot {\rm Ann}_{{\mathcal
O}(m;\un{1})}\,(t_0)$ and we may take $d=0$.

\medskip

\noindent {\it Claim~6.\,\,} {\it Let $\nu$ be as in Claim~2. Then}
$$
\overline{H}\cap\widetilde{S}\,\subset\,
\Psi_{\alpha,\beta}\big([(\rad\,L(\nu))_{-\nu},L_\nu]\big)
+\overline{H}\cap\big(S\ot{\mathcal O}(m;\un{1})_{(1)}\big).
$$

\smallskip

\noindent By definition, there is
$\overline{E}\in\overline{L}_{\nu}$ such that
$$\pi(\overline{E})\,\equiv\,\textstyle{\sum_{i=1}^m}\,a_i\partial_i\,\not\equiv\,
0\ \, \big({\rm mod}\ W(m;\un{1})_{(0)}\big), \qquad a_1\ne 0.$$ We
have shown in the course of the proof of Claim~3 that  ${\mathfrak
c}_S(h_0)\ot x_1\subset \overline{\rad\, L(\nu)}_{-\nu}$. Then
${\mathfrak c}_S(h_0)\ot
F\subset\,\big[\overline{E},\widetilde{S}_{-\nu}\big]+\overline{H}\cap
\big(\widetilde{S}\cap{\mathcal O}(m;\un{1})_{(1)}\big)$. As a
consequence,
\begin{eqnarray*}
\overline{H}\cap\widetilde{S}&=&{\mathfrak c}_S(h_0)\ot{\rm
Ann}_{{\mathcal O}(m;\un{1})}\,(t_0)\subset\,{\mathfrak c}_S(h_0)\ot
F+{\mathfrak c}_S(h_0)\ot{\mathcal O}(m;\un{1})_{(1)}\\
&\subset&\big[\overline{L}_{\nu},\overline{(\rad\,L(\nu))}_{-\nu}\big]+
\overline{H}\cap\big(S\ot {\mathcal O}(m;\un{1})_{(1)}\big).\ \
\checkmark
\end{eqnarray*}

\medskip

\noindent {\it Claim~7.\,\,} {\it If $\nu$ is as in Claim~2, then
$\nu(H)=0$.}

\smallskip

\noindent As $S\ot F$ is $\overline{T}$-stable and $S$ is not
nilpotent, there is $\mu\in\Gamma(\widetilde{S},\overline{T})$ with
$(S\ot F)_\mu\ne (0)$. Then $\mu({\rm Id}_S\ot t_0)=0$ and hence
$\mu(h_0\ot 1)\ne 0$. It follows that
$$
L[\alpha,\beta](\mu)\,\subset\,S\ot{\mathcal
O}(m;\un{1})+\overline{H}\,\subset\,(\Der S)\ot {\mathcal
O}(m;\un{1})+{\rm Id}_S\ot W(m;\un{1})_{(0)}.$$ Let
$\Phi\colon\,L[\alpha,\beta](\mu)\rightarrow\,\Der S$ denote the
natural $\overline{T}$-equivariant Lie algebra homomorphism with
$\ker\Phi\,=\,L[\alpha,\beta](\mu)\,\cap\big((\Der S)\ot {\mathcal
O}(m;\un{1})_{(1)}+{\rm Id}_S\ot W(m;\un{1})_{(0)}\big)$ and
$S\subset\mathrm{im}\,\Phi$. Then \cite[Thms~1.2.8 \& 1.3.11]{St04}
show that
$$TR(\ker\Phi)\le TR(L[\alpha,\beta](\mu))-TR(S)\le
TR(L(\mu))-TR(S)\le 1-TR(S)\le 0,$$ implying that $\ker\Phi$ is a
nilpotent ideal of $L[\alpha,\beta](\mu)$. As
$\Phi(L[\alpha,\beta](\mu))$ contains $S$, it is semisimple, hence
isomorphic to $L[\mu]$. Note that $\mu\in\Omega$ by Claim~3. As
$L[\mu]\ne (0)$, Theorem~\ref{1sec} says that $p=5$ and
$\Phi(L[\alpha,\beta](\mu))\cong H(2;\un{1})^{(2)}\oplus
F(1+x_1)^4\partial_2$. In particular, $\mu$ is Hamiltonian. Observe
that
\begin{equation}\label{ha}
\big(\ad
(1+x_1)^4\partial_2)\big)^2(D_H((1+x_1)^3x_2^3)\,=\,D_H((1+x_1)x_2).
\end{equation}
By (the proof of) Lemma~\ref{ham3}, we may assume that
$h_0=D_H((1+x_1)x_2)$. Then (\ref{ha}) shows that there exists
${\mathcal D}\in\Phi(\overline{H})$ such that $\big[{\mathcal
D},\big[{\mathcal D},{\mathfrak c}_{S}(h_0)\big]\big]\not\subset{\rm
nil}\,{\mathfrak c}_S(h_0)$.

Note that ${\rm nil}\,{\mathfrak c}_S(h_0)$ has codimension $1$ in
${\mathfrak c}_S(h_0)$. As $\ker\Phi$ acts nilpotently on
$L[\alpha,\beta](\mu)$, there is $\widetilde{\mathcal
D}\in\overline{H}$ with $\mu\big(\big[{\widetilde{\mathcal
D}},\big[\widetilde{{\mathcal
D}},\widetilde{S}\cap\overline{H}\big]\big]\big)\ne 0$. Since
$\overline{H}\cap\big(S\ot{\mathcal O}(m;\un{1})_{(1)}\big)$ is an
ideal of $\overline{H}$, Claim~6 entails that
$\Psi_{\alpha,\beta}\big((\rad\,L(\nu))_{-\nu},L_\nu]\big)\cap
\overline{H}^3$ does not consist of $p$-nilpotent elements of $L_p$.
In view of Lemma~\ref{omega'}(1), this yields that $\nu(H)=0$.
\checkmark

\smallskip

Since $t_0\in W(m;\un{1})_{(0)}$, the $2$-section $L[\alpha,\beta]$
is semisimple (not just $T$-semisimple), and $\widetilde{S}$ is the
unique minimal ideal of $L[\alpha,\beta]$. On the other hand,
applying Proposition~\ref{pro} with ${\mathfrak t}=T\cap\ker \nu$
shows that the unique minimal ideal of $L[\alpha,\beta]$ is a simple
Lie algebra (notice that ${\mathfrak c}_L({\mathfrak t})=L(\nu)$ is
nilpotent by the Engel--Jacobson theorem). But then $m=0$, a
contradiction. This means that the case where $t_0\in
W(m;\un{1})_{(0)}$ cannot occur.

\medskip

\noindent {\em Part~B}.\,\, Thus, we may assume that $t_0\not\in
W(m;\un{1})_{(0)}$. Because of \cite[Thm.~3.2]{PS2} it can be
assumed further that $\overline{T}=F(h_0\ot 1)\oplus F({\rm Id}_S\ot
(1+x_1)\partial_1)$. Then $\overline{H}\cap\widetilde{S}={\mathfrak
c}_{S}(h_0)\ot F[x_2,\ldots,x_m]$. Since $\alpha$ and $\beta$ are
${\mathbb F}_p$-independent, there exists  $\lambda\in {\mathbb
F}_p\alpha+{\mathbb F}_p\beta$ such that $\lambda(h_0\ot 1)=0$ and
$\lambda({\rm Id}_S\ot (1+x_1)\partial_1)=1$. Note that
\begin{equation}\label{non-nilp}
Fh_0\ot (1+x_1)^i\subset\,\widetilde{S}_{i\lambda}\subset\,
\overline{(\rad\,L(\lambda))}_{i\lambda}\qquad\quad\ \,
\forall\,i\in{\mathbb F}_p^\times.
\end{equation}
Hence $(\rad\,L(\lambda))_{i\lambda}$ contains nonnilpotent elements
of $L_p$ for all $i\in{\mathbb F}_p^\times$. Lemma~\ref{omega'}(b)
yields $\lambda\not\in\Omega$. Since $S\ot F$ is
$\overline{T}$-stable and not nilpotent, there is
$\kappa\in\Gamma(L[\alpha,\beta],\overline{T})$ with $(S\ot
F)_\kappa\ne (0)$. As $\kappa({\rm Id}_S\ot (1+x_1)\partial_1)=0$,
it must be that $\kappa(h_0\ot 1)\ne 0$.

\medskip

\noindent {\it Claim~1.\,\,} {\it If
$\gamma\in\Gamma(L[\alpha,\beta],\overline{T})$, then
$\gamma\in\Omega\,\Leftrightarrow\, \gamma(h_0\ot 1)\ne 0$.}

\smallskip

\noindent As $\alpha\in\Omega$ and $\lambda\not\in\Omega$, one has
$i\alpha+j\lambda\in\Omega$ for all $i\in{\mathbb F}_p^\times$ and
$j\in{\mathbb F}_p$. So
$$i\alpha+j\lambda\in\Omega\,\Leftrightarrow\,i\ne 0\,\Leftrightarrow\,
(i\alpha+j\lambda)(h_0\ot 1)\ne 0\qquad\quad \ \forall\,i,j \in
{\mathbb F}_p.$$ Since $\alpha$ and $\lambda$ are ${\mathbb
F}_p$-independent, their ${\mathbb F}_p$-span contains
$\Gamma(L[\alpha,\beta],\overline{T})$.  \checkmark

\smallskip

It follows from Claim~1 and (\ref{non-nilp}) that
$\Gamma(L[\alpha,\beta],\overline{T})\setminus\Omega\,=\, \,{\mathbb
F}_p^\times\lambda$ and $L_\gamma$ contains nonnilpotent elements of
$L_p$ for all $\gamma\in
\Gamma(L[\alpha,\beta],\overline{T})\setminus\Omega$. Thus, it
remains to show that $m=1$.

\smallskip

\noindent {\it Claim~2.\,\,} {\it The subspace $\sum_{j=2}^m S\ot
x_j{\mathcal O}(m;\un{1})$ is $\overline{H}$-invariant.}

\medskip

\noindent Note that $L[\alpha,\beta](\kappa)\,=\,\overline{H}+S\ot
F[x_2,\ldots, x_m]$. In particular, $\kappa$ is nonsolvable. Let
$\psi\colon\, L[\alpha,\beta](\kappa)\twoheadrightarrow\,L[\kappa]$
denote the canonical homomorphism. By Theorem~\ref{1sec}, the Lie
algebra $L[\kappa]^{(1)}$ is simple. As the ideal $S\ot
F[x_2,\ldots, x_m]$ is perfect, $\psi$ maps it onto
$L[\kappa]^{(1)}$. As a consequence, $S\ot F[x_2,\ldots,
x_m]_{(1)}=\,\ker\psi\cap \big(S\ot F[x_2,\ldots, x_m]\big)$,
showing that $S\ot F[x_2,\ldots, x_m]_{(1)}$ is
$\overline{H}$-invariant. \checkmark

\medskip

\noindent {\it Claim~3.\,\,} {\it $S\cong H(2;\un{1})^{(2)}$ and
$[D,[D,h]]$ acts nonnilpotently on $\widetilde{S}$ for some
$D\in\overline{H}$ and $h\in\overline{H}\cap\widetilde{S}$.}

\smallskip

\noindent We have seen in the proof of Claim~2 that
$$L[\kappa]\,=\psi\big(L[\alpha,\beta](\kappa)\big)\,\cong \,
L[\alpha,\beta](\kappa)/\rad(L[\alpha,\beta](\kappa))\cong
S+H/\big(H\cap\rad\,L(\kappa)\big).$$ Our choice of $\kappa$ and
Claim~1 imply that $\kappa\in\Omega$. So Theorem~\ref{1sec} implies
that $L[\kappa]\cong H(2;\un{1})^{(2)}\oplus F\bar{D}$ and there
exists $\widetilde{h}\in \sum_{i\in{\mathbb
F}_p^\times}\,[L_{i\kappa},L_{-i\kappa}]$ such that
$[\bar{D},[\bar{D},\Psi_\kappa(\widetilde{h})]]$ acts nonnilpotently
on $L[\kappa]$. Pick $D\in \psi^{-1}(\bar{D})\cap\overline{H}$ and
set $h:=\Psi_{\alpha,\beta}(\widetilde{h})$. Standard toral rank
considerations show that $\ker\psi$ acts nilpotently of
$L[\alpha,\beta](\kappa)$ (see the proof of Claim~7 in Part~A for a
similar argument). In light of the preceding remark this implies
that $\kappa\big([D,[D,h]]\big)\ne 0$. \checkmark

\medskip

\noindent {\it Claim~4.\,\,} $m=1$.

\smallskip

\noindent We first note that $L[\alpha,\beta]\subset
\overline{L(\lambda)}+(\Der S)\ot {\mathcal O}(m;\un{1})$. If all
derivations from the set $\bigcup_{i\in{\mathbb F}_p^\times}\,{\rm
Id}_S\ot\pi(\overline{L}_{i\lambda})$ preserve the ideal
$I:=\sum_{j=2}^m\,S\ot x_j{\mathcal O}(m;\un{1})$ of $(\Der S)\ot
{\mathcal O}(m;\un{1})$, then Claim~2 entails that $I$ is a
nilpotent $\overline{T}$-stable ideal of $L[\alpha,\beta]$. Since
$L[\alpha,\beta]$ is $T$-semisimple, this would force $m=1$.

So assume for a contradiction that there exists $E\in L_{k\lambda}$
for some $k\in{\mathbb F}_p^\times$ such that ${\rm
Id}_S\ot\pi(\overline{E})$ does not preserve $I$. Since
$\pi(\overline{E})$ is an eigenvector for $(1+x_1)\partial_1$ with
eigenvalue $k\ne 0$, it has the form
$$
\pi(\overline{E})\,=\,f_1(x_2,\ldots,
x_m)(1+x_1)^{k+1}\partial_1+{\textstyle
\sum_{j=2}^m}\,f_j(x_2,\ldots, x_m)(1+x_1)^k\partial_j$$ for some
$f_1,\ldots, f_m\in F[x_2,\ldots, x_m]$. As $\pi(\overline{E})$ does
not stabilize $I$, it must be that $f_{j_0}(0)\ne 0$ for some
$j_0\ge 2$. After renumbering we may assume that $j_0=2$. Since
${\mathfrak c}_S(h_0)\ot(1+x_1)^{p-k}x_2\subset
\widetilde{S}_{-k\lambda}\subset
(\overline{\rad\,L(\lambda)})_{-k\lambda}$, we have that
\begin{eqnarray*}
{\mathfrak c}_S(h_0)\ot F&\subset&
\big[\overline{E},\widetilde{S}_{-k\lambda}\big]+\Big(S\ot
F[x_1,\ldots,x_m]_{(1)}\Big)
\cap\overline{H}\\
&=&\big[\overline{E},\widetilde{S}_{-k\lambda}\big]+\mathfrak{c}_{S}(h_0)\ot
F[x_1,\ldots,x_m]_{(1)}.
\end{eqnarray*}
From this it follows that
$$
\overline{H}\cap\widetilde{S}={\mathfrak c}_S(h_0)\ot
F[x_1,\ldots,x_m]\subset\,\big[\overline{L}_{k\lambda},
(\overline{\rad\,L(\lambda)})_{-k\lambda}\big]+{\mathfrak
c}_{S}(h_0)\ot F[x_2,\ldots,x_m]_{(1)}.
$$
The subspace $I\cap \overline{H}={\mathfrak c}_{S}(h_0)\ot
F[x_2,\ldots,x_m]_{(1)}$ is $\overline{H}$-invariant by Claim~2 and
acts nilpotently on $L[\alpha,\beta](\kappa)$. These observations in
conjunction with Claim~3 imply that $(\ad
D)^2\big(\big[\overline{L}_{k\lambda},(\overline{\rad\,L(\lambda)})_{-k\lambda}
\big]\big)\subset\,
\overline{H}^3\cap\big[\overline{L}_{k\lambda},(\overline{\rad\,L(\lambda)})_{-k\lambda}
\big]$ does not consist of nilpotent derivations of $\widetilde{S}$.
But then $\lambda(H)=0$ by Lemma~\ref{omega'}(a).

We now set ${\mathfrak t}:=T\cap\ker \lambda$. Since
$L(\lambda)={\mathfrak c}_L({\mathfrak t})$ is nilpotent by the
Engel--Jacobson theorem, Proposition~\ref{pro} says that $L(
\alpha,\beta)/\rad\,L(\alpha,\beta)$ has a unique minimal ideal,
$S'$ say, which is a simple Lie algebra. Then $S'$ must be the image
of $\widetilde{S}=S\ot {\mathcal O}(m;\un{1})$ under the natural
homomorphism $\phi\colon\,L[\alpha,\beta]\twoheadrightarrow
L(\alpha,\beta)/\rad\,L(\alpha,\beta)$. As a consequence,
$\ker\phi\cap\widetilde{S}$ coincides with the radical of
$\widetilde{S}$. As the latter equals $S\ot{\mathcal
O}(m;\un{1})_{(1)}$, we derive that $S\ot{\mathcal
O}(m;\un{1})_{(1)}=\ker\phi\cap\widetilde{S}$ is an ideal of
$L[\alpha,\beta]$. On the other hand, $\pi(\overline{E})\not\in
W(m;\un{1})_{(0)}$. This shows that our present assumption is false
and $m=1$. \checkmark

The proof of the proposition is now complete.
\end{proof}
\begin{cor}\label{type4rad}
Let $\alpha\in\Omega,$ $\beta\in\Gamma(L,T)$ and suppose
$L[\alpha,\beta]$ is as in case~4) of Proposition~\ref{types}. Then
$\sum_{i\in{\mathbb
F}_p^\times}\,(\rad\,L(\gamma))_{i\gamma}\subset\,\rad_T\,L[\alpha,\beta]$
for all $\gamma\in\Omega\cap ({\mathbb F_p}\alpha+{\mathbb
F}_p\beta)$.
\end{cor}
\begin{proof}
Pick $\gamma\in\Omega\cap ({\mathbb F_p}\alpha+{\mathbb F}_p\beta)$
and view  it as a $\overline{T}$-root of $L[\alpha,\beta]$. In the
present case
$L[\alpha,\beta](\gamma)=\overline{H}+\widetilde{S}(\gamma)$ and
$\widetilde{S}=H(2;\un{1})^{(2)}\ot{\mathcal O}(1;\un{1})$; see
Proposition~\ref{type4}. Furthermore, in the notation of
Proposition~\ref{type4} we have that $\gamma=i\kappa+j\lambda$ for
some $i\in{\mathbb F}_p^\times$ and $j\in{\mathbb F}_p$, where
$\kappa,\lambda\in\overline{T}^*$ are such that $\kappa(h_0\ot
1)=r\in{\mathbb F}_p^\times$,\, $\kappa({\rm
Id}_S\ot(1+x_1)\partial_1)=0$,\, $\lambda(h_0\ot 1)=0$ and
$\lambda({\rm Id}_S\ot(1+x_1)\partial_1)=1$.  Let $S_\ell$ denote
the $\ell$-eigenspace of ${\rm ad}_S\, h_0$. Then
$$\widetilde{S}(\gamma)\,=\,\textstyle{\bigoplus_{k\in{\mathbb F}_p}}\,
S_{kir}\ot(1+x_1)^{kj}\,\cong\,\textstyle{\bigoplus_{k\in{\mathbb
F}_p}}\,S_{kir}=H(2;\un{1})^{(2)}$$  as Lie algebras. Hence
$\rad(L[\alpha,\beta](\gamma))\,=\,\rad\,\big(\overline{H}+
\widetilde{S}(\gamma)\big)\subset \overline{H}$. The result follows.
\end{proof}

We are now in a position to prove our first result on the global
structure of $L$.
\begin{theo}\label{rootspaces}
If $\alpha\in\Omega$, then $\alpha$ is Hamiltonian, $\dim
L_\alpha=5$, and $\rad\,L(\alpha)\subset H$.
\end{theo}
\begin{proof}
For $\gamma\in\Gamma(L,T)$ put $R_\gamma:=(\rad\,L(\gamma))_\gamma$.
Let $\mu\in\Omega$ be such that $\rad\, L(\mu)\not\subset H$. By
Theorem~\ref{1sec}, the radical of $L(\mu)$ is $T$-stable. Hence
there is $a\in{\mathbb F}_p^\times$ such that
$(\rad\,L(\mu))_{a\mu}\ne (0)$. Put $\nu:=a\mu$ and note that
$\nu\in\Omega$. For $k\in\Z_+$ define
$$I_0\,:=\,R_\nu,\qquad \ \, I_k\,:=\,\sum_{\gamma_1,\ldots,\gamma_k}
\,[L_{\gamma_1},[\cdots [L_{\gamma_k},R_\nu]\cdots]],\qquad \ \,
I:=\,\sum_{k\ge 0}\,I_k.$$ Clearly, $I$ is an ideal of $L$
containing $R_\nu$. We intend to show that $I\subsetneq L$. As a
first step we are going to use induction on $k$ to prove the
following:

\medskip

\noindent {\it Claim.\,\,} {\it If
$\nu+\gamma_1+\cdots+\gamma_k\in\Omega$, then $[L_{\gamma_1},[\cdots
[L_{\gamma_k},R_\nu]\cdots]]\subset
R_{\nu+\gamma_1+\cdots+\gamma_k}$.}

\smallskip

\noindent The claim is obviously true for $k=0$, and it also holds
for $k=1$ thanks to Corollaries~\ref{not4} and \ref{type4rad}.
Suppose it is true for all $k<n$ and let
$\gamma_1,\ldots,\gamma_n\in\Gamma(L,T)$ be such that
$\nu+\gamma_1+\cdots+\gamma_n\in\Omega$. If $\nu+\gamma_i\in\Omega$
or $\nu+\gamma_i\not\in\Gamma(L,T)$ for some $i\le n$, then applying
Corollaries~\ref{not4} and \ref{type4rad} gives
\begin{eqnarray*}
[L_{\gamma_1},[\cdots[L_{\gamma_n},R_\nu]\cdots]]&\subset&[L_{\gamma_1},
[\cdots[\widehat{L_{\gamma_i}}\cdots[L_{\gamma_n},[L_{\gamma_i},R_\nu]]\cdots]\cdots]]+I_{n-1}\\
&\subset&[L_{\gamma_1},[\cdots[\widehat{L_{\gamma_i}}
\cdots[L_{\gamma_n},R_{\nu+\gamma_i}]\cdots]\cdots]]+ I_{n-1}.
\end{eqnarray*}
In this case the claim holds by our induction hypothesis. So assume
from now that $\nu+\gamma_i\in\Gamma(L,T)\setminus\Omega$ for all
$i\le n$. We may also assume that
$\tilde{\nu}:=\nu+\gamma_1+\ldots+\gamma_n$ is not solvable, for
otherwise we are done. According to Lemma~\ref{lemA3} there is
$\kappa\in\Gamma(L,T)$ such that
$L[\tilde{\nu},\kappa]\cong{\mathcal M}(1,1)$. Moreover, it follows
from \cite[Lemmas~4.1 \& 4.4]{P94} that the radical of every
$1$-section $L[\tilde{\nu},\kappa](\delta)$ is contained in
$\Psi_{\tilde{\nu},\kappa}(T)$ and
\begin{equation}\label{me}
({\mathbb F}_p\tilde{\nu}+{\mathbb
F}_p\kappa)\setminus\{0\}\,\subset\, \Omega.
\end{equation}
Take an arbitrary $\kappa'\in({\mathbb F}_p\tilde{\nu}+{\mathbb
F}_p\kappa)\setminus {\mathbb F}_p\tilde{\nu}$. It follows from
(\ref{me}) that $\tilde{\nu}+{\mathbb
F}_p\kappa'\subset\Gamma(L,T)$. Note that the rule
$$\gamma\asymp\gamma'\,\Leftrightarrow\,(\gamma-\gamma')_{|H^3}=0$$
defines an equivalence relation on the set of all $F$-valued
functions on $H$. Since $\gamma_i\asymp -\nu$ for all $i\le n$, we
have that $\tilde{\nu}\asymp (1-n)\nu$. If $\nu+\kappa'\asymp 0$,
then $\tilde{\nu}+(1-n)\kappa'\not\in\Omega$. As
$\tilde{\nu}+(1-n)\kappa'\ne 0$ by our choice of $\kappa'$, this is
not true; see (\ref{me}). Thus, $\nu+\kappa'\not\asymp 0$, showing
that $\nu+\kappa'\in\Omega$ whenever $\nu+\kappa'\in\Gamma(L,T)$.
But then
 $[R_\nu,L_{\kappa'}]\subset
R_{\nu+\kappa'}$ by Corollaries~\ref{not4} and \ref{type4rad}. As
$\nu+\gamma_i\asymp 0$ and $\kappa'\in\Omega$ by (\ref{me}), we also
have that $\nu+(\gamma_i+\kappa')\in\Omega$ whenever
$\nu+(\gamma_i+\kappa')\in\Gamma(L,T)$ for all $i\le n$. So arguing
as above one now obtains that
$[[L_{\gamma_i},L_{\kappa'}],R_\nu]\subset
R_{\nu+\gamma_i+\kappa'}$. This implies that
$$
\Big[[L_{\gamma_1},[\cdots[L_{\gamma_n},R_\nu]\cdots]],
L_{\kappa'}\Big]\subset R_{\tilde{\nu}+\kappa'}\subset
\,\rad\,L(\tilde{\nu},\kappa').
$$
As ${\mathcal M}(1,1)$ is a simple Lie algebra, Schue's lemma
\cite[Prop.~1.3.6(1)]{St04} yields
$$\Big[\Psi_{\tilde{\nu},\kappa'}\big([L_{\gamma_1},[\cdots[L_{\gamma_n},R_\nu]\cdots]]\big),
{\mathcal M}(1,1)\Big]\,=\,(0),
$$
forcing $ [L_{\gamma_1},[\cdots[L_{\gamma_n},R_\nu]\cdots]]\subset
\big(\rad\,
L(\tilde{\nu},\kappa'))_{\tilde{\nu}}\subset(\rad\,L(\tilde{\nu}))_{\tilde{\nu}}\subset
R_{\tilde{\nu}}$. This completes the induction step.

\smallskip

As a consequence, $I_\gamma\subset R_\gamma$ for all
$\gamma\in\Omega$. On the other hand, it follows from
\cite[Lemma~3.8]{P94} that $\Omega$ contains at least one
Hamiltonian root, $\lambda$ say. Then $I_\lambda\ne L_\lambda$,
implying $I\ne L$. Then $I=(0)$, proving that $\rad\,L(\mu)\subset
H$ for all $\mu\in\Omega$. As a consequence, all roots in $\Omega$
are nonsolvable.

Now let $\alpha\in \Omega$. Because $\alpha$ is nonsolvable, it
follows from Theorem~\ref{1sec} that $\alpha$ is Hamiltonian. Since
$\rad\,L(\alpha)\subset H$, this gives $\dim L_\alpha=5$.
\end{proof}
\section{\bf Further reductions}
In this section we are going to prove that no root in
$\Gamma(L,T)$ vanishes on $H^3$. Theorem~\ref{rootspaces} will
play a crucial role in our arguments.
\begin{lemm}\label{mu(H)ne0}
If $\gamma\in\Gamma(L,T)$ does not vanish on $H$, then
$\gamma\in\Omega$.
\end{lemm}
\begin{proof}
Suppose there is $\beta\in\Gamma(L,T)\setminus\Omega$ such that
$\beta(H)\ne 0$. By (\ref{eq0}), there is $\alpha\in\Omega$ such
that $\beta([L_\alpha,L_{-\alpha}])\ne 0$. Then
$[L_\beta,[L_\alpha,L_{-\alpha}]]=L_\beta$, implying that
$\alpha+\beta\in\Gamma(L,T)$ or $-\alpha+\beta\in\Gamma(L,T)$. Since
$\beta\not\in\Omega$ by our assumption, we have that
$\alpha+\beta\in\Omega$ or $-\alpha+\beta\in\Omega$.
Theorem~\ref{rootspaces} then shows that $\{\alpha, \alpha+\beta\}$
or $\{\alpha, -\alpha+\beta\}$ consists of nonsolvable roots. Then
$L[\alpha,\beta]$ cannot be of type~1) or 2) of
Proposition~\ref{types}.

Suppose $L[\alpha,\beta]$ is as in case~3) of
Proposition~\ref{types} and set
$\delta_1:=\alpha,\,\delta_2:=\alpha+\beta$ if
$\alpha+\beta\in\Gamma(L,T)$ and
$\delta_1:=\alpha,\,\delta_2:=\alpha-\beta$ if
$-\alpha+\beta\in\Gamma(L,T)$. In either case, we can find elements
$h_1,\,h_2\in H^3$ such that $\delta_i(h_j)=\delta_{ij}$ for
$i,j\in\{1,2\}$. As a consequence, $\alpha(h_2)=0$ and
$\beta(h_2)\ne 0$. But then $\beta\in\Omega$, a contradiction.

Suppose $L[\alpha,\beta]$ is as in case~4) of
Proposition~\ref{types}. Then Proposition~\ref{type4} applies. As
$\alpha\in\Omega$, Proposition~\ref{type4} says that $\alpha(h_0\ot
1)\ne 0$. This forces
$\Psi_{\alpha,\beta}(L_{\pm\alpha})\subset\widetilde{S}$. Since
$\beta([L_\alpha,L_{-\alpha}])\ne 0$, we now deduce that $\beta$
does not vanish on $\Psi_{\alpha,\beta}(H)\cap\widetilde{S}$. This
forces $\beta(h_0\ot 1)\ne 0$. Applying Proposition~\ref{type4} once
again we obtain $\beta\in\Omega$, a contradiction.

Suppose $L[\alpha,\beta]$ is of type~5) of Proposition~\ref{types}.
Then $\widetilde{S}=H(2;(2,1))^{(2)}$ and $L[\alpha,\beta]\subset
H(2;(2,1))$. In this case $\Psi_{\alpha,\beta}(H)^3\subset
\widetilde{S}$, and it follows from Lemma~\ref{ham1} and Demu{\v
s}kin's description of maximal tori in $H(2;{\un 1})^{(2)}$ that
$\Psi_{\alpha,\beta}(H)\cap\widetilde{S}$ is abelian and ${\rm
nil}\big(\Psi_{\alpha,\beta}(H)\cap\widetilde{S}\big)$ has
codimension $1$ in $\Psi_{\alpha,\beta}(H)\cap\widetilde{S}$; see
\cite[Thm.~7.5.8]{St04} for instance. As $\alpha\in\Omega$, this
means that
$\Psi_{\alpha,\beta}(H)\cap\widetilde{S}\,=\,\Psi_{\alpha,\beta}(H)^3+{\rm
nil}\big(\Psi_{\alpha,\beta}(H)\cap\widetilde{S}\big)$. As a
consequence,
$\gamma\in\Gamma(L[\alpha,\beta],\Psi_{\alpha,\beta}(T))$ is in
$\Omega$ if and only if
$\gamma\big(\Psi_{\alpha,\beta}(H)\cap\widetilde{S}\big)\ne 0$. As
$\alpha\in\Omega$, Theorem~\ref{rootspaces} implies that $\alpha$
does not vanish on
$\big[\Psi_{\alpha,\beta}(L_\alpha),\Psi_{\alpha,\beta}(L_{-\alpha})\big]$.
As $\Psi_{\alpha,\beta}(L_{\pm\alpha})\subset\widetilde{S}$, this
shows that
$$
\Psi_{\alpha,\beta}(H)\cap\widetilde{S}\,=\big[\Psi_{\alpha,\beta}(L_\alpha),
\Psi_{\alpha,\beta}(L_{-\alpha})\big]+{\rm
nil}\big(\Psi_{\alpha,\beta}(H)\cap\widetilde{S}\big).
$$
But then
$\beta\big(\Psi_{\alpha,\beta}(H)\cap\widetilde{S}\big)\ne 0$ by
our choice of $\beta$, implying that $\beta\in\Omega$. Since this
contradicts our choice of $\beta$, we derive that
$L[\alpha,\beta]$ cannot be of type~5).

If $L[\alpha,\beta]$ is as in case~6) of Proposition~\ref{types},
then $({\mathbb F}_p\alpha+{\mathbb
F}_p\beta)\setminus\{0\}\subset\Omega$ by \cite[Lemmas~4.1 \&
4.4]{P94}. So this case cannot occur either, and our proof is
complete.
\end{proof}
\begin{prop}\label{mu(H)=0}
If $\mu\in\Gamma(L,T)$ vanishes on $H$, then $L_\mu$ consists of
$p$-nilpotent elements of $L_p$.
\end{prop}
\begin{proof}
Suppose for a contradiction that there is $\mu\in\Gamma(L,T)$ with
$\mu(H)=0$ such that $\alpha(L_\mu^{[p]})\ne 0$ for some
$\alpha\in\Gamma(L,T)$. It follows from (\ref{eq0}) that every root
is the sum of two roots in $\Omega$. Therefore, we may assume that
$\alpha\in\Omega$.  Since $\alpha$ is nonsolvable by
Theorem~\ref{rootspaces}, there exists $\beta\in\Omega$ such that
$L[\alpha,\beta]\cong{\mathcal M}(1,1)$ and
$\alpha\big([L_{i\alpha},L_{-i\alpha}],[L_\beta,L_{-\beta}]\big)\ne
0$ for some $i\in{\mathbb F}_p^\times$; see Lemma~\ref{lemA3}.
Lemma~\ref{mu(H)ne0} shows that $\beta\in\Omega$.

We now consider the $T$-semisimple $3$-section
$L[\alpha,\beta,\mu]$. Set
$\overline{T}:=\Psi_{\alpha,\beta,\mu}(T),\,$
$\overline{H}:=\Psi_{\alpha,\beta,\mu}(H)$ and
$\widetilde{S}:=\widetilde{S}(\alpha,\beta,\mu)$. Given a Lie
subalgebra $M$ of $L[\alpha,\beta,\mu]$ we denote by $M_{[p]}$ the
$p$-envelope of $M$ in $\Der \widetilde{S}$. Note that the
restricted Lie algebra
$\overline{T}+L[\alpha,\beta,\mu]_{[p]}\subset\Der\widetilde{S}$ is
centerless. As $T$ is a torus of maximal dimension in
$T+L(\alpha,\beta,\mu)_p$, it follows from
\cite[Thm.~1.2.8(4a)]{St04} that $\overline{T}$ is a torus of
maximal dimension in $\overline{T}+L[\alpha,\beta,\mu]_{[p]}$. Let
$J$ be a minimal $T$-invariant ideal of $L[\alpha,\beta,\mu]$. Then
$TR(J)\le TR(L[\alpha,\beta,\mu])\le 3$; see \cite[Thms~1.2.7(1) \&
1.3.11(3)]{St04}.

\smallskip

\noindent (a) Suppose $TR(J)=3$. Then it follows from
\cite[Thm.~1.2.9(3)]{St04} that the restricted Lie algebra
$\big(\overline{T}+L[\alpha,\beta,\mu]_{[p]}\big)/J_{[p]}$ is
$p$-nilpotent. From this it is immediate that $\overline{T}\subset
J_{[p]},\,$ $J=\widetilde{S}$ and
$L[\alpha,\beta,\mu]=\overline{H}+\widetilde{S}$. By Block's
theorem, $\widetilde{S}=S\ot{\mathcal O}(m;\un{1})$, where $S$ is a
simple Lie algebra and $m\in\Z_+$.  Let $\pi$ denote the canonical
projection
$${\rm Der}\big(S\ot{\mathcal O}(m;\un{1})\big)=\,(\Der
S)\ot{\mathcal O}(m;\un{1})\rtimes{\rm Id}_S\ot
W(m;\un{1})\twoheadrightarrow \,W(m;\un{1}).$$ In the present
situation \cite[Thm.~2.6]{PS2} implies that the torus $\overline{T}$
is conjugate under ${\rm Aut}(S\ot{\mathcal O}(m;\un{1}))$ to
$T_0\ot F$ for some torus $T_0$ in $S_p$. Hence we can choose
$\Psi_{\alpha,\beta,\mu}$ such that $\overline{T}=T_0\ot F$. Then
$L[\alpha,\beta,\mu](\alpha)=\overline{H}+S(\alpha)\ot{\mathcal
O}(m;\un{1}).$ Since $\alpha$ is nonsolvable, there is a surjective
homomorphism
$\psi\colon\,L[\alpha,\beta,\mu](\alpha)\twoheadrightarrow
L[\alpha]\ne (0)$. By Theorem~\ref{1sec}, $({\rm im}\,\psi)^{(1)}$
is a simple Lie algebra and the unique minimal ideal of ${\rm
im}\,\psi$. Since $T_0$ is a torus of maximal dimension in $S_p$,
Theorem~\ref{1sec} also applies to the $1$-section $S[\alpha]$. So
it must be that $({\rm im}\,\psi)^{(1)}\cong S[\alpha]^{(1)}$. As a
consequence,
$$\widetilde{S}(\alpha)^{(1)}\cap\ker\psi\,=\,
\big(\rad\, S(\alpha)\cap S(\alpha)^{(1)}\big)\ot
F+S(\alpha)^{(1)}\ot{\mathcal O}(m;\un{1})_{(1)}$$ is
$\overline{H}$-invariant. As $S(\alpha)$ is not solvable, it follows
that $\pi(\overline{H})\subset W(m;\un{1})_{(0)}$. But then
$S\ot{\mathcal O}(m;\un{1})_{(1)}$ is an ideal of
$L[\alpha,\beta,\mu]$. As $L[\alpha,\beta,\mu]$ is $T$-semisimple
and $T=T_0\ot F$, we now obtain that $m=0$ and
$L[\alpha,\beta,\mu]=\overline{H}+\widetilde{S}$.

As a consequence, $\Psi_{\alpha,\beta,\mu}(L_\gamma)\subset
\widetilde{S}$ for all
$\gamma\in\Gamma(L[\alpha,\beta,\mu],\overline{T})$. This implies
that $L[\alpha,\beta]\cong{\mathcal M}(1,1)$ is a homomorphic image
of the $2$-section $\widetilde{S}(\alpha,\beta)$, showing that
$\overline{H}\cap\widetilde{S}$ is a nontriangulable subalgebra of
$\widetilde{S}$. We now set ${\mathfrak
t}:=\Psi_{\alpha,\beta,\mu}(T\cap\ker \mu)$ and ${\mathfrak
h}:=\widetilde{S}(\mu)$. Then $\widetilde{S}$ is simple, ${\mathfrak
t}$ is a torus of dimension at most $2$ in $\widetilde{S}_p$, and
$\overline{H}\cap\widetilde{S}\subset \mathfrak h.$ This inclusion
in conjunction with our assumption on $\mu$ and the Engel--Jacobson
theorem shows that $\mathfrak h$ is a nontriangulable nilpotent
subalgebra of $\widetilde{S}$. But then \cite[Thm.~1(ii)]{P94}
yields $\widetilde{S}\cong{\mathcal M}(1,1)$. As $TR({\mathcal
M}(1,1))=2$ by \cite[Lemma~4.3]{P94}, we reach a contradiction
thereby establishing that $TR(J)\le 2$.

\smallskip

\noindent (b) We now put $T':=\overline{T}\cap J_{[p]}$ and observe
that $$\dim T'\ge
TR\big(J_{[p]},\overline{T}+L[\alpha,\beta,\mu]_{[p]}\big)=TR(J_{[p]})\ne
0;$$ see \cite[Thms~1.2.9 \& 1.2.8(2)]{St04} (one should also keep
in mind that $\overline{T}+L[\alpha,\beta,\mu]_{[p]}$ is
centerless).

Suppose $\mu(T')\ne 0$. Then
$\Psi_{\alpha,\beta,\mu}(L_{i\mu})\subset J$ for all $i\in{\mathbb
F}_p^\times$ and hence $\Psi_{\alpha,\beta,\mu}(L_\alpha)\subset J$
by our choice of $\alpha$. Since $L[\alpha,\beta]\cong{\mathcal
M}(1,1)$ is simple, it follows that
$\Psi_{\alpha,\beta,\mu}(L_{i\alpha+j\beta})\subset J$ for all
nonzero $(i,j)\in{\mathbb F}_p^2$. As a consequence, the
$p$-envelope of $\overline{H}\cap J$ in $J_{[p]}$ contains a torus
of dimension at least $2$. This torus must be smaller than $T'$,
because $\mu$ vanishes on $H$. But then $TR(J)>2$ which is not true.

Thus, $\mu(T')=0$. Then $\alpha(T')\ne 0$ or $\beta(T')\ne 0$.
Relying on the simplicity of $L[\alpha,\beta]\cong {\mathcal
M}(1,1)$ and arguing as before, we derive that
$J(\alpha,\beta)/\rad\,J(\alpha,\beta)\cong{\mathcal M}(1,1).$ As
$\mu(T')=0$, it follows that $\dim T'=TR(J)=2$. By Block's theorem,
$J=J'\ot{\mathcal O}(k;\un{1})$ for some simple Lie algebra $J'$ and
some $k\in\Z_+$. The above shows that $TR(J')=2$. The natural
homomorphism $J\twoheadrightarrow J/J'\ot{\mathcal
O}(k;\un{1})_{(1)}\cong J'$ maps $J(\alpha,\beta)$ onto a subalgebra
$\mathfrak g$ of $J'$ such that ${\mathfrak g}/\rad\,{\mathfrak
g}\cong {\mathcal M}(1,1)$. As $TR(J')=2$, this implies that $J_p$
contains a nonstandard $2$-dimensional torus. Applying
\cite[Thm.~1(ii)]{P94} now yields $J'\cong {\mathcal M}(1,1)$. Since
this holds for every minimal $\overline{T}$-invariant ideal of
$L[\alpha,\beta,\mu]$ and $TR(L[\alpha,\beta,\mu])\le 3$, we may
conclude at this point that the $T$-socle
$\widetilde{S}=\widetilde{S}(\alpha,\beta,\mu)=S\ot{\mathcal
O}(m;\un{1})$ is the unique minimal ideal of $L[\alpha,\beta,\mu]$.

Recall that all derivations of $S={\mathcal M}(1,1)$ are inner; see
\cite[Thm.~7.1.4]{St04} for instance. In this situation
\cite[Thm.~3.2]{PS2} says that $\Psi_{\alpha,\beta,\mu}$ can be
chosen such that $\overline{T}=(T_0\ot 1)+F({\rm Id}_{S}\ot t_0),$
where $T_0$ is a $2$-dimensional torus in $S_p=S$ and $t_0\in
W(m;\un{1})$.  Furthermore, $L[\alpha,\beta,\mu]\,=\,{\mathcal
M}(1,1)\ot{\mathcal O}(m;\un{1})\rtimes{\rm Id}_{S}\ot{\mathfrak d}$
for some Lie subalgebra ${\mathfrak d}$ of $W(m;\un{1})$. Note that
$T'=\overline{T}\cap \widetilde{S}=T_0\ot 1$. Using the simplicity
of $L[\alpha,\beta]$ and arguing as before, we observe that
$\Psi_{\alpha,\beta,\mu}(L_{i\alpha+j\beta})\subset \widetilde{S}$
for all nonzero $(i,j)\in{\mathbb F}_p^\times$. By the choice of
$\beta$, we then have
$\alpha\big([\widetilde{S}_{i\alpha},\widetilde{S}_{-i\alpha}],
[\widetilde{S}_\beta,\widetilde{S}_{-\beta}]\big)\ne 0$ for some
$i\in{\mathbb F}_p^\times$. This means that $T_0$ is a nonstandard
torus in $S={\mathcal M}(1,1)$.

If $t_0\not\in W(m;\un{1})_{(0)}$, then we may assume further that
$t_0=(1+x_1)\partial_1$; see \cite[Thm.~3.2]{PS2}. Choose
$h,\,h'\in\mathfrak{c}_S(T_0)$ such that $[h,h']$ acts
nonnilpotently on $S$. Recall that $\mu(T_0\ot F)=0$. Then
$\mu(\mathrm{Id}_S\ot t_0)\ne 0$ and hence there exists
$r\in\mathbb{F}_p^\times$ such that $h\ot(1+x_1)\in
\widetilde{S}_{r\mu}$ and $h'\ot
(1+x_1)^{p-1}\in\widetilde{S}_{-r\mu}$. Clearly, the element
$$[h\ot(1+x_1),h'\ot(1+x_1)^{p-1}]\in[\widetilde{S}_{r\mu},\widetilde{S}_{-r\mu}]$$
acts nonnilpotently on $\widetilde{S}$.

Suppose $t_0\in W(m;\un{1})_{(0)}$. Since $\widetilde{S}$ is
$\big({\rm Id_S}\ot(Ft_0+{\mathfrak d})\big)$-simple, there is
$r\in{\mathbb F}_p$ such that ${\mathfrak d}_{r\mu}\not\subset
W(m;\un{1})_{(0)}$ (here ${\mathfrak d}_0=\pi(\overline{H})$ is the
centraliser of $t_0$ in $\mathfrak d$). On the other hand, looking
at the $1$-section
$L[\alpha,\beta,\mu](\alpha)=\overline{H}+S(\alpha)\ot {\mathcal
O}(m;\un{1})$ and applying Theorem~\ref{1sec} to $L[\alpha]\ne (0)$
one observes that $\pi(\overline{H})\subset W(m;\un{1})_{(0)}$ (see
part~(a) for a similar argument). So it must be that $t_0\ne 0$ and
$r\in{\mathbb F}_p^\times$.

Let $E\in L_{r\mu}$ be such that
$\pi(\Psi_{\alpha,\beta,\mu}(E))\,\equiv\,\sum_{j=1}^m\,
a_i\partial_i\quad\big({\rm mod}\ W(m;\un{1})_{(0)}\big)$, where not
all $a_j$ are zero. We may assume after renumbering and rescaling
that $a_1=1$. In the present situation \cite[Thm.~3.2]{PS2} says
that $\Psi_{\alpha,\beta,\gamma}$ can be chosen such that
$t_0=\sum_{j=1}^m\,s_i\, x_j\partial_j$ for some $s_j\in{\mathbb
F}_p$. As $\big[t_0,\pi(\Psi_{\alpha,\beta,\mu}(E))\big]$ is a {\it
nonzero} multiple of $\pi(\Psi_{\alpha,\beta,\mu}(E))$, it must be
that $s_1\ne 0$. Therefore, ${\mathfrak c}_S(T_0)\ot x_1\subset
\widetilde{S}_{-r\mu}$, implying that
$\big[\Psi_{\alpha,\beta,\mu}(L_{r\mu}),\widetilde{S}_{-r\mu}\big]$
contains nonilpotent elements of $\widetilde{S}$.

\smallskip

\noindent (c) We have thus shown that there is $r\in{\mathbb
F}_p^\times$ such that $[L_{r\mu},L_{-r\mu}]$ contains nonnilpotent
elements of $L_p$. Therefore, the set
$$\Omega_1:=\,\{\gamma\in\Gamma(L,T)\,|\,\,\gamma([L_{r\mu},L_{-r\mu}])\ne
0\}.$$ is nonempty. By Lemma~\ref{mu(H)ne0}, we have the inclusion
$\Omega_1\subset\Omega$. Also, $\mu\not\in\Omega_1$, because
$\mu(H)=0$. Since $\mu\ne 0$, there is $\gamma\in\Gamma(L,T)$ such
that $\mu(L_\gamma^{[p]})\ne 0$.

Suppose $\gamma\in\Omega$. Since $\mu(L_\gamma^{[p]})\ne 0$, all
elements from $\mu+{\mathbb F}_p\gamma$ are in $\Gamma(L,T)$. Since
$\mu(H)=0$, we then have $\mu+{\mathbb
F}_p^\times\gamma\subset\Omega$. Since all roots in $\Omega$ are
nonsolvable by Theorem~\ref{rootspaces}, the $T$-semisimple
$2$-section $L[\gamma,\mu]$ cannot be as in cases~1), 2) or 3) of
Proposition~\ref{types}. If $L[\gamma,\mu]$ is of type~4), then
Proposition~\ref{type4} implies that
$\Psi_{\gamma,\mu}(L_\gamma)\subset\widetilde{S}$. As
$\mu(L_\gamma^{[p]})\ne 0$, this forces
$\Psi_{\gamma,\mu}(L_{i\mu})\subset\widetilde{S}$ for all
$i\in{\mathbb F}_p^\times$. Since $\mu$ vanishes on $H$, it follows
from the description of $\Psi_{\gamma,\mu}(T)$ given in
Proposition~\ref{type4} that $$\sum_{i\in{\mathbb
F}_p^\times}\Psi_{\gamma,\mu}(L_{i\mu})\subset\, {\mathfrak
c}_{H(2;\un{1})^{(2)}}(h_0)\ot{\mathcal O}(1;\un{1}).$$ As the
subalgebra on the right is abelian and
$\Psi_{\gamma,\mu}(L_{i\gamma})\ne (0)$ for all $i\in{\mathbb
F}_p^\times$, this contradicts our choice of $\mu$. So
$L[\gamma,\mu]$ is not of that type. If $L[\gamma,\mu]$ is as in
cases~5) or~6) of Proposition~\ref{types}, then Corollary~\ref{not4}
shows that no root in
$\Gamma(L[\gamma,\mu],\Psi_{\gamma,\mu}(T))=({\mathbb
F}_p\gamma\oplus{\mathbb F}_p\mu)\setminus\{0\}$ vanishes on
$\Psi_{\gamma,\mu}(H)$. As $\mu(H)=0$, this is false.

Thus, $\gamma\not\in\Omega$. Schue's lemma
\cite[Prop.~1.3.6(1)]{St04} yields
$L_\gamma=\sum_{\delta\in\Omega_1}\,[L_\delta,L_{\gamma-\delta}]$.
If $x_1\ldots,x_d\in L_\gamma$, then
$$\big(\textstyle{\sum_{j=1}^d}\,x_j\big)^{[p]}\equiv\,
\textstyle{\sum_{j=1}^d}\,x_j^{[p]}\quad\,\,({\rm mod}\,\,H),$$ by
Jacobson's formula. Note that the set
$H\cup\big(\bigcup_{\,\delta\in\Omega_1,\,\,k\ge 1}
\,\,[L_\delta,L_{-\delta}]^{[p]^k}\big)$ is weakly closed. Since
$\mu$ vanishes on $H$, the Engel--Jacobson theorem implies that
there is $\kappa\in\Omega_1$ such that
$\mu\big([L_\kappa,L_{\gamma-\kappa}]^{[p]}\big)\ne 0$. Note that
$\kappa$ and $\gamma-\kappa$ are both in $\Omega$, hence
$\Psi_{\gamma,\kappa,\mu}(L_\kappa)\ne (0)$ and
$\Psi_{\gamma,\kappa,\mu}(L_{\gamma-\kappa})\ne (0)$ by
Theorem~\ref{rootspaces}. Let
$\widetilde{S}=\widetilde{S}(\gamma,\kappa,\mu)$ and let $J$ be any
minimal ideal of $L[\gamma,\kappa,\mu]$. Put
$T_1:=\Psi_{\gamma,\kappa,\mu}(T)\cap J_{[p]}$, where $J_{[p]}$ is
the $p$-envelope of $J$ in $\Der \widetilde{S}$. Since $J_{[p]}$ is
centerless, it follows from \cite[Thm.~1.2.8(a)]{St04} that $T_1$ is
a torus of maximal dimension in $J_{[p]}$.

Suppose $\mu(T_1)=0$. Then either $\kappa(T_1)\ne 0$ or
$(\gamma-\kappa)(T_1)\ne 0$, for $T_1\ne (0)$. In any event,
$\Psi_{\gamma,\kappa,\mu}([L_\kappa,L_{\gamma-\kappa}])\subset J$
and therefore $\mu(J_\gamma^{[p]})\ne 0$. But then $\mu(T_1)\ne 0$,
a contradiction. Thus, $\mu(T_1)\ne 0$, forcing $\sum_{i\in{\mathbb
F}_p^\times}\Psi_{\gamma,\kappa,\mu}(L_{i\mu})\subset J$. As
$\kappa\in\Omega_1$, this yields $\sum_{i\in{\mathbb
F}_p^\times}\Psi_{\gamma,\kappa,\mu}(L_{i\kappa})\subset J$. As a
result, the nilpotent subalgebra $J(\mu)$ acts nontriangulably on
$J$. As $\kappa([L_{r\mu},L_{-r\mu}])\ne 0$ and
$\Psi_{\gamma,\kappa,\mu}\big([L_\kappa,L_{\gamma-\kappa}]^{[p]}\big)\subset
J_{[p]}$, we have that $TR(J)=\dim T_1\ge 2$ (one should keep in
mind that $\mu$ vanishes on $H$ but not on
$[L_\kappa,L_{\gamma-\kappa}]^{[p]}$).

Since $\kappa\in\Omega$, we can now argue as in part~(a) of this
proof to deduce that $TR(J)\le 2$. As a result, $TR(J)=2$ for any
minimal ideal $J$ of of $L[\gamma,\kappa,\mu]$. As
$TR(L[\gamma,\kappa,\mu])\le 3$, this shows that
$\widetilde{S}=S\ot{\mathcal O}(m;\un{1})$ is the unique minimal
ideal of $L[\gamma,\kappa,\mu]$ and $TR(\widetilde{S})=TR(S)=2$.
According to \cite[Thm.~2.6]{PS2}, we can choose
$\Psi_{\gamma,\kappa,\mu}$ such that
$$\Psi_{\gamma,\kappa,\mu}(T)\,=\,(T'_0\ot1)+F(d\ot 1+{\rm
Id}_S\ot t_0),\quad\ T_0'\subset S_p,\ d\in\Der S,\ t_0\in
W(m;\un{1}).$$ Moreover, if $d$ is an inner derivation of $S$, then
we can assume further that $d=0$. Since $T_1=T_0'\ot 1$, we get
$\dim T_0'=2$. Set ${\mathfrak t}:=T_0'+Fd$, a torus in $\Der S$.
The subalgebra $S\ot F$ of $\widetilde{S}$ is invariant under the
action of $\Psi_{\gamma,\kappa,\mu}(T)$. Given
$\delta\in\Gamma((S\ot F),\Psi_{\gamma,,\kappa,\mu}(T))$ we denote
by $\bar{\delta}$ the unique $\mathfrak t$-root in
$\Gamma(S,{\mathfrak t})$ for which $S_{\bar{\delta}}\ot F= (S\ot
F)_\delta$.

\smallskip

\noindent (d)  Suppose $t_0\in W(m;\un{1})_{(0)}$. Because
$\widetilde{S}$ and $S\ot{\mathcal O}(m;\un{1})_{(1)}$ are both
$T$-invariant, $T$ acts on $S\cong \widetilde{S}/(S\ot{\mathcal
O}(m;\un{1})_{(1)})$ as the torus ${\mathfrak t}\subset\Der S$.
Since $\widetilde{S}_\kappa\ne (0)$ and $\kappa\in\Omega_1$, we also
have that $\Psi_{\gamma,\kappa,\mu}(L_{\pm r\mu})\ne (0)$. We
mentioned above that $\Psi_{\gamma,\kappa,\mu}(L_{\pm r\mu})\subset
\widetilde{S}$.  Define ${\mathfrak t}_0:={\mathfrak t}\cap\ker
\bar{\mu}$. Then $\dim{\mathfrak t}_0\le 2$ and ${\mathfrak
c}_S({\mathfrak t}_0)=S(\bar{\mu})$. Because $S_p\ot{\mathcal
O}(m;\un{1})_{(1)}$ is $p$-nilpotent and $\widetilde{S}(\mu)$ acts
nontriangulably on $\widetilde{S}$ by our discussion in part~(c),
the subalgebra $S(\bar{\mu})$ is nilpotent and acts nontriangulably
on $S$. Applying \cite[Thm.~2(ii)]{P94} now yields $S\cong{\mathcal
M}(1,1)$. But then all derivations of $S$ are inner; see
\cite[Thm.~7.1.4]{St04} for example. Then $d=0$ and ${\mathfrak t}$
is a torus of maximal dimension in $S_p$. It follows that
$S(\bar{\mu})={\mathfrak c}_S({\mathfrak t}_0)$ is a Cartan
subalgebra of toral rank $1$ in $S$. Since such Cartan subalgebras
are triangulable by \cite[Thm.~2]{P94}, our assumption on $t_0$ is
false.

Thus, $t_0\not\in W(m;\un{1})_{(0)}$. Recall that $\mu$ and $\kappa$
are both nonzero on $T_1=T_0'\ot 1$. Since $\mu$ vanishes on $H$ and
the nonsolvable root $\kappa$ does not vanish on
$\Psi_{\gamma,\kappa,\mu}\big([L_{i\kappa},L_{-i\kappa}]\big)\subset
\Psi_{\gamma,\kappa,\mu}(H)\cap \widetilde{S}$ for some
$i\in{\mathbb F}_p^\times$, the roots $\mu$ and $\kappa$ are
linearly independent on $T_1$. Hence
$$\Psi_{\gamma,\kappa,\mu}(T)\,=\,T_1\oplus(\Psi_{\gamma,\kappa,\mu}(T)\cap
\ker\mu\cap\ker\kappa),$$ implying that
$\pi(\Psi_{\gamma,\kappa,\mu}(T)\cap\ker\mu\cap\ker\kappa)\not\subset
W(m;\un{1})_{(0)}$. In that case \cite[Thm.~2.6]{PS2} says that
$\Psi_{\gamma,\kappa,\mu}$ can be selected such that $d=0,$
$t_0=(1+x_1)\partial_1$, and
$\Psi_{\gamma,\kappa,\mu}(T)\cap\ker\mu\cap\ker\kappa\,=\,F({\rm
Id}_S\ot t_0)$.

Then $\widetilde{S}(\kappa,\mu)=S\ot F[x_2,\ldots, x_m]$ and the
evaluation map ${\rm
ev}\colon\,\widetilde{S}(\kappa,\mu)\twoheadrightarrow S$, taking
$s\ot f\in S\ot F[x_2,\ldots,x_m]$ to $f(0)s\in S$, is
$T$-equivariant. As before, $S(\bar{\mu})$ acts nontriangulably on
$S$. Since in the present case $\mathfrak t$ is a torus of maximal
dimension in $S_p$, its $1$-section $S(\bar{\mu})$ has toral rank
$1$ in $S$. Since such a Cartan subalgebra
 must act triangulably on $S$ by \cite[Thm.~2]{P94}, we reach a
 contradiction, thereby proving the proposition.
\end{proof}
\begin{cor}\label{allomega}
The following are true:

\begin{itemize}
\item[(i)] $\Gamma(L,T)=\Omega$.

\smallskip

\item[(ii)] If $\alpha,\beta\in\Gamma(L,T)$, then
$L[\alpha,\beta]$ is not as in case~4) of Proposition~\ref{types}.
\end{itemize}
\end{cor}
\begin{proof}
(1) Suppose $\Gamma(L,T)\ne \Omega$ and let
$\lambda\in\Gamma(L,T)\setminus\Omega$. Take any $\alpha\in\Omega$
and consider the $T$-semisimple $2$-section $L[\alpha,\lambda]$. By
Theorem~\ref{rootspaces}, $L(\alpha)$ is not solvable, hence
$L[\alpha,\lambda]$ is not as  in case~1) of
Proposition~\ref{types}. Because of Lemma~\ref{mu(H)ne0} we have
$\lambda(H)=0$, hence $L(\lambda)$ is solvable. If
$L[\alpha,\lambda]$ is as in cases 2), 3), 5) or 6) of
Proposition~\ref{types}, then
$L_\lambda\subset\rad_T\,L(\alpha,\lambda)$ by Corollary~\ref{not4},
hence
$$[L_\alpha,L_\lambda]\subset
\big(\rad_T\,L(\alpha,\lambda)\big)\cap L_{\alpha+\lambda}
\subset\big(\rad_T\,L(\alpha+\lambda)\big)_{\alpha+\lambda}=(0)$$ by
Theorem~\ref{rootspaces} (because $\alpha+\lambda\in\Omega$). If
$L[\alpha,\lambda]$ is as in case~4) of Proposition~\ref{types},
then it follows from Proposition~\ref{type4} that $L_\lambda$
contains nonnilpotent elements of $L_p$. Since this contradicts
Proposition~\ref{mu(H)=0}, we see that $L[\alpha,\lambda]$ is not of
that type. As a consequence, $[L_\alpha,L_\lambda]=0$ for all
$\alpha\in\Omega.$ But then (\ref{eq0}) yields that $L_\lambda$ is
contained in the center of $L$. This contradiction proves the first
statement.

\smallskip

\noindent (2) If $L[\alpha,\beta]$ is as in case~4) of
Proposition~\ref{types}, then Proposition~\ref{type4} implies that
one of the roots in $\Gamma(L,T)\cap ({\mathbb F}_p\alpha+{\mathbb
F}_p\beta)$ is not contained in $\Omega$. Since this is impossible
by part~(1), our proof is complete.
\end{proof}
\begin{cor}\label{cenH}
For every $\alpha\in\Gamma(L,T)$ the radical of $L(\alpha)$ lies
in the center of $H$.
\end{cor}
\begin{proof}
 Recall that $\rad\,L(\alpha)\subset H$ by Theorem~\ref{rootspaces} and Corollary~\ref{allomega}.
 Set
$$\Omega_2:=\{\gamma\in\Gamma(L,T)\,|\,\,\,\gamma\big([H,\rad\,L(\alpha)]\big)\ne
0\}.$$ Suppose $\Omega_2\neq\emptyset$ and let $\beta\in\Omega_2$.
Since $\alpha,\beta\in\Omega$ by Corollary~\ref{allomega},
Proposition~\ref{types} applies to $L[\alpha,\beta]$. Since $\alpha$
vanishes on $[H,\rad\,L(\alpha)]$, the roots $\alpha$ and $\beta$
are ${\mathbb F}_p$-independent. As $\alpha$ and $\beta$ are both
nonsolvable by Theorem~\ref{rootspaces}, $L[\alpha,\beta]$ cannot be
as in cases~1) or ~2) of Proposition~\ref{types}. It cannot be
governed by cases~5) or 6) either, because in case~5) the radical of
$L[\alpha,\beta](\alpha)$ is trivial by Proposition~\ref{ham1}(2)
and in case~6) the radical of $L[\alpha,\beta](\alpha)$ is contained
in $\Psi_{\alpha,\beta}(T)$; see \cite[Lemmas~4.1 \& 4.4]{P94}.

Thus, $L[\alpha,\beta]$ is as in case~3) of Proposition~\ref{types}.
But then
$L[\alpha,\beta]=L[\alpha,\beta](\alpha)+L[\alpha,\beta](\beta)$ and
$[L_\alpha,L_\beta]\subset \rad_T\,L(\alpha,\beta)$. Since
$(\alpha+\beta)\big([H,\rad\,L(\alpha)]\big)\ne 0$, and
$L(\alpha+\beta)$ is solvable, it must be that
$\alpha+\beta\not\in\Gamma(L,T)$. We now derive that
$[L_\alpha,L_\beta]=(0)$ for all $\beta\in\Omega_2$. In view of
Schue's lemma~\cite[Prop.~1.3.6(1)]{St04}, this means that
$L_\alpha$ lies in the center of $L$.

This contradiction shows that $\Omega_2=\emptyset$. Hence the ideal
$H_\alpha:=[H,\rad\,L(\alpha)]$ of $H$ consists of $p$-nilpotent
elements of $L_p$. Now let $\beta$ be any root in $\Gamma(L,T)$.
Since $H_\alpha\subset H^{(1)}$, it follows from Theorem~\ref{1sec}
and (the proof of) Lemma~\ref{ham3} that $\Psi_\beta(H_\alpha)=(0)$.
Then $[H_\alpha,L(\beta)]\subset\rad\,L(\beta)$, forcing
$[H_\alpha,L_\beta]=(0)$; see Theorem~\ref{rootspaces}. As a result,
$[H_\alpha,L]=(0)$, and hence $H_\alpha=(0)$ by the simplicity of
$L$. This proves the corollary.
\end{proof}

We are finally in a position to describe the $2$-sections of $L$
with respect to $T$. Let ${\mathfrak z}(H)$ denote the center of
$H={\mathfrak c}_L(T)$.
\begin{theo}\label{CSA}
The following are true:
\begin{itemize}
\item[(i)] $H^4=(0)$ and $\,H^{[p]}\subset T$.

\smallskip

\item[(ii)] $\dim H^2=3$ and $\,\dim H^3=2$.
\smallskip

\item[(iii)] $H^3\subset T$ and $\,\dim H/{\mathfrak z}(H)=3$.

\smallskip

\item[(iv)] ${\mathfrak z}(H)=H\cap T$.
\end{itemize}
\end{theo}
\begin{proof}
(a) Let $\alpha\in \Gamma(L,T)$. Then $\alpha\in\Omega$ by
Corollary~\ref{allomega}(i). It is immediate from Theorem~\ref{1sec}
that $H^4\subset\rad\,L(\alpha)$. Then
$[H^4,L_\alpha]\subset(\rad\,L(\alpha))_\alpha=(0)$ by
Theorem~\ref{rootspaces}. Since this holds for every root $\alpha$
and $L$ is simple, we derive $H^4=(0)$.

Let ${\mathcal N}(H_p)$ denote the set of all $p$-nilpotent
elements of $H_p$. Since $\dim L_\gamma=5$ for all
$\gamma\in\Gamma(L,T)$ any $p$-nilpotent element $x\in{\mathcal
N}(H_p)$ has the property that $(\ad
x)^5\big(\sum_{\gamma\in\Gamma(L,T)}\,L_\gamma\big)=0$. Then
$x^{[p]}=0$ by the simplicity of $L$. The Jordan--Chevalley
decomposition in $H_p$ now yields $(H_p)^{[p]}\subset T$, forcing
$H^{[p]}\subset T$. As a result, statement~(i) follows, and we
also deduce that ${\mathcal N}(H_p)=\{x\in H_p\,|\,\,x^{[p]}=0\}$
and $H_p\subset H+T$.

Since $H^4=(0)$ and $[T,H]=0$, Jacobson's formula implies that
$(x+y)^{[5]}=x^{[5]}+y^{[5]}$ for all $x,y\in H_p$. Therefore,
${\mathcal N}(H_p)$ is a subspace of $H$. By the Jordan--Chevalley
decomposition in $H_p$, we also get $H_p\subset {\mathcal
N}(H_p)\oplus T$.

\smallskip

\noindent (b) Since $\Gamma(L,T)=\Omega$, it follows from
Theorem~\ref{1sec} and (the proof of) Lemma~\ref{ham3} that
$H^2+\rad\,L(\alpha)$ has codimension $2$ in $H$ for every
$\alpha\in\Gamma(L,T)$. Since $\rad\,L(\alpha)\subset{\mathfrak
z}(H)$ by Corollary~\ref{cenH}, there exist $x,y\in H$ such that
$H=Fx+Fy+H^2+{\mathfrak z}(H)$. As a consequence, $H^2=F[x,y]+H^3$
and $H^3=F[x,[x,y]]+F[y,[y,x]]+H^4$. As $H^4=(0)$, this gives
$\dim H^3\le 2$ and $\dim H^2=1+\dim H^3$.

Let $\alpha,\beta\in\Gamma(L,T)$ be such that
$L[\alpha,\beta]\cong{\mathcal M}(1,1)$ (such a pair of roots exists
by \cite[Thm~1(ii)]{P94}). It is immediate from \cite[Lemmas~4.1 \&
4.4]{P94} that $\dim\Psi_{\alpha,\beta}(H^3)=2$. Hence $\dim H^3\ge
2$. In conjunction with the above remarks, this gives $\dim H^3=2$
and $\dim H^2=3$. Statement~(ii) follows.

\smallskip

\noindent (c) Since $H^4=(0)$, we have that $H^3\subset{\mathfrak
z}(H)$. If the nilpotent Lie algebra $H/\mathfrak{z}(H)$ has
codimension $< 3$ in $H$, then it is abelian. In this case
$H^2\subset \mathfrak{z}(H)$, forcing $H^3=(0)$. This contradiction
shows that $\mathfrak{z}(H)$ has codimension $\ge 3$ in $H$.  Since
$H^3\ne (0)$ has codimension $1$ in $H^2$, the equality
$H^2\cap{\mathfrak z}(H)=H^3$ holds. Therefore,
\begin{eqnarray*}3&\le& \dim H/{\mathfrak
z}(H)=\dim H/(H^2+{\mathfrak z}(H))+\dim H^2/H^3\\
&\le&\dim
H/(H^2+\rad\,L(\alpha))+\dim H^2/H^3=3.
\end{eqnarray*}
This implies that ${\mathfrak z}(H)$ has codimension $3$ in $H$.

Let $h\in {\mathfrak z}(H)$ and write $h=h_s+h_n$ with $h_s\in T$
and $h_n\in{\mathcal N}(H_p)$. In view of our earlier remarks,
$h_n\in{\mathfrak z}(H)\cap (T+H)$. Because $\Gamma(L,T)=\Omega$,
Theorem~\ref{1sec} shows that for every $\gamma\in\Gamma(L,T)$ the
element
$\Psi_\gamma(h_n)\in\Psi_\gamma(T)+\Psi_\gamma(H)=\Psi_\gamma(H)$ of
$L[\gamma]\cong H(2;\un{1})^{(2)}\oplus F(1+x_1)^4\partial_2$ is
$p$-nilpotent in $L[\gamma]$ and commutes with $\Psi_\alpha(H)$.
Arguing as in the proof of Lemma~\ref{ham3} it is now
straightforward to see that $\Psi_\gamma(h_n)=0$. Then
$[h_n,L(\gamma)]\subset \rad\,L(\gamma)$. In view of
Theorem~\ref{rootspaces}, this entails that $[h_n,L_\gamma]=0$ for
all $\gamma\in\Gamma(L,T)$. As a consequence, $h_n=0$, forcing
${\mathfrak z}(H)=H\cap T$. Combined with our remarks in part~(b)
this gives (iii), completing the proof.
\end{proof}
\begin{cor}\label{nottype3}
Let $\alpha,\beta\in\Gamma(L,T)$. Then case~3) of
Proposition~\ref{types} does not occur for $L[\alpha,\beta]$.
\end{cor}
\begin{proof}
Indeed, otherwise the $T$-socle of $L[\alpha,\beta]$ has the form
$S_1\oplus S_2=S_1(\delta_1)\oplus S_2(\delta_2)$. Then
$\Psi_{\alpha,\beta}(H)\cap S_i(\delta_i)\cong \Psi_{\delta_i}(H)$
for $i=1,2$. As $\delta_1,\delta_2\in\Omega$ by
Corollary~\ref{allomega}(i), it follows from Theorem~\ref{1sec} that
$S_i(\delta_i)\cong H(2;\un{1})^{(2)}\oplus F(1+x_2)^4\partial_2$
and $\Psi_{\delta_i}(H)$ is a nonabelian Cartan subalgebra of
$S_i(\delta_i)$. Then Lemma~\ref{ham3} implies that $\dim
\Psi_{\delta_i}(H^2)=2$. As a consequence,
$\Psi_{\alpha,\beta}(H^2)\cap S_i(\delta_i)$ is $2$-dimensional for
$i=1,2$. But then $\dim H^2\ge 4$ contrary to Theorem~\ref{CSA}(ii).
The result follows.
\end{proof}
\begin{cor}\label{allGamma}
The following are true:
\begin{itemize}
\item[(1)] $\Gamma(L,T)\cup\{0\}$ is an ${\mathbb F}_p$-subspace
of $T^*$.

\smallskip

\item[(2)] The $p$-envelope of $H^3$ in $L_p$ coincides with $T$.

\smallskip

\item[(3)] $H_p=H+T$.
\end{itemize}
\end{cor}
\begin{proof}
(1) Since every $\gamma\in \Gamma(L,T)$ is Hamiltonian by
Theorem~\ref{rootspaces}, we have ${\mathbb
F}_p^\times\gamma\subset\Gamma(L,T)$. Let
$\alpha,\beta\in\Gamma(L,T)$ be ${\mathbb F}_p$-independent. Then
$\Gamma(L[\alpha,\beta],\Psi_{\alpha,\beta}(T))$ contains two
nonsolvable roots. In view of Corollary~\ref{nottype3}, this implies
that $L[\alpha,\beta]$ is determined by cases~5) or 6) of
Proposition~\ref{types}. In both cases,
$\Gamma(L[\alpha,\beta],\Psi_{\alpha,\beta}(T))\cup\{0\}={\mathbb
F}_p\alpha+{\mathbb F}_p\beta$; see Lemma~\ref{ham1}(4) and
\cite[Lemmas~4.1 \& 4.4]{P94}. As a consequence,
$\alpha+\beta\in\Gamma(L,T)$. Statement~(1) follows.

\smallskip

\noindent (2) By Theorem~\ref{CSA}(3), $H^3\subset T$. Denote by
$T_0$ the $p$-envelope of $H^3$ in $T$ and suppose that $T_0\ne T$.
Then $T_0$ is a proper subtorus of $T$. By part~(1), there exists
$\gamma\in\Gamma(L,T)$ such that $\gamma(T_0)=0$. Then
$\gamma(H^3)=0$ contrary to Corollary~\ref{allomega}(i). Therefore,
$(H^3)_p=T$.

\smallskip

\noindent (3) It is immediate from Theorem~\ref{CSA}(i) that
$H_p\subset H+T$. Since $T=(H^3)_p\subset H_p$ by part~(2), we now
derive that $H_p=H+T$.
\end{proof}
We now summarize the results of this section:
\begin{theo}\label{sum}
Let $L$, $T$ and $H$ be as above. Then the following hold:
\begin{itemize}
\item[1)] $\Gamma(L,T)\cup\{0\}$ is an ${\mathbb F}_p$-subspace of
$T^*$ and no root in $\Gamma(L,T)$ vanishes on $H^3$.

\smallskip

\item[2)] $H^3\subset T$,\, ${\mathfrak z}(H)=H\cap T$,\,
$H_p=H+T$,\, $\dim H/(H\cap T)=3$,\, $\dim H^2=3$, and $\dim
H^3=2$. The $p$-envelope of $H^3$ in $L_p$ coincides with $T$.

\smallskip

\item[3)] $\rad\,L(\alpha)=H\cap T \cap\ker \alpha,\,$ $\dim
L_\alpha=5$, and $L[\alpha]\cong H(2;\un{1})^{(2)}\oplus
F(1+x_1)^4\partial_2$ for every $\alpha\in\Gamma(L,T)$.

\smallskip

\item[4)] If $\alpha,\beta\in\Gamma(L,T)$ are ${\mathbb
F}_p$-independent, then either $L[\alpha,\beta]\cong {\mathcal
M}(1,1)$ or $$H(2;(2,1))^{(2)}\subset L[\alpha,\beta]\subset
H(2;(2,1)).$$ Furthermore, $L[\alpha,\beta]\cong
L(\alpha,\beta)/H\cap T\cap\ker\alpha\cap\ker\beta$.
\end{itemize}
\end{theo}
\begin{proof}
Parts~1) and 2) are just reformulations of our earlier results. In
order to get~3) and 4) it suffices to observe that
$\rad\,L(\alpha)\subset {\mathfrak z}(H)=H\cap T$; see
Corollary~\ref{cenH} and Theorem~\ref{CSA}(iv).
\end{proof}
\section{\bf Some properties of the restricted Melikian algebra}
In order to proceed further with our investigation, we now need more
information on central extensions and irreducible representations of
the Melikian algebra ${\mathcal M}(1,1)$.
\begin{prop}\label{bilinear}
Every Melikian algebra ${\mathcal M}(\un{n})$, where ${\un
n}=(n_1,n_2)$, possesses a nondegenerate invariant symmetric
bilinear form.
\end{prop}
\begin{proof}
Adopt the notation of \cite[Sect.~4.3]{St04} and consider the
natural grading
$${\mathcal M}(\un{n})\,=\,{\mathcal M}_{-3}\oplus {\mathcal
M}_{-2}\oplus{\mathcal M}_{-1}\oplus{\mathcal M}_0\oplus{\mathcal
M}_1\oplus\cdots\oplus{\mathcal M}_s,\qquad s=3(5^{n_1}+5^{n_2})-7$$
of the Melikian algebra ${\mathcal M}={\mathcal M}(\un{n})$. Recall
that ${\mathcal
M}_0=\bigoplus_{i,j=1}^2\,x_i\partial_j\cong\mathfrak{gl}(2),\,$
${\mathcal M}_{-3}=F\partial_1\oplus F\partial_2\,$ and ${\mathcal
M}_s=Fx^{(\tau(\un{n}))} \tilde{\partial}_1\oplus
Fx^{(\tau(\un{n}))}\tilde{\partial}_2$, where
$\tau(\un{n})=(5^{n_1}-1,5^{n_2}-1)$. Both ${\mathcal M}_{-3}$ and
${\mathcal M}_s$ are $2$-dimensional irreducible ${\mathcal
M}_0$-modules. Using the multiplication table \cite[(4.3.1)]{St04},
it is easy to observe that
\begin{eqnarray*}
\big[x_1\partial_1,x^{(\tau(\un{n}))}\tilde{\partial}_1\big]&=&
(-2+2)x^{(\tau(\un{n}))}\tilde{\partial}_1\,=\,0,
\quad\quad  \big[x_2\partial_1,x^{(\tau(\un{n}))}\tilde{\partial}_1\big]=0,\\
\big[x_2\partial_2,x^{(\tau(\un{n}))}\tilde{\partial}_1\big]&=&
(-1+2)x^{(\tau(\un{n}))}\tilde{\partial}_1.
\end{eqnarray*}
This shows that $x^{(\tau(\un{n}))}\tilde{\partial}_1$ is a
primitive vector of weight $(0,1)$ for the Borel subalgebra
${\mathfrak b}:= Fx_1\partial_1\oplus Fx_2\partial_2\oplus
Fx_2\partial_1$ of ${\mathcal M}_0$. Now let $f$ be the linear
function on ${\mathcal M}_{-3}$ such that $f(\partial_1)=0$ and
$f(\partial_2)=1$. Then
$(x_1\partial_1)(f)=-f\circ(x_1\partial_1)=0,\,\,$
$(x_2\partial_2)(f)=-f\circ(x_2\partial_2)=f$ and
$(x_2\partial_1)(f)=-f\circ(x_2\partial_1)=0$, showing that $f\in
({\mathcal M}_{-3})^*$ is a primitive vector of weight $(0,1)$ for
the Borel subalgebra ${\mathfrak b}$. From this it is immediate that
$({\mathcal M}_{-3})^*\cong {\mathcal M}_s$ as ${\mathcal
M}_0$-modules. As ${\mathcal M}$ is an irreducible graded ${\mathcal
M}_p$-module, \cite[Lemma~4]{P85} shows that there exists a module
isomorphism $\theta\colon\,{\mathcal
M}\stackrel{\sim}{\longrightarrow}{\mathcal M}^*$ sending ${\mathcal
M}_i$ onto $({\mathcal M}_{s-3-i})^*$ for all $i\in\{-3,\ldots,s\}$
(as usual, we identify $({\mathcal M}_i)^*$ with the subspace of
${\mathcal M}^*$ consisting of all linear functions vanishing on all
${\mathcal M}_k$ with $k\ne i$).

Define a bilinear form $b\colon\,{\mathcal M}\times{\mathcal
M}\rightarrow F$  by setting $b(x,y):=(\theta(x))(y)$ for all
$x,y\in{\mathcal M}$. Since $\theta$ is an isomorphism of ${\mathcal
M}$-modules, the form $b$ is nondegenerate and $\mathcal
M$-invariant. Next we define a bilinear skew-symmetric form $b'$ on
$\mathcal M$ by setting
 $b'(x,y):=b(x,y)-b(y,x)$ for all
$x,y\in{\mathcal M}$. As $\mathcal M$ is a simple Lie algebra, the
invariant form $b'$ is either nondegenerate or zero. As
$\dim{\mathcal M}=5^{n_1+n_2+1}$ is odd, it must be that $b'=0$.
Therefore, the form $b$ is symmetric.
\end{proof}
From now on we denote by $\mathcal M$ the restricted Melikian
algebra ${\mathcal M}(1,1)$.
\begin{prop}\label{centralext}
If $\widetilde{\mathcal M}$ is a Lie algebra with center
$\mathfrak{z}=\mathfrak{z}(\widetilde{\mathcal M})$ such that
$\widetilde{\mathcal M}/\mathfrak{z}\cong{\mathcal M}$, then
$\widetilde{\mathcal M}^{(1)}\cong{\mathcal M}$ and
$\widetilde{\mathcal M}\,=\,\widetilde{\mathcal
M}^{(1)}\oplus{\mathfrak z}$.
\end{prop}
\begin{proof}
We need to show that the second cohomology group ${\mathrm
H}^2({\mathcal M},F)$ vanishes. Let $b$ be the nondegenerate
bilinear form from the proof of Proposition~\ref{bilinear}. By a
standard argument explained in detail in \cite[p.~681]{P94}, for
every $2$-cocycle $\varphi\colon\,{\mathcal M}\times{\mathcal
M}\rightarrow F$ there exists a derivation $d\in\Der{\mathcal M}$
such that $b(d(x),y)=-b(x,d(y))$ and $\varphi(x,y)=b(d(x),y)$ for
all $x,y\in{\mathcal M}$. Moreover, $\varphi$ is a $2$-coboundary if
and only if the derivation $d$ is inner. Since $\Der{\mathcal
M}\,=\,\ad{\mathcal M}$ by \cite[Thm.~7.1.4]{St04}, for instance, we
now obtain ${\mathrm H}^2({\mathcal M},F)=0$, as desired.
\end{proof}
If $V$ is an irreducible module over a finite dimensional restricted
Lie algebra $\mathcal L$ over $F$, then there exists a linear
function $\chi=\chi_V\in{\mathcal L}^*$ such that for every $x\in
\mathcal L$ the central element $x^p-x^{[p]}$ of $U({\mathcal L})$
acts on $V$ as the scalar operator $\chi(x)^p\,{\rm Id}_V$. The
linear function $\chi$ is called the $p$-{\it character} of $V$.
Given $f\in{\mathcal L}^*$ we denote by $\mathfrak{z}_{\mathcal
L}(f)$ the stabilizer of $f$ in $\mathcal L$. Recall that
$\mathfrak{z}_{\mathcal L}(f)=\{x\in\mathcal L\,|\,\,f([x,{\mathcal
L}])=0\}$ is a restricted subalgebra of even codimension in
$\mathcal L$.

For our constructions in the final sections of this work we need
some information on the $p$-characters of irreducible
representations of dimension $\le 125$ of the restricted Melikian
algebra ${\mathcal M}\,=\,\bigoplus_{i= -3}^s\,{\mathcal M}_i$.
\begin{prop}\label{p-character}
If $V$ is an irreducible $\mathcal M$-module of dimension $\le 125$,
then the $p$-character of $V$ vanishes on the subspace
$\bigoplus_{i\ge -2}\,{\mathcal M}_i$. If $V$ has a nonzero
$p$-character, then $\dim V=125$.
\end{prop}
\begin{proof}
Write ${\mathcal M}^*=\,\bigoplus_{i=-3}^s\,({\mathcal M}_i)^*$,
where $({\mathcal M}_i)^*=\{f\in{\mathcal M}^*\,|\,\,\bigoplus_{j\ne
i}{\mathcal M}_j\subset \ker f\}$ and $s=3(5+5)-7=23$. Let $\chi$ be
the $p$-character of the $\mathcal M$-module $V$. If $\chi=0$, then
there is nothing to prove; so suppose $\chi\ne 0$. Then
$\chi=\sum_{i=-3}^d\,\chi_i$, where $\chi_i\in({\mathcal M}_i)^*$
and $\chi_d\ne 0$.

\smallskip

\noindent (a) We first suppose that $d>0$ and let $2q\,=\,
\mathrm{codim}_{\mathcal M}\,\mathfrak{z}_{\mathcal M}(\chi_d)$.
Then \cite[Prop.~5.5]{PSk} yields that $5^q\,\vert\, \dim V$. Since
$\dim V\le 5^3$, it follows that $\mathfrak{z}_{\mathcal M}(\chi_d)$
has codimension $\le 6$ in $\mathcal M$. Let $b$ be the $\mathcal
M$-invariant nondegenerate bilinear form from the proof of
Proposition~\ref{bilinear}. Then $\chi_d=b(z,\cdot\,)=\theta(z)$ for
some nonzero $z\in{\mathcal M}_{s-3-d}$ and  $\mathfrak{z}_{\mathcal
M}(\chi_d)=\mathfrak{c}_{\mathcal M}(z)$. It follows that the set
$${\mathcal X}:=\{x\in{\mathcal
M}_{s-3-d}\,|\,\,\mathrm{codim}_{\mathcal M}\,\mathfrak{c}_{\mathcal
M}(x)\le 6\}$$ is nonzero. It is straightforward to see that
$\mathcal X$ is a Zariski closed, conical subset of ${\mathcal
M}_{s-3-d}$ invariant under the subgroup $\mathrm{Aut}_0\,\mathcal
M$ of all automorphisms of $\mathcal M$ preserving the natural
grading of $\mathcal M$. Let ${\mathbb P}({\mathcal X})$ be the
closed subset of the projective space ${\mathbb P}({\mathcal
M}_{s-3-d})$ corresponding to $\mathcal X$ and let $\mathbf{T}$
denote the $2$-dimensional torus of the algebraic group
$\mathrm{Aut}_0\,\mathcal M$ whose group of rational characters is
described in \cite[p.~72]{Skr}. Note that the Lie algebra of
$\mathbf T$ equals $F(\ad x_1\partial_1)\oplus F(\ad
x_2\partial_2)$.

The connected abelian group $\mathbf{T}$ acts regularly on $\mathcal
X$, hence fixes a point in ${\mathbb P}(\mathcal X)$ by Borel's
theorem. This means that there exists a nonzero $x_0\in{\mathcal
M}_{s-3-d}$ such that $\mathfrak{c}_{\mathcal M}(x_0)$ has
codimension $\le 6$ in $\mathcal M$ and $\mathbf{T}\cdot x_0\subset
Fx_0$. Let $\mathfrak{n}_0$ denote the normalizer of $Fx_0$ in
$\mathcal M$ and set ${\mathfrak t}:=F(x_1\partial_1)\oplus
F(x_2\partial_2)$, a $2$-dimensional torus in $\mathcal M$. By our
choice of $x_0$ (and $\mathbf T$) we have that
$[\mathfrak{t},x_0]\subset Fx_0$.

Suppose $[\mathfrak{t},x_0]\ne 0$. Then
$\mathfrak{n}_0\supsetneq\mathfrak{c}_{\mathcal M}(x_0)$. As a
consequence, $\mathfrak{n}_0$ is a proper subalgebra of codimension
$\le 5$ in $\mathcal M$. By a result of Kuznetsov
\cite[Thm.~4.7]{Ku}, every proper subalgebra of $\mathcal M$ has
codimension $\ge 5$ and every subalgebra of codimension $5$ contains
$\bigoplus_{i\ge 1}\,{\mathcal M}_i$ (see also
\cite[Thm.~4.3.3]{St04} and \cite[Sect.~1]{Skr}). Since the
subalgebra $\bigoplus_{i\ge 1}\,{\mathcal M}_i$ of ${\mathfrak n}_0$
acts nilpotently on $\mathcal M$, it must annihilate $Fx_0$. On the
other hand, it is immediate from the simplicity of the graded Lie
algebra $\mathcal M$ that the graded subspace
$\mathrm{Ann}_{\mathcal M}\big(\bigoplus_{i>0}\,{\mathcal M}_i\big)$
coincides with ${\mathcal M}_s$. So $x_0\in{\mathcal M}_s$ forcing
$d=-3$, a contradiction.

Now suppose $[\mathfrak{t},x_0]=0$. Using \cite[(4.3.1)]{St04} one
checks immediately that $\mathfrak{c}_{\mathcal
M}(\mathfrak{t})\,=\,\mathfrak{t}\oplus Fx_1^3x_2^3\oplus
Fx_1^4x_2^3\tilde{\partial}_1\oplus Fx_1^3x_2^4\tilde{\partial}_2$.
In view of \cite[p.~200]{St04}, we have that
$\mathfrak{t}\subset{\mathcal M}_0,\,$ $x_1^2x_2^2\in{\mathcal
M}_{10}$ and $Fx_1^4x_2^3\tilde{\partial}_1\oplus
Fx_1^3x_2^4\tilde{\partial}_2\subset\mathcal{M}_{20}$. As $d>0$ by
our present assumption, we have $s-3-d=23-3-d<20$. Rescaling $x_0$
if need be we thus may assume that either $x_0=x_1^2x_2^2$ or
$x_0=x_1\partial_1+\alpha\, x_2\partial_2$ for some $\alpha\in F$
(by symmetry). Applying \cite[(4.3.1)]{St04} it is easy to observe
that $\mathfrak{c}_{{\mathcal M}_i}(x_1^2x_2^2)=(0)$ for $i<0$ and
$\mathfrak{c}_{{\mathcal M}_0}(x_1^2x_2^2)=\mathfrak{t}$. This shows
that the case $x_0=x_1^2x_2^2$ is impossible (as
$\mathfrak{c}_{\mathcal M}(x_0)$ has codimension $\le 6$ in
$\mathcal M$). If $x_0=x_1\partial_1+\alpha\, x_2\partial_2$, then
$\big[x_0,\mathcal{M}\big]$ contains all $x_1^i\partial_1$ with
$i\in\{0,2,3,4\}$ and all $x_1^jx_2\partial_2$ with
$j\in\{1,2,3,4\}$. It follows that $\mathrm{codim}_{\mathcal
M}\,\mathfrak{c}_{\mathcal M}(x_0)\ge 8$ in this case, showing that
the case where $d>0$ cannot occur.

\smallskip

\noindent (b) Thus $d\le 0$. Recall from \cite[p.~72]{Skr} that the
group of rational characters of $\mathbf{T}$ has $\Z$-basis
$\{\varepsilon_1,\,\varepsilon_2\}$ and the $\mathbf{T}$-weight
vectors $\partial_1,\partial_2\in{\mathcal M}_{-3},$ $1\in{\mathcal
M}_{-2},$ $\tilde{\partial}_1,\tilde{\partial}_2\in{\mathcal
M}_{-1}$ and $x_1\partial_2, x_2\partial_1\in {\mathcal M}_0$ have
weights
$-2\varepsilon_1-\varepsilon_2,-\varepsilon_1-2\varepsilon_2,-\varepsilon_1-\varepsilon_2,
-\varepsilon_1,-\varepsilon_2$ and
$\varepsilon_1-\varepsilon_2,-\varepsilon_1+\varepsilon_2$,
respectively.

Assume that $\chi_0(x_1\partial_2)\ne 0$ and consider the
cocharacter
$\varepsilon^*_1\colon\,F^\times\rightarrow\mathrm{Aut}\,\mathcal{M}$
such that $(\varepsilon_1^*(t))(x)=t^nx$ for all $t\in F^\times$ and
all weight vectors $x\in{\mathcal
M}_{n\varepsilon_1+m\varepsilon_2}$, where $m,n\in\Z$. Let
${\mathcal M}\,=\,\bigoplus_{i\in\Z}\,{\mathcal M}(i)$ be the
$\Z$-grading of $\mathcal M$ induced by $\varepsilon_1^*$. Since
$d\le 0$ and $\chi_0(x_1\partial_2)\ne 0$ by our assumption, we have
that $\chi=\chi(-2)+\chi(-1)+\chi(0)+\chi(1)$, where
$\chi(i)\in\mathcal{M}(i)^*$ and $\chi(1)\ne 0$. Applying
\cite[Prop.~5.5]{PSk} to the  graded Lie algebra
$\bigoplus_{i\in\Z}\,\mathcal{M}(i)$ we deduce that
$\mathfrak{z}_{\mathcal M}(\chi(1))$ has codimension $\le 6$ in
$\mathcal M$. Since in the present case $x_1\partial_1\in
\mathfrak{n}_{\mathcal M}(F\chi(1))\setminus\mathfrak{z}_{\mathcal
M}(\chi(1))$, the normalizer $\mathfrak{n}_{\mathcal M}(F\chi(1))$
has codimension $\le 5$ in $\mathcal M$. Using Kuznetsov's
description of subalgebras of codimension $5$ in $\mathcal M$ and
arguing as in part (a) we now obtain that $\chi(1)=b(y,\cdot\,)$ for
some $y\in\mathcal{M}_s$. Since in the present case $s-3-d\ne s$, we
reach a contradiction, thereby showing that
$\chi_0(x_1\partial_2)=0$. Arguing in a similar fashion one obtains
that $\chi_0$ vanishes on $x_2\partial_1$.

\smallskip

\noindent (c) Thus we may assume from now that $d\le 0$ and $\chi_0$
vanishes on $F(x_1\partial_2)\oplus F(x_2\partial_1)$. In this
situation \cite[Prop.~5.5]{PSk} is no longer useful, so we have to
argue differently. Denote by $\mathfrak{g}$ the Lie subalgebra of
$\mathcal M$ generated by the graded components $\mathcal{M}_{\pm
1}$. Using \cite[(4.3.1)]{St04} it is easy to check that
$\mathcal{M}_1\,=\,Fx_1\oplus Fx_2,\,$
$\mathcal{M}_1^2\,=\,F(x_1\widetilde{\partial}_1+x_2\widetilde{\partial}_2),\,$
$\mathcal{M}_1^3\,=\,F(x_1^2\partial_1+x_1x_2\partial_2)\oplus
F(x_1x_2\partial_1+x_2^2\partial_2)$ and $\mathcal{M}_1^4\,=\,(0)$.
Then it is immediate from \cite[Thm.~5.4.1]{St04} that
$\mathfrak{g}$ is a $14$-dimensional simple Lie algebra of type
$\mathrm{G}_2$. We identify $\chi$ with its restriction to
$\mathfrak g$, denote by $\mathbf{G}$ the simple algebraic group
$\mathrm{Aut}\,\mathfrak{g}$, and regard
$\mathbf{L}:=\mathrm{Aut}_0\,\mathcal{M}$ as a Levi subgroup of
$\mathbf{G}$. Clearly, $\mathbf{T}$ is a maximal torus of
$\mathbf{G}$ contained in $\mathbf{L}$. Also,
$\mathrm{Lie}(\mathbf{G})=\ad\mathfrak{g}$ and $5$ is a good prime
for the root system
$\mathbf{\Phi}=\mathbf{\Phi}(\mathbf{G},\mathbf{T})$. Since the
Killing form $\kappa$ of the Lie algebra $\mathfrak g$ is
nondegenerate, we may identify $\mathfrak{g}$ with $\mathfrak{g}^*$
via the $\mathbf{G}$-equivariant map sending $x\in \mathfrak{g}$ to
the linear function $\kappa(x,\cdot\,)\in\mathfrak{g}^*$.

Let $\mathbf{P}$ be the parabolic subgroup of $\mathbf{G}$ with
$\mathrm{Lie}(\mathbf{P})\,=\,\mathrm{ad}\big(\bigoplus_{i\ge
0}\,\mathfrak{g}_i\big)$, where
$\mathfrak{g}_i=\mathfrak{g}\cap\mathcal{M}_i$, and let
$\mathbf{\Phi}^+$ be a positive system in $\mathbf{\Phi}$ containing
the $\mathbf{T}$-weights of $\bigoplus_{i>0}\,\mathfrak{g}_i$. Let
$\{\alpha_1,\alpha_2\}$ be the basis of simple roots of
$\mathbf{\Phi}$ contained in $\mathbf{\Phi}^+$. Adopting Bourbaki's
numbering we will assume that $\mathfrak{g}_0$ is spanned by
$\mathfrak{t}$ and root vectors $e_{\pm\alpha_2}$ and
$\mathfrak{g}_1$ is spanned by root vectors $e_{\alpha_1}$ and
$e_{\alpha_1+\alpha_2}$. We stress that $\alpha_1$ is a {\it short}
root of $\mathbf{\Phi}$.

Since $g(\chi_0)=\chi_0$ for all $g\in \mathbf{T}$ and
$\chi_{-1}+\chi_{-2}+\chi_{-3}$ is a linear combination of
$\mathbf{T}$-weight vectors corresponding to positive roots, the
Zariski closure of $\mathbf{T}\cdot\chi$ contains $\chi_0$. It
follows that $\dim \mathbf{G}\cdot\chi\ge \dim
\mathbf{G}\cdot\chi_0$. Since $\chi_0$ vanishes on all root vectors
$e_\alpha\in \mathfrak g$ with $\alpha\in\mathbf{\Phi}$ and $5$ is a
good prime for $\mathbf{\Phi}$, the stabilizer
$Z_\mathbf{G}(\chi_0)$ of $\chi_0$ in $\mathbf{G}$ is a Levi
subgroup of $\mathbf{G}$; see \cite[(3.1)]{P95} and references
therein. Since the $\mathfrak{g}$-module $V$ has $p$-character
$\chi$, the Kac--Weisfeiler conjecture proved in \cite{P95} shows
that $5^{(\dim \mathbf{G}\cdot\chi)/2}\mid\dim V$.

Suppose $\chi_0\ne 0$. Then $Z_\mathbf{G}(\chi_0)$ is a proper Levi
subgroup of $\mathbf{G}$. Since any Levi subgroup of $\mathbf{G}$ is
conjugate to a standard Levi subgroup, this implies that $\dim
Z_\mathbf{G}(\chi_0)\le 4$. As a consequence,
$$\dim \mathbf{G}\cdot\chi\ge \dim \mathbf{G}\cdot \chi_0=\dim \mathbf{G}-
\dim Z_\mathbf{G}(\chi_0)\ge 10.$$ But then $5^5\mid \dim V$, a
contradiction. Thus, $\chi=\kappa(y_1+y_2+y_3,\cdot\,)$ for some
$y_i\in\mathfrak{g}_i$.

Suppose $y_1\ne 0$. Since $y$ is a nilpotent element of
$\mathfrak{g}$, all nonzero scalar multiples of $y$ are
$\mathbf{G}$-conjugate. From this it is immediate that the Zariski
closure of $\mathbf{G}\cdot y$ contains $y_1$, implying $\dim
\mathbf{G}\cdot y\ge \dim \mathbf{G}\cdot y_1$. As all nonzero
elements of $\mathfrak{g}_1$ are conjugate under the action of
$\mathbf{L}$, we may assume that $y_1=e_{\alpha_1}$. As
$\dim\mathfrak{c}_{\mathfrak g}(e_{\alpha_1})=6$, it follows that
$$\dim \mathbf{G}\cdot\chi=\dim \mathbf{G}\cdot y\ge \dim \mathbf{G}\cdot y_1=
\dim \mathbf{G}-\dim Z_\mathbf{G}(y_1)\ge \dim \mathbf{G}-\dim
\mathfrak{c}_{\mathfrak g}(y_1)=8.$$ Applying \cite[Thm.~I]{P95} now
gives $5^4\mid \dim V$. Since this is false, it must be that
$y_1=0$. If $y_2\ne 0$, then $y_2$ is a nonzero multiple of
$e_{2\alpha_1+\alpha_2}$ (for $\mathfrak{g}_2=\big[{\mathcal
M}_1,\mathcal{M}_1\big]=Fe_{2\alpha_1+\alpha_2}$). As $y=y_2+y_3$,
it is easy to see that the orbit $\mathbf{P}\cdot y$ contains
$e_{2\alpha_1+\alpha_2}$. As $2\alpha_1+\alpha_2$ is a short root of
$\Phi$, we can argue as before to obtain $5^4\mid \dim V$, a
contradiction.

As a result, $y=y_3$. Then $\chi=\chi_{-3}$ vanishes on
$\bigoplus_{i\ge -2}\,\mathcal{M}_i$ as stated. If $\chi\ne 0$, then
we can assume that $y=e_{3\alpha_1+2\alpha_2}$ (for all nonzero
elements in
$\mathfrak{g}_3=\big[e_{2\alpha_1+\alpha_2},\mathcal{M}_1\big]$ are
conjugate under the action of $\mathbf{L}$). Since
$\dim\mathfrak{c}_{\mathfrak{g}}(e_{3\alpha_1+2\alpha_2})=8$, it
follows from \cite[Thm.~I]{P95} that $5^3\mid\dim V$. Then $\dim
V=125$, completing the proof.
\end{proof}
\section{\bf Melikian pairs}
Set $\Gamma:=\Gamma(L,T)$. According to Theorem~\ref{sum}(4), if
$\alpha,\beta\in\Gamma$ are $\mathbb{F}_p$-independent, then either
$L[\alpha,\beta]\cong\mathcal{M}$ or $H(2;(2,1))^{(2)}\subset
L[\alpha,\beta]\subset H(2;(2,1))$. If $L[\alpha,\beta]\cong
{\mathcal M}$ we say that $(\alpha,\beta)\in\Gamma^2$ is a {\em
Melikian pair}. Recall from Theorem~\ref{sum}(2) that $H^3$ is a
$2$-dimensional subspace of $T$.
\begin{lemm}\label{melpair1}
A pair $(\alpha,\beta)\in\Gamma^2$ is Melikian if and only if
$H^3\cap\ker\alpha\ne H^3\cap\ker\beta$, i.e. if and only if
$\alpha_{\vert H^3}$ and $\beta_{\vert H^3}$ are linearly
independent over $F$.
\end{lemm}
\begin{proof}
Suppose $H(2;(2,1))^{(2)}\subset L[\alpha,\beta]\subset H(2;(2,1))$.
Recall from Sect.~2 that $H(2;(2,1))=H(2;(2,1))^{(2)}\oplus V$ and
$V^3=(0)$. Then $L[\alpha,\beta]^3\subset H(2;(2,1))^{(2)}$, forcing
$\Psi_{\alpha,\beta}(H)^3\subset H(2;(2,1))^{(2)}$. But then
$\Psi_{\alpha,\beta}(H^3)\subset\Psi_{\alpha,\beta}(T)\cap
H(2;(2,1))^{(2)}$ has dimension $\le 1$ by Lemma~\ref{tori}. In view
of Theorem~\ref{sum}(4) and the inclusion $H^3\subset T$, this means
that $H^3\cap\ker\alpha\cap\ker\beta$ has codimension $\ge 1$ in
$H^3$. It follows that $\alpha$ and $\beta$ are linearly dependent
as linear functions on $H^3$.

Now suppose that $L[\alpha,\beta]\cong \mathcal{M}$. In view of
Theorem~\ref{sum}(1), both $\alpha$ and $\beta$ are in $\Omega$.
Therefore, $\Psi_{\alpha,\beta}(T)$ is a nonstandard $2$-dimensional
torus in $L[\alpha,\beta]\cong \Der L[\alpha,\beta]$. Applying
\cite[Lemmas~4.1 \& 4.4]{P94} now gives $\dim
\Psi_{\alpha,\beta}(H)^3=2$, which in conjunction with
Theorem~\ref{sum}(5) yields that $H^3\cap\ker\alpha\cap\ker\beta$
has codimension $\le 2$ in $H^3$. So $\alpha$ and $\beta$ must be
linearly independent on $H^3$.
\end{proof}
\begin{cor}\label{melpair2}
For any $\alpha\in\Gamma$ there exists $\beta\in\Gamma$ such that
$(\alpha,\beta)$ is a Melikian pair.
\end{cor}
\begin{proof}
It follows from Theorem~\ref{sum} that $H^3\cap\ker\alpha=Ft$ for
some nonzero $t\in H^3$. Since $H^3\subset T$ and $L$ is centerless,
there is a $\beta\in\Gamma$ with $\beta(t)\ne 0$. Then
$(\alpha,\beta)$ is a Melikian pair by Lemma~\ref{melpair1}.
\end{proof}
\begin{lemm}\label{melpair3}
If $(\alpha,\beta)$ is a Melikian pair, then
$$L_p(\alpha,\beta)=L(\alpha,\beta)^{(1)}\oplus
T\cap\ker\alpha\cap\ker \beta,\quad\ \ L_p(\alpha,\beta)^{(1)}\,=\,
L(\alpha,\beta)^{(1)}\cong\,\mathcal{M}.$$
\end{lemm}
\begin{proof}
(a) Since $\rad_T\,L(\alpha,\beta)\,=\,H\cap
T\cap\ker\alpha\cap\ker\beta$ by Theorem~\ref{sum}(5), we have that
$\rad_T\,L(\alpha,\beta)\,=\,\mathfrak{z}(L(\alpha,\beta))$. Hence
$$(0)\,\longrightarrow \,H\cap T\cap\ker\alpha\cap
\ker\beta\,\longrightarrow\, L(\alpha,\beta)\longrightarrow
\mathcal{M}\,\longrightarrow\, (0)$$ is a central extension
$\mathcal M$. By Proposition~\ref{centralext}, this extension
splits; that is, $L(\alpha,\beta)\,=\,L(\alpha,\beta)^{(1)}\oplus
H\cap T\cap\ker\alpha\cap \ker\beta$ and
$L(\alpha,\beta)^{(1)}\cong\mathcal{M}$.

\smallskip

\noindent (b) Note that
$L_p(\alpha,\beta)\,=\,\widetilde{H}+L(\alpha,\beta)$, where
$\widetilde{H}=\mathfrak{c}_{L_p}(T)$, and
$\big[\widetilde{H},L(\alpha,\beta)^{(1)}\big]\subset
L(\alpha,\beta)^{(1)}$. Hence $\widetilde{H}$ acts on
$L(\alpha,\beta)^{(1)}$ as derivations. As all derivations of
$L(\alpha,\beta)^{(1)}\cong \mathcal{M}$ are inner by
\cite[Thm.~7.1.4]{St04}, it must be that
$\widetilde{H}=H'\oplus\widetilde{H}_0$, where
$\widetilde{H}_0=\mathfrak{c}_{\widetilde{H}}\big(L(\alpha,\beta)^{(1)}\big)$
and $H'=L(\alpha,\beta)^{(1)}\cap H$. From part~(a) of this proof it
follows that $H\subset T+H'$. Consequently, $[H,\widetilde{H}_0]=0$.

Put $\Gamma':=\{\gamma\,|\,\,\gamma(H')\ne 0\}$ and let $\mu$ be any
root in $\Gamma'$. Recall that $\dim L_\mu=5$; see
Theorem~\ref{sum}(3). As $H'$ is a nontriangulable Cartan subalgebra
of $L(\alpha,\beta)^{(1)}\cong\mathcal{M}$ by \cite[Lemmas~4.1 \&
4.4]{P94}, the $H'$-module $L_\mu$ is irreducible. But then
$\widetilde{H}_0$ acts on $L_\mu$ as scalar operators. On the other
hand, it follows from Schue's lemma \cite[Prop.~1.3.6(1)]{St04} that
$L$ is generated by the root spaces $L_\gamma$ with
$\gamma\in\Gamma'$. It follows that $\widetilde{H}_0$ acts
semisimply on $L$, implying $\widetilde{H}_0\subset T$. From this it
is immediate that
$\widetilde{H}_0\,=\,T\cap\ker\alpha\cap\ker\beta$. As a result,
$$L_p(\alpha,\beta)\,=\,L(\alpha,\beta)^{(1)}+\widetilde{H}_0\,=\,
L(\alpha,\beta)^{(1)}\oplus T\cap\ker\alpha\cap\ker\beta,
$$ finishing the proof.
\end{proof}

Let $(\alpha,\beta)$ be a Melikian pair. Note that
$T_0:=T\cap\ker\alpha\cap\ker\beta$ is a restricted ideal of
$L_p(\alpha,\beta)$ and $T=H^3\oplus T_0$. So the Lie algebra
$L_p(\alpha,\beta)/T_0$ inherits a $p$th power map from
$L_p(\alpha,\beta)$. Since $L_p(\alpha,\beta)/T_0\cong\,
\mathcal{M}$ by Lemma~\ref{melpair3} and both Lie algebras are
centerless and restricted, every isomorphism between
$L_p(\alpha,\beta)/T_0$ and $\mathcal{M}$ is an isomorphism of
restricted Lie algebras. Any such isomorphism maps the torus $T/T_0$
of the restricted Lie algebra $L_p(\alpha,\beta)/T_0$ onto a
$2$-dimensional nonstandard torus of $\mathcal M$. According to
\cite[Lemmas~4.1 \& 4.4]{P94}, any such torus is conjugate under
$\mathrm{Aut}\,\mathcal{M}$ to the torus
$\mathfrak{t}:=F(1+x_1)\partial_1\oplus F(1+x_2)\partial_2$.

Recall from Sect.~6 the natural grading of the Lie algebra $\mathcal
M$. For $i\ge -3$, we set $\mathcal{M}_{(i)}:=\bigoplus_{j\ge
i}\,\mathcal{M}_i$. The decreasing filtration
$\big(\mathcal{M}_{(i)}\big)_{i\ge -3}$ of the Lie algebra $\mathcal
M$ can be regarded as a standard (Weisfeiler) filtration of
$\mathcal M$ associated with its maximal subalgebra
$\mathcal{M}_{(0)}$. It is referred to as the {\it natural}
filtration of $\mathcal M$, because $\mathcal{M}_{(0)}$ is the only
subalgebra of codimension $5$ and depth $3$ in $\mathcal{M}$. All
components $\mathcal{M}_{(i)}$ of this filtration are invariant
under the automorphism group of $\mathcal M$; see
\cite[Thm.~4.3.3(2) and Rem.~4.3.4]{St04} for more detail. Note that
$\mathcal{M}\,=\,\mathfrak{t}\oplus\mathcal{M}_{(-2)}.$

Regard $\widetilde{\mathcal M}:=\mathcal{M}\oplus T_0$ as a direct
sum of Lie algebras and define a $p$th power map $u\mapsto u^p$ on
$\widetilde{\mathcal M}$ by setting $u^p=u^{[p]}$ for all
$u\in\mathcal{M}$ and $u^p=0$ for all $u\in T_0$ (here $u\mapsto
u^{[p]}$ is {\em the} $p$th power map on $\mathcal M$). The above
discussion in conjunction with Lemma~\ref{melpair3} shows that there
exists a Lie algebra isomorphism
\begin{equation}\label{Phi1}
\Phi\colon\,L_p(\alpha,\beta)\,\stackrel{\sim}{\longrightarrow}\,\,
\widetilde{\mathcal{M}}\,=\,\mathcal{M}_{(-2)}\oplus
\Phi(T)
\end{equation}
such that
\begin{equation}\label{Phi2}
\Phi\big(L(\alpha,\beta)^{(1)}\big)=\,\mathcal{M},\qquad
\Phi(H^3)\,=\,\mathfrak{t},\qquad
\Phi\vert_{T_0}\,=\,\mathrm{Id}_{T_0}.
\end{equation}
Note that $\Phi$ maps $L_p(\alpha,\beta)^{(1)}$ onto
$\widetilde{\mathcal M}^{(1)}=\mathcal{M}$. We stress that $H^3$ is
not a restricted subalgebra of $L_p(\alpha,\beta)$, whilst
$\Phi(H^3)$ is a maximal torus of $\widetilde{\mathcal M}$. There
exists a $p$-linear mapping $\Lambda\colon\,\widetilde{\mathcal
M}\longrightarrow\,\mathfrak{z}(\widetilde{\mathcal M})= T_0$ such
that
$$\Lambda(u)\,=\,\Phi^{-1}(u)^{[p]}-\Phi^{-1}(u^p)
\qquad\quad\quad(\forall\,u\in\widetilde{\mathcal M}),$$ where
$\Phi^{-1}(u)\longmapsto \Phi^{-1}(u)^{[p]}$ is the $p$th power map
in $L_p$.
\begin{lemm}\label{Lambda}
The $p$-linear mapping $\Lambda$ vanishes on the subspace
$\mathcal{M}_{(-2)}$ of $\widetilde{\mathcal M}$.
\end{lemm}
\begin{proof}
Suppose $\Lambda(u)\ne 0$ for some $u\in\mathcal{M}_{(-2)}$. Then
there is $\gamma\in\Gamma$ which does not vanish on $\Lambda(u)\in
T_0\setminus\{0\}$. Since $\Lambda(u)\subset
T\cap\ker\alpha\cap\ker\beta$, the root $\gamma$ is
$\mathbb{F}_p$-independent of $\alpha$ and $\beta$. Let
$M(\gamma;\alpha,\beta):=\bigoplus_{i,j\in\mathbb{F}_p}\,
L_{\gamma+i\alpha+j\beta}$. By Theorem~\ref{sum},
$M(\gamma;\alpha,\beta)$ is $125$-dimensional submodule of the
$\big(T+L(\alpha,\beta)_p\big)$-module $L$. The map
$\ad\circ\Phi^{-1}$ gives $M(\gamma;\alpha,\beta)$ an
$\mathcal{M}$-module structure. Note that $T_0$ acts on
$M(\gamma;\alpha,\beta)$ as scalar operators. This means that the
$\mathcal{M}$-module $M(\gamma;\alpha,\beta)$ has a $p$-character;
we call it $\chi$. It is straightforward to see that
$\Lambda(x)=\chi(x)^p$ for all $x\in\mathcal{M}$. But then $\chi$
does not vanish on $\mathcal{M}_{(-2)}$. Since $\dim
M(\gamma;\alpha,\beta)=125$, this contradicts
Proposition~\ref{p-character}. The result follows.
\end{proof}

We now set
$\big(L_p(\alpha,\beta)^{(1)}\big)_{(i)}:=\,\Phi^{-1}(\mathcal{M}_{(i)})$
for all $i\ge -3.$ Then the following hold:
\begin{itemize}
\item[$\bullet$] $\big(L_p(\alpha,\beta)^{(1)}\big)_{(-3)}=\,L_p(\alpha,\beta)^{(1)}$;
\item[$\bullet$] $\big(L_p(\alpha,\beta)^{(1)}\big)_{(0)}$ is a
subalgebra of codimension $5$ in $L_p(\alpha,\beta)^{(1)}$;
\item[$\bullet$] $u^{[p]}\in L_p(\alpha,\beta)^{(1)}$ for all
$u\in \big(L_p(\alpha,\beta)^{(1)}\big)_{(-2)}$;
\item[$\bullet$]
$\big(L_p(\alpha,\beta)^{(1)}\big)_{(0)}$ is a restricted subalgebra
of $L_p(\alpha,\beta)$.
\end{itemize}
Since the natural filtration of $\mathcal M$ is invariant under all
automorphisms of $\mathcal M$ (see \cite[Rem.~4.3.4(3)]{St04}), the
above definition of the subspaces
$\big(L_p(\alpha,\beta)^{(1)}\big)_{(i)}$ is independent of the
choice of $\Phi$ satisfying (\ref{Phi1}) and (\ref{Phi2}).
\section{\bf Describing $L_p(\alpha)$}
Fix $\alpha\in\Gamma$ and pick $\beta\in\Gamma$ be such that
$(\alpha,\beta)$ is a Melikian pair; see Corollary~\ref{melpair2}.
As before, we put $T_0:=T\cap\ker\alpha\cap\ker\beta$ and let $\Phi$
be a map satisfying (\ref{Phi1}) and (\ref{Phi2}). It gives rise to
the {\it restricted} Lie algebra isomorphism
$$\bar{\Phi}\colon\,L_p(\alpha,\beta)/T_0\stackrel{\sim}{\longrightarrow}\,
\mathcal{M}\,=\,\mathcal{M}_{(-2)}\oplus\bar{\Phi}(H^3),\qquad \
\bar{\Phi}(H^3)=\,\mathfrak{t}.$$

By Theorem~\ref{sum}(1), no root in $\Gamma$ vanishes on $H^3$. As
$\dim H^3=2$, there exists a nonzero $h_\alpha\in H^3$ such that
$Fh_\alpha=\,H^3\cap\ker\alpha$. As $\bar{\Phi}(Fh_\alpha)$ is a
$1$-dimensional subtorus of the nonstandard torus $\mathfrak{t}$, it
follows from \cite[Thm.~2.1]{Skr} that there is an automorphism of
$\mathcal M$ which maps $\mathfrak{t}$ onto itself and
$F\bar{\Phi}(h_\alpha)$ onto $F(1+x_1)\partial_1$. Hence we may
assume without loss of generality that
\begin{equation}\label{Phi3}
\Phi(L_p(\alpha))\,=\,\mathfrak{c}_{\mathcal
M}((1+x_1)\partial_1)\oplus T_0,\quad\ \Phi(T)=\mathfrak{t}\oplus
T_0,\quad\ \bar{\Phi}(h_\alpha)=(1+x_1)\partial_1.
\end{equation}
For $f\in \mathcal{O}(2;(1,1))_{(0)}$ set $f^{(k)}:=f^k/k!$ for
$0\le k\le 4$ and $f^{(k)}:=0$ for $k<0$ and $k\ge 5$. Direct
computations show that $\mathfrak{c}_{\mathcal
M}((1+x_1)\partial_1)$ has basis
$$\big\{x_2^{(r)}\partial_2,\,x_2^{(r)}(1+x_1)\partial_1,\,
x_2^{(r)}(1+x_1)^2,\,
x_2^{(r)}(1+x_1)^3\tilde{\partial}_2,\,x_2^{(r)}(1+x_1)^4\tilde{\partial}_1\,|\,\,\,0\le
r\le 4 \big\}.
$$
Using the multiplication table in \cite[(4.3.1)]{St04} it is easy to
observe that
\begin{eqnarray*}
\big[x_2^{(r)}\partial_2,\,x_2^{(s)}\partial_2\big]&=&\big[\textstyle{r+s-1\choose
r}-
\textstyle{r+s-1\choose s}\big]x_2^{(r+s-1)}\partial_2;\\
\big[x_2^{(r)}(1+x_1)\partial_1,\,x_2^{(s)}\partial_2\big]&=&
-\textstyle{r+s-1\choose s}x_2^{(r+s-1)}(1+x_1)\partial_1;\\
\big[x_2^{(r)}(1+x_1)\partial_1,\,x_2^{(s)}(1+x_1)\partial_1\big]&=&0;\\
\big[x_2^{(r)}(1+x_1)^2,\,x_2^{(s)}\partial_2\big]&=&-\big[\textstyle{r+s-1\choose
s}-
2\textstyle{r+s-1\choose s-1}\big]x_2^{(r+s-1)}(1+x_1)^2;\\
\big[x_2^{(r)}(1+x_1)^2,\,x_2^{(s)}(1+x_1)\partial_1\big]&=&-\big[2\textstyle{r+s\choose
s}-
2\textstyle{r+s\choose s}\big]x_2^{(r+s)}(1+x_1)^2\,=\,0;\\
\big[x_2^{(r)}(1+x_1)^2,\,x_2^{(s)}(1+x_1)^2\big]
&=&2\big[-\textstyle{r+s-1\choose r}+
\textstyle{r+s-1\choose s}\big]x_2^{(r+s-1)}(1+x_1)^4\tilde{\partial}_1;\\
\big[x_2^{(r)}(1+x_1)^3\tilde{\partial}_2,\,x_2^{(s)}\partial_2\big]&=&-\textstyle{r+s\choose
r} x_2^{(r+s-1)}(1+x_1)^3\tilde{\partial}_2;\\
\big[x_2^{(r)}(1+x_1)^3\tilde{\partial}_2,\,x_2^{(s)}(1+x_1)\partial_1\big]&=&
\textstyle{r+s-1\choose r}x_2^{(r+s-1)}(1+x_1)^4\tilde{\partial}_1;\\
\big[x_2^{(r)}(1+x_1)^3\tilde{\partial}_2,\,x_2^{(s)}(1+x_1)^2\big]&=&-\textstyle{r+s\choose
r} x_2^{(r+s)}\partial_2;\\
\big[x_2^{(r)}(1+x_1)^3\tilde{\partial}_2,\,x_2^{(s)}(1+x_1)^3\tilde{\partial}_2\big]&=&
0;\\
\big[x_2^{(r)}(1+x_1)^4\tilde{\partial}_1,\,x_2^{(s)}\partial_2\big]&=&
-\big[\textstyle{r+s-1\choose s}+2
\textstyle{r+s-1\choose s-1}\big]x_2^{(r+s-1)}(1+x_1)^4\tilde{\partial}_1;\\
\big[x_2^{(r)}(1+x_1)^4\tilde{\partial}_1,\,x_2^{(s)}(1+x_1)\partial_1\big]&=&
-\big[3\textstyle{r+s\choose r}+2\textstyle{r+s\choose
s}\big]x_2^{(r+s)}\tilde{\partial}_1\,=\,0;\\
\big[x_2^{(r)}(1+x_1)^4\tilde{\partial}_1,\,x_2^{(s)}(1+x_1)^2\big]&=&-\textstyle{r+s\choose
r}x_2^{(r+s)}(1+x_1)\partial_1;\\
\big[x_2^{(r)}(1+x_1)^4\tilde{\partial}_1,\,x_2^{(s)}(1+x_1)^3\tilde{\partial}_2\big]&=&\textstyle{r+s\choose
r} x_2^{(r+s)}(1+x_2)^2;\\
\big[x_2^{(r)}(1+x_1)^4\tilde{\partial}_1,\,x_2^{(s)}(1+x_1)^4\tilde{\partial}_1\big]&=&0.
\end{eqnarray*}

In order to obtain a more invariant description of $L_p(\alpha)$ we
now consider a vector space $R=R'\oplus C$ over $F$ with $\dim
C=\dim T-2$ such that $R'$ has basis
$\big\{x_1^{(i)}x_2^{(j)}\,|\,\,\,0\le i,j\le 4,\,\,1\le i+j\le
7\big\}\cup\{x_2^{(5)}\}\cup\{z\}$. We  give $R$ a Lie algebra
structure  by setting
$$
\big[x_1^{(i)}x_2^{(j)},\,x_1^{(k)}x_2^{(l)}\big]:=\,
\Big[\textstyle{{i+k-1\choose i-1}{j+l-1\choose j}-{i+k-1\choose
i}{j+l-1\choose j-1}}\Big]x_1^{(i+k-1)}x_2^{(j+l-1)}$$ for all $
i,j,k,l$ with $3\le i+j+k+l\le 7$ such that $(j,l)\ne (0,0)$
whenever $i+k=5$, and by requiring that $[Fz+C,R]=0$ and
$$
\big[x_1^{(i)}x_2^{(j)},\,x_1^{(k)}x_2^{(l)}\big]:=\,\left\{
\begin{array}{ll}
0 &\quad \mbox{ if  } i+j+k+l\le 2,\\
(-1)^i z &\quad \mbox{ if  } j=l=0 \mbox{ and } i+k=5.
\end{array}
\right.
$$
The Lie algebra $R$ is a (nonsplit) central extension of
$H(2;\un{1})^{(2)}\oplus FD_H(x_2^{(5)})$. Computations show that
\begin{eqnarray*}
\big[x_1x_2^{(r)},\,x_1x_2^{(s)}]&=&\big[\textstyle{r+s-1\choose r}-
\textstyle{r+s-1\choose s}\big]x_1x_2^{(r+s-1)};\\
\big[-x_1^{(4)}x_2^{(r-1)},\,x_1x_2^{(s)}\big]&=&\left\{
\begin{array}{ll}
-\textstyle{r+s-1\choose s}\big(-x_1^{(4)}x_2^{(r+s-2)}\big)&\ \mbox{ if }\,r+s\ge 2,\\
-z&\ \mbox{ if }\,r=1,\,\,s=0;
\end{array}
\right.\\
\big[-x_1^{(4)}x_2^{(r-1)},\,-x_1^{(4)}x_2^{(s-1)}\big]&=&0;\\
\big[x_1^{(2)}x_2^{(r)},\,x_1x_2^{(s)}\big]&=&
-\big[\textstyle{r+s-1\choose s}-
2\textstyle{r+s-1\choose s-1}\big]x_1^{(2)}x_2^{(r+s-1)};\\
\big[x_1^{(2)}x_2^{(r)},\,-x_1^{(4)}x_2^{(s-1)}\big]&=&0;\\
\big[x_1^{(2)}x_2^{(r)},\,x_1^{(2)}x_2^{(s)}\big]
&=&2\big[-\textstyle{r+s-1\choose r}+
\textstyle{r+s-1\choose s}\big]x_1^{(3)}x_2^{(r+s-1)};\\
\big[x_2^{(r+1)},\,x_1x_2^{(s)}\big]&=&-\textstyle{r+s\choose
r} x_2^{(r+s)};\\
\big[x_2^{(r+1)},\,-x_1^{(4)}x_2^{(s-1)}\big]&=&
\textstyle{r+s-1\choose r}x_1^{(3)}x_2^{(r+s-1)};\\
\big[x_2^{(r+1)},\,x_1^{(2)}x_2^{(s)}\big]&=&-\textstyle{r+s\choose
r}x_1x_2^{(r+s)};\\
\big[x_2^{(r+1)},\,x_2^{(s+1)}\big]&=&0;\\
\big[x_1^{(3)}x_2^{(r)},\,x_1x_2^{(s)}\big]&=&
-\big[\textstyle{r+s-1\choose s}+2 \textstyle{r+s-1\choose
s-1}\big]x_1^{(3)}x_2^{(r+s-1)};\\
\big[x_1^{(3)}x_2^{(r)},\,-x_1^{(4)}x_2^{(s-1)}\big]&=&0;\\
\big[x_1^{(3)}x_2^{(r)},\,x_1^{(2)}x_2^{(s)}\big]&=&\left\{
\begin{array}{ll}-\textstyle{r+s\choose
r}\big(-x_1^{(4)}x_2^{(r+s-1)}\big)&\ \mbox{ if }\,r+s\ge 1,\\
-z&\ \mbox{ if }\,r=s=0;
\end{array}
\right.\\
\big[x_1^{(3)}x_2^{(r)},\,x_2^{(s+1)}\big]&=&\textstyle{r+s\choose
r} x_1^{(2)}x_2^{(r+s)}\\
\big[x_1^{(3)}x_2^{(r)},\,x_1^{(3)}x_2^{(s)}\big]&=&0.
\end{eqnarray*}
By comparing the displayed multiplications tables it is
straightforward to see that the following statement holds:
\begin{prop}\label{iso}
Any linear map $\Theta'\colon\,\mathfrak{c}_{\widetilde{\mathcal
M}}((1+x_1)\partial_1)\longrightarrow\,R$ which takes $T_0$
isomorphically onto $C$ and satisfies the conditions
\begin{eqnarray*}
\Theta'\big(x_2^{(r)}(1+x_1)\partial_1\big)&=&\left\{\begin{array}{ll}
-x_1^{(4)}x_2^{(r-1)}& \,\,\,\mbox{ \rm{if } } 1\le r\le 4,\\
z& \,\,\,\mbox{ \rm{if } } r=0,
\end{array}
\right.\\
\Theta'\big(x_2^{(r)}\partial_2\big)&=&x_1x_2^{(r)},\qquad\quad\,\,\, 0\le r\le 4,\\
\Theta'\big(x_2^{(r)}(1+x_1)^2\big)&=&x_1^{(2)}x_2^{(r)},\qquad\quad 0\le r\le 4,\\
\Theta'\big(x_2^{(r)}(1+x_1)^3\tilde{\partial}_2)\big)&=&x_2^{(r+1)},\qquad\quad\,\,\, 0\le r\le 4,\\
\Theta'\big(x_2^{(r)}(1+x_1)^4\tilde{\partial}_1\big)&=&x_1^{(3)}x_2^{(r)},\qquad \quad 0\le r\le 4,\\
\end{eqnarray*}
is an isomorphism of Lie algebras.
\end{prop}

We now fix $\Theta'$ described in Proposition~\ref{iso} and set
$\Theta:=\Theta'\circ\Phi_{\vert L_p(\alpha)}$, where
$\Phi\colon\,L_p(\alpha,\beta)\stackrel{\sim}{\longrightarrow}\widetilde{\mathcal
M}$ is a Lie algebra isomorphism satisfying (\ref{Phi1}),
(\ref{Phi2}) and (\ref{Phi3}). Clearly,
$\Theta\colon\,L_p(\alpha)\stackrel{\sim}{\longrightarrow}\,R$ is a
Lie algebra isomorphism. We give $R$ a $p$th power map by setting
\begin{equation}\label{p-R}
r^p:=\,\Theta\big(\Theta^{-1}(r)^{[p]}\big)\qquad\quad\quad
(\forall\,r\in R).\end{equation} This turns $\Theta$ into an
isomorphism of restricted Lie algebras. Because the $p$-linear map
$\Lambda\colon\,\widetilde{\mathcal M}\longrightarrow\,T_0$ vanishes
on the subspace $\mathcal{\mathcal M}_{(-2)}$ of
$\widetilde{\mathcal M}$ by Lemma~\ref{Lambda} and $\Theta$ is
defined via $\Phi$, the explicit description of $\Theta'$ in
Proposition~\ref{iso} shows that the map (\ref{p-R}) has the
following properties:
\begin{eqnarray}\label{p-R-prop}
\big(x_2^{(r+1)}\big)^p=0\qquad&\mathrm{if}&\, 0\le r\le 4;
\nonumber\\ \big(x_1x_2^{(r)}\big)^p=0\qquad&\mathrm{if}&\, r\ne
0,1; \nonumber\\
\big(x_1^{(2)}x_2^{(r)}\big)^p=0\qquad&\mathrm{if}&\, 0\le r\le 4;\label{pth power}\\
\big(x_1^{(3)}x_2^{(r)}\big)^p=0\qquad&\mathrm{if}&\, 0\le r\le 4;
\nonumber\\
\big(x_1^{(4)}x_2^{(r-1)}\big)^p=0\qquad&\mathrm{if}&\, 1\le r\le 4;\nonumber\\
\big(x_1x_2\big)^p=x_1x_2\quad&\mathrm{\ \,i.e.}&\  x_1x_2\
\mathrm{is\ toral}\nonumber
\end{eqnarray}
(we refer to \cite{Skr} for more detail on the $p$-structure in the
restricted Melikian algebra). Note that $(x_1)^p$ and $z^p$ lie in
$\Theta(T)=Fz\oplus C$. Moreover,
$Fz=\Theta\big(H^3\cap\ker\alpha\big)$ coincides the image of
$F(1+x_1)\partial_1$ under $\Phi^{-1}$ and
$\Theta'\big((1+x_2)\partial_2)\big)=x_1+x_1x_2$.

We stress that all constructions of Sections~7 and 8 depend on the
choice of a Melikian pair.
\section{\bf The subalgebra $Q(\alpha)$}
The results obtained so far apply to all nonstandard tori of maximal
dimension in $L_p$. However, such tori need not be conjugate under
the automorphism group of $L$. In order to identify $L$ with one of
the Melikian algebras, we will require a sufficiently generic
nonstandard torus of maximal dimension in $L_p$.
\begin{prop}\label{generic}
There exists a nonstandard torus $T'$ of maximal dimension in $L_p$
for which $(\mathfrak{c}_L(T'))^3$ contains no nonzero toral
elements of $L_p$.
\end{prop}
\begin{proof}
Let $T$ and $\Gamma$ be as Sect.~8 and let
$(\alpha,\beta)\in\Gamma^2$ be a Melikian pair. Choose an
isomorphism
$\Phi\colon\,L_p(\alpha,\beta)\stackrel{\sim}{\longrightarrow}
\widetilde{\mathcal M}$ satisfying (\ref{Phi1}) and (\ref{Phi2}).
Then $H^3=\Phi^{-1}(\mathfrak{t})$. Set
$q_i:=\Phi^{-1}(x_i\partial_i),\,$ $n_i=\Phi^{-1}(\partial_i)$ and
$h_i:=n_i^{[p]}$, where $i=1,2$. As the elements $x_i\partial_i$ are
toral in $\mathcal M$, Lemma~\ref{Lambda} says that both $q_1$ and
$q_2$ are toral elements of $L_p$. Note that $T=F(q_1+n_1)\oplus
F(q_2+n_2)\oplus T_0$, where $T_0=T\cap\ker\alpha\cap\ker\beta$.

As $\Phi$ is a Lie algebra isomorphism, it is straightforward to see
that $[q_i,n_i]=-n_i$ and $h_i\in T_0$ for $i=1,2$. So it follows
from Jacobson's formula that
$(q_i+n_i)^{[p]^k}=q_i+n_i+\sum_{j=0}^{k-1}\,h_i^{[p]^j}$ for all
$k\ge 1$. Since $(H^3)_p=T$ by Theorem~\ref{sum}(3) and
$H^3=F(q_1+n_1)\oplus  F(q_2+n_2)$, it follows that the $p$-closure
of $Fh_1+ Fh_2$ coincides with $T_0$.

Recall that $\dim T_0\ge 1$. Let $\{t_1,\ldots, t_s\}$ be a basis of
$T_0$ consisting of toral elements of $L_p$. For
$x=\sum_{j=1}^s\,\alpha_jt_j\in T_0$ define
$\mathrm{Supp}(x):=\,\{j\,|\,\,\alpha_j\ne 0\}.$ Write
$h_1=\sum_{j=1}^s\,\lambda_it_i$ and $h_2=\sum_{j=1}^s\,\mu_jt_j$
with $\lambda_j,\mu_j\in F$. Since the $[p]$-th powers of $h_1$ and
$h_2$ span $T_0$, it must be that
\begin{eqnarray*}
\mathrm{Supp}(h_1)\cup\mathrm{Supp}(h_2)=\{1,\ldots, s\}.
\end{eqnarray*}
In particular, $h_1\ne 0$ or $h_2\ne 0$. Recall from Sect.~6 the
maximal torus $\mathbf{T}$ of the group
$\mathrm{Aut}_0\,\mathcal{M}$ of all automorphisms of $\mathcal M$
preserving the natural grading of $\mathcal M$. For every
$\sigma\in\mathrm{Aut}_0\,\mathcal{M}$ the subalgebra
$\Phi^{-1}\big(\sigma(\mathfrak{t})+T_0\big)$ is a nonstandard torus
of maximal dimension in $L_p$ and the elements
$(\Phi^{-1}\circ\sigma)(x_1\partial_1)$ and
$(\Phi^{-1}\circ\sigma)(x_2\partial_2)$  are toral in $L_p$ by
Lemma~\ref{Lambda}. Since the group $\mathrm{Aut}_0\,\mathcal{M}$
acts transitively on the set of bases of $\mathcal{M}_{-3}$, there
is $\tau\in\mathrm{Aut}_0\,\mathcal{M}$ such that the elements
$\big((\Phi^{-1}\circ\tau)(\partial_1)\big)^{[p]}$ and
$\big(\Phi^{-1}\circ\tau)(\partial_2)\big)^{[p]}$ are  both nonzero.
Replacing $\mathfrak t$ by $\tau(\mathfrak{t})$ and renumbering the
$t_i$'s if necessary, we thus may assume that $\lambda_1$ and
$\mu_1$ are both nonzero.

Since $F$ is infinite, there exist $a,b\in F^\times$ such that the
elements $a^p\lambda_1$ and $b^p\mu_1$ of $F$ are linearly
independent over $\mathbb{F}_p$. Applying a suitable automorphism
from the subgroup $\mathbf{T}$ of $\mathrm{Aut}_0\,\mathcal{M}$ one
observes that $\mathfrak{t}':=F(a+x_1)\partial_1\oplus
F(b+x_2)\partial_2,$ is a $2$-dimensional nonstandard torus in
$\mathcal M$ and $\mathfrak{t}'\,=\,(\mathfrak{c}_{\mathcal
M}(\mathfrak{t}'))^3$ (alternatively, one can apply \cite[Lemmas~4.1
\& 4.4]{P94}). This entails that
$$T':=\,\Phi^{-1}(\mathfrak{t}'\oplus T_0)\,=\,
F(q_1+an_1)\oplus F(q_2+bn_2)\oplus T_0$$ is a nonstandard torus of
maximal dimension in $L_p$ with $F(q_1+an_1)\oplus
F(q_2+bn_2)\,=\,(\mathfrak{c}_L(T'))^3$. Suppose
\begin{eqnarray}\label{[p]}
\big(x(q_1+an_1)+y(q_2+bn_2)\big)^{[p]}=\,x(q_1+an_1)+y(q_2+bn_2)\end{eqnarray}
for some $x,y\in F$. Applying $\Phi$ to both sides of (\ref{[p]})
gives
\begin{equation*}
\big(x(a+x_1)\partial_1+y(b+x_2\partial_2)\big)^{[p]}\,=\,x(a+x_1)\partial_1+y(b+x_2)\partial_2.
\end{equation*}
As both $(a+x_1)\partial_1$ and $(b+x_2)\partial_2$ are toral
elements of $\mathcal M$, we get $x,y\in\mathbb{F}_p$. Hence
\begin{eqnarray*}
x(q_1+an_1)+y(q_2+bn_2)&=&
\big(x(q_1+an_1)+y(q_2+bn_2)\big)^{[p]}\\
&=&x(q_1+an_1+a^ph_1) + y(q_2+bn_2+b^ph_2),\end{eqnarray*} implying
$xa^ph_1+yb^ph_2=0$. As a consequence, $xa^p\lambda_j+yb^p\mu_j=0$
for all $j\le s$. But then $a^p\lambda_1$ and $b^p\mu_1$ are
linearly dependent over $\mathbb{F}_p$, a contradiction. We conclude
that $(\mathfrak{c}_L(T'))^3$ contains no nonzero toral elements of
$L_p$.
\end{proof}
Retain the notation introduced in Sections~7 and 8. In view of
Proposition~\ref{generic}, we may assume that for every
$\alpha\in\Gamma$ no nonzero element of $H^3\cap\ker\alpha$ is toral
in $L_p$.

The map
$\Theta\colon\,L_p(\alpha)\stackrel{\sim}{\longrightarrow}\,R$
defined in Sect.~8 induces a natural Lie algebra isomorphism
$$\bar{\Theta}\colon\,L_p(\alpha)/\mathfrak{z}(L_p(\alpha))
\stackrel{\sim}{\longrightarrow}\,R/\mathfrak{z}(R)\,\cong\,
H(2;\un{1})^{(2)}\oplus FD_H(x_2^{(5)}).$$ Let
$\big(R/\mathfrak{z}(R)\big)_{(i)}$ denote the $i$th component of
the standard filtration of the Cartan type Lie algebra
$R/\mathfrak{z}(R)$, where $i\ge -1$, and denote by
$L_p(\alpha)_{(i)}$ the inverse image of
$\big(R/\mathfrak{z}(R)\big)_{(i)}$ under $\bar{\Theta}$. We thus
obtain a filtration $\{L_p(\alpha)_{(i)}\,|\,\,i\ge -1\}$ of the Lie
algebra $L_p(\alpha)$ with $\bigcap_{i\ge
-1}L_p(\alpha)_{(i)}\,=\,T\cap\ker\alpha$ and $\dim
\big(L_p(\alpha)/L_p(\alpha)_{(0)}\big)=2$. This filtration is, in
fact, independent of the choice of $\bar{\Theta}$, because
$\big(R/\mathfrak{z}(R)\big)_{(0)}$ is the unique subalgebra of
codimension $2$ in the Cartan type Lie algebra $R/\mathfrak{z}(R)$.
Since $\bar{\Theta}$ is a restricted Lie algebra isomorphism, all
$L_p(\alpha)_{(i)}$ are restricted subalgebras of $L_p(\alpha)$. We
denote by $\mathrm{nil}_{[p]}\big(L_p(\alpha)_{(i)}\big)$ the
maximal ideal of $L_p(\alpha)_{(i)}$ consisting of $p$-nilpotent
elements of $L_p$.
\begin{defi}\label{defi}
{\rm Define}
\begin{eqnarray*}
W&:=&\big\{u\in L_p(\alpha)^{(1)}\cap L_p(\alpha)_{(0)}\,|\,\,\,
u^{[p]}\in L_p(\alpha)^{(1)}\big\};\\
P&:=&\big\{u\in W\,|\,\,\,[u,W]\subset W\big\};\\
Q(\alpha)&:=&P+\mathrm{nil}_{[p]}\big(L_p(\alpha)_{(3)}\big).
\end{eqnarray*}
\end{defi}
Because of the uniqueness of the filtration
$\{L_p(\alpha)_{(i)}\,|\,\,i\ge -1\}$ this definition is independent
of the choices made earlier. The main result of this section is the
following:
\begin{prop}\label{Q-alpha}
If $(\alpha,\beta)$ is a Melikian pair in $\Gamma^2$, then
$$Q(\alpha)\,=\,L_p(\alpha)\cap\big(L_p(\alpha,\beta)^{(1)}\big)_{(0)}.$$
\end{prop}
\begin{proof}
(a) Choose any Lie algebra isomorphism
$\Phi\colon\,L_p(\alpha,\beta)\stackrel{\sim}{\longrightarrow}\,\widetilde{\mathcal
M}\,=\,\mathcal{M}\oplus T_0$ satisfying (\ref{Phi1}), (\ref{Phi2})
and (\ref{Phi3}). Then
$\Phi\big(L_p(\alpha)\cap\big(L_p(\alpha,\beta)^{(1)}\big)_{(0)}\big)$
is spanned by
$$\big\{x_2^{(r)}\partial_2,\,x_2^{(r)}(1+x_1)\partial_1,\,x_2^{(r)}(1+x_1)^2,\,
x_2^{(r)}(1+x_1)^3\tilde{\partial}_2,\,x_2^{(r)}(1+x_1)^4\tilde{\partial}_1\,|\,\,\,1\le
r\le 4\big\}.$$ Let
$\Theta=\Phi\circ\Theta'\colon\,L_p(\alpha)\stackrel{\sim}{\longrightarrow}\,R$
be the isomorphism associated with $\Phi$. The explicit formulae for
$\Theta'$ yield that
$\Theta\big(L_p(\alpha)\cap\big(L_p(\alpha,\beta)^{(1)}\big)_{(0)}\big)$
is spanned by the set
$$
\big\{x_1x_2^{(r)},\, x_1^{(2)}x_2^{(r)},\,
x_1^{(3)}x_2^{(r)}\,|\,\,\,1\le r\le 4\big\}\cup
\big\{x_1^{(4)}x_2^{(r)}\,|\,\,\,0\le r\le
3\big\}\cup\big\{x_2^{(r)}\,|\,\,\, 2\le r\le 5\big\};
$$
see Proposition~\ref{iso}.

\smallskip

\noindent (b) Next we are going to determine $\Theta(W)$,
$\Theta(P)$ and $\Theta(Q(\alpha))$ by using Definition~\ref{defi}.
First we observe that
$$\Theta\big(L_p(\alpha)^{(1)}\cap
L_p(\alpha)_{(0)}\big)\,=\,Fz\oplus\big(\textstyle{\bigoplus}_{0\le
i,j\le 4,\,2\le i+j\le 7}\,Fx_1^{(i)}x_2^{(j)}\big);$$ see
Proposition~\ref{iso}. It is immediate from equations~(\ref{pth
power}) that
$$\big(x_1^{(i)}x_2^{(j)}\big)^p\in\,\Theta\big(L_p(\alpha)^{(1)}\cap
L_p(\alpha)_{(0)}\big)\ \quad\mbox{whenever }\, i+j\ge 2.$$ Recall
that $\Theta$ is an isomorphism of restricted Lie algebras. In
conjunction with Jacobson's formula, this shows that $\Theta(W)$ is
a subspace of $R$. As a consequence, we have the inclusion
$$\textstyle{\bigoplus}_{0\le i,j\le 4,\,2\le i+j\le
7}\,Fx_1^{(i)}x_2^{(j)}\subset\,\Theta(W).$$ On the other hand, if
$z\in \Theta(W)$, then the definition of $\Theta'$ and our
assumption on $\Phi$ yield $H^3\cap\ker\alpha\subset\, W$. Then
$h_\alpha\in W$. As $Fh_\alpha=H^3\cap\ker\alpha= F\Theta^{-1}(z)$,
our assumption on $h_\alpha$ in (\ref{Phi3}) yields
$h_\alpha=\Phi^{-1}((1+x_1)\partial_1)$. It follows that
$h_\alpha^{[p]}-h_\alpha\in L_p(\alpha)^{(1)}\cap T_0$. As
$h_\alpha^{[p]}\ne h_\alpha$ by our choice of $T$, this entails
$L_p(\alpha,\beta)^{(1)}\cap T_0\ne (0)$ contradicting
Lemma~\ref{melpair3}. We conclude that
$$
\Theta(W)\,=\,\textstyle{\bigoplus}_{0\le i,j\le 4,\,2\le i+j\le
7}\,Fx_1^{(i)}x_2^{(j)}.
$$

Let $u=\sum_{i,j}s_{i,\,j}\,x_1^{(i)}x_2^{(j)}\in\Theta(P)$. Since
$x_1^{(2)},\, x_1^{(3)}\in\Theta(W)$ and $[x_1^{(2)},x_1^{(3)}]=z$,
it follows readily from the definition of $P$ that
$s_{2,\,0}=s_{3,\,0}=0$. The multiplication table for $R$ given
Sect.~8 now shows that $\Theta(P)$ is spanned by
$$\big\{x_1^{(4)},\,x_2^{(2)},\,x_2^{(3)}\big\}\cup
\big\{x_1^{(i)}x_2,\,x_1^{(i)}x_2^{(2)},\,x_1^{(i)}x_2^{(3)}\,|\,\,\,
1\le i\le 4\big\} \cup\big\{x_1^{(i)}x_2^{(4)}\,|\,\,\, 0\le i\le
3\big\}.
$$

\smallskip

\noindent (c) Finally, the nilpotent subalgebra
$\Theta\big(L_p(\alpha)_{(3)}\big)$ is spanned by
$$
\big\{x_1^{(i)}x_2^{(4)}\,|\,\,\,0\le i,j\le 4;\,5\le i+j\le
7\big\}\cup\big\{x_2^{(5)},\,z\big\}\cup C.
$$
By (\ref{pth power}), the Lie product of any two elements in this
set is $p$-nilpotent in $R$. Since $\Theta$ is an isomorphism of
restricted Lie algebras, it follows that
$\Theta\big(\mathrm{nil}_{[p]}\big(L_p(\alpha)_{(3)}\big)\big)$ is
spanned by $\big\{x_1^{(i)}x_2^{(4)}\,|\,\,\,0\le i,j\le 4;\,5\le
i+j\le 7\big\}\cup\big\{x_2^{(5)}\big\}.$ Comparing the spanning set
of
$\Theta\big(L_p(\alpha)\cap\big(L_p(\alpha,\beta)^{(1)}\big)_{(0)}\big)$
from part~(a) of this proof with that of
$\Theta(Q(\alpha))\,=\,\Theta(P)+\Theta\big(\mathrm{nil}_{[p]}\big(L_p(\alpha)_{(3)}\big)\big)$
we now obtain that
$$
\Theta\big(L_p(\alpha)\cap\big(L_p(\alpha,\beta)^{(1)}\big)_{(0)}\big)
\,=\,\Theta(Q(\alpha)).
$$
Since $\Theta$ is an isomorphism, the proposition follows.
\end{proof}

\begin{rem}\label{sub}
{\rm Proposition~\ref{Q-alpha} implies that $Q(\alpha)$ is a
subalgebra of $L(\alpha)$.}
\end{rem}

At the end of Sect.~8 we mentioned that
$\Theta'((1+x_2)\partial_2)=x_1(1+x_2)$. In what follows we require
some computations in the subalgebra $\Theta(H)\subset
\mathfrak{c}_{R}(x_1(1+x_2))$. It follows from the multiplication
table for $R$ that $\mathfrak{c}_{R}(x_1(1+x_2))$ contains
$x_1^{(2)}(1+x_2)^2$ and $x_1^{(3)}(1+x_2)^3$. Set
$w:=x_2-x_2^{(2)}+2x_2^{(3)}-x_2^{(4)}-x_2^{(5)}$ and observe that
\begin{eqnarray}
[x_1(1+x_2),w]&=&[x_1,w]+[x_1x_2,w]\,=\,\big(-x_2+2x_2^{(2)}-x_2^{(3)}-x_2^{(4)}\big)\label{w1}\\
&+&\Big(x_2-\textstyle{2\choose 1}x_2^{(2)}+ 2\textstyle{3\choose
1}x_2^{(3)}-\textstyle{4\choose 1}x_2^{(4)}\Big)\,=\,0.\nonumber
\end{eqnarray}
Applying Proposition~\ref{iso} it is now easy to see that
$$\textstyle{\bigoplus}_{i=1}^3\,Fx_1^{(i)}(1+x_2)^{i}\oplus Fw\oplus
Fz\subset\,\Theta\big(L_p(\alpha,\beta)^{(1)}\cap
H_p\big)\subset\,\Theta(H).$$  Direct computations show that
\begin{eqnarray}
[x_1^{(2)}(1+x_2)^2,w]&=&x_1(1+x_2)^2(1-x_2+2x_2^{(2)}-x_2^{(3)}-x_2^{(4)})\label{w2}\\
&=&x_1(1+x_2)^6\,=\,x_1(1+x_2);\nonumber\\
{}[x_1^{(3)}(1+x_2)^3,w]&=&x_1^{(2)}(1+x_2)^3(1-x_2+2x_2^{(2)}-x_2^{(3)}-x_2^{(4)})\label{w3}\\
&=&x_1^{(2)}(1+x_2)^7\,=\,x_1^{(2)}(1+x_2)^2.\nonumber
\end{eqnarray}
\begin{prop}\label{linear}
Let $\alpha$ be an arbitrary root of $\Gamma$. Then for any
$r\in\mathbb{F}_p^\times$ there exists a linear map
$l_{r\alpha}\colon\,L_{r\alpha}\rightarrow H$ such that
$x-l_{r\alpha}(x)\in Q(\alpha)$ for all $x\in L_{r\alpha}$.
Furthermore, $H\cap Q(\alpha)\,=\,(0)$ and
$L(\alpha)\,=\,H+Q(\alpha).$
\end{prop}
\begin{proof}
In order to perform computations in $L_p(\alpha)$ we are going to
invoke the isomorphism $\Theta=\Theta'\circ\Phi$; see
Proposition~\ref{iso}. Recall that
$$\Theta(T)\,=\,Fx_1(1+x_2)\oplus Fz\oplus C.$$
Replacing $\alpha$ by an $\mathbb{F}_p^\times$-multiple of $\alpha$,
if necessary, we may assume that $\alpha((x_1(1+x_2))=1$. Using the
multiplication table for $R$ it is then straightforward to see that
$$\Theta(L_{r\alpha})\,=\,\textstyle{\bigoplus}_{i=1}^3\,Fx_1^{(i)}(1+x_2)^{r+i}
\oplus F\big(x_1^{(4)}(1+x_2)^{r-1}-r^{-1}z\big)\oplus
F\big((1+x_2)^r-1\big)$$ for all $r\in\mathbb{F}_p^\times$ and that
$\Theta(H)$ is sandwiched between
$\textstyle{\bigoplus}_{i=1}^3\,Fx_1^{(i)}(1+x_2)^{i}\oplus Fw\oplus
Fz$ and
$\Theta(H_p)\,=\,\textstyle{\bigoplus}_{i=1}^3\,Fx_1^{(i)}(1+x_2)^{i}\oplus
Fw\oplus Fz\oplus C.$ We now define a linear map
$l_{r\alpha}\colon\,L_{r\alpha}\rightarrow H$ by the formula
$l_{r\alpha}=\Theta^{-1}\circ m_r\circ\Theta$, where $m_{r}$ is the
linear map from $\Theta(L_{r\alpha})$ into $\Theta(H)$ given by
\begin{eqnarray*}
m_r\big(x_1^{(4)}(1+x_2)^{r-1}-r^{-1}z\big)&=&-r^{-1}z;\\
m_r\big(x_1^{(i)}(1+x_2)^{r+i}\big)&=&x_1^{(i)}(1+x_2)^i,\qquad\quad
1\le
i\le 3;\\
m_r\big((1+x_2)^r-1\big)&=&rw.
\end{eqnarray*}
Using the spanning set of $\Theta(Q(\alpha))$ from the proof of
Proposition~\ref{Q-alpha} one observes that
$w-x_2\in\Theta(Q(\alpha))$ and $x_1^{(i)}(1+x_2)^i-x_1^{(i)}\in
\Theta(Q(\alpha))$ for $1\le i\le 3$. By the same token, one finds
that the subspace $\bigoplus_{i=1}^3\, Fx_1^{(i)}\oplus Fx_2\oplus
Fz\oplus C$ of $R$ complements $\Theta(Q(\alpha))$. Since
$x_1^{(4)}(1+x_2)^{r-1}\in \Theta(Q(\alpha))$ for all
$r\in\mathbb{F}_p^\times$, this implies that $y-m_r(y)\in
\Theta(Q(\alpha))$ for all $y\in\Theta(L_{r\alpha})$ and
$R=\Theta(H_p)\oplus \Theta(Q(\alpha))$.

As a result, $x-l_{r\alpha}(x)\in Q(\alpha)$ for all
$r\in\mathbb{F}_p^\times$ and all $x\in L_{r\alpha}$. Consequently,
$L_p(\alpha)=H_p\oplus Q(\alpha)$. Since $Q(\alpha)\subset
L(\alpha)$, this yields $L(\alpha)=H\oplus Q(\alpha)$ and the
proposition follows.
\end{proof}
\begin{prop}\label{nilp}
Let $\mathcal{N}(H)$ denote the set of all $p$-nilpotent elements of
$L_p$ contained in $H$. Then the following hold:
\begin{itemize}
\item[(1)] $\mathcal{N}(H)$ is a $3$-dimensional subspace of $H$.

\smallskip

\item [(2)] There exists a unique $2$-dimensional subspace $H_{(-1)}$ in
$\mathcal{N}(H)$ satisfying the condition
$\big[H_{(-1)},H_{(-1)}\big]\subset \mathcal{N}(H)$. Moreover,
$\big[H_{(-1)},\big[H_{(-1)},H_{(-1)}\big]\big]=H^3$.

\smallskip

\item[(3)] For every $\alpha\in\Gamma$ the subspace $H_{(-1)}+Q(\alpha)$ is stable under
the adjoint action of $Q(\alpha)$.
\end{itemize}
\end{prop}
\begin{proof}
Jacobson's formula together with (\ref{pth power}) and the
multiplication table for $R$ shows that the subspace
$N:=Fx_1^{(2)}(1+x_2)^2\oplus Fx_1^{(3)}(1+x_2)^3)\oplus Fw$
consists of $p$-nilpotent elements of $R$. On the other hand, it is
clear from our remarks in the proof of Proposition~\ref{linear} that
$\Theta(H_p)=\,\Theta(T)\oplus N$. Since $\Theta(T)$ is a torus,
this entails that $N$ coincides with the set of all $p$-nilpotent
elements of the restricted Lie algebra $\Theta(H_p)$. Since
$\Theta:\,L_p(\alpha)\stackrel{\sim}{\longrightarrow} R$ is an
isomorphism of restricted Lie algebras, we deduce that
$\mathcal{N}(H)=\Theta^{-1}(N)$ is a $3$-dimensional subspace of
$H$.

The elements $D_H(x_1^{(2)}(1+x_2)^2)$ and
$D_H((x_1^{(3)}(1+x_2)^3)$ of the Hamiltonian algebra
$H(2;\un{1})^{(2)}$ commute. Therefore, in our central extension $R$
we have the equality
\begin{equation}\label{ext}
\big[x_1^{(2)}(1+x_2)^2,x_1^{(3)}(1+x_2)^3)\big]\,=\,\big[x_1^{(2)},x_2^{(3)}\big]\,=\,z.
\end{equation}
Now take any linearly independent elements
$u_1=a_1x_1^{(2)}(1+x_2)^2+b_1x_1^{(3)}(1+x_2)^3+c_1w$ and
$u_2=a_2x_1^{(2)}(1+x_2)^2+b_2x_1^{(3)}(1+x_2)^3+c_2w$ in $N$ such
that $[u_1,u_2]\in N$. Then (\ref{ext}) together with (\ref{w2}) and
(\ref{w3}) yields
$$
N\ni
[u_1,u_2]=(a_1b_2-a_2b_1)z+(a_1c_2-a_2c_1)x_1(1+x_2)+(b_1c_2-b_2c_1)x_1^{(2)}(1+x_2)^2,
$$
forcing $a_1b_2=a_2b_1$ and $a_1c_2=a_2c_1$. If  $a_1\ne 0$, then
$u_2=\frac{a_2}{a_1}u_2$ which is false. Therefore, $a_1=0$. Arguing
similarly, one obtains $a_2=0$. This shows that
$H_{(-1)}:=\Theta^{-1}(Fx_1^{(3)}(1+x_2)^3\oplus Fw)$ is the only
$2$-dimensional subspace of $\mathcal{N}(H)$ with the property that
$\big[H_{(-1)},H_{(-1)}\big]\subset \mathcal{N}(H)$. Combining
(\ref{w3}), (\ref{w2}) and (\ref{ext}) one derives that
$\big[H_{(-1)},\big[H_{(-1)},H_{(-1)}\big]\big]=H^3$.

Using the spanning set for $\Theta(Q(\alpha))$ displayed in part~(a)
the proof of Proposition~\ref{Q-alpha} and the multiplication table
for $R$, it is routine to check that
$$[\Theta(Q(\alpha)),Fx_1^{(3)}(1+x_2)^3\oplus Fw]\subset
\Theta(Q(\alpha))+Fx_1^{(3)}(1+x_2)^3\oplus Fw.$$ This implies that
$H_{(-1)}+Q(\alpha)$ is invariant under the adjoint action of
$Q(\alpha)$.
\end{proof}
\section{\bf Conclusion}
For any $\gamma\in\Gamma$ we fix a map $l_\gamma\colon\,
L_\gamma\rightarrow H$ satisfying the conditions of
Proposition~\ref{linear}. Given $x\in L_\gamma$ we set
$\widetilde{x}:=x-l_\gamma(x)$, an element of $Q(\alpha)$. Define
$$
L_{(0)}\,:=\,\textstyle{\sum}_{\gamma\in\Gamma}\,Q(\gamma),
$$ a subspace of $L$. We are going to show that $L_{(0)}$ is actually a
subalgebra of $L$. Since it follows from Remark~\ref{sub} that
$[Q(\gamma),Q(\gamma)]\subset L_{(0)}$ for all $\gamma\in\Gamma$, we
just need to check that $[Q(\alpha),Q(\beta)]\subset L_{(0)}$ for
all $\mathbb{F}_p$-independent $\alpha,\beta\in\Gamma$.
\begin{lemm}\label{L(0)1}
Let $(\alpha,\beta)$ be an arbitrary Melikian pair in $\Gamma^2$ and
let $x\in L_\alpha$, $y\in L_\beta$. Then
$[\widetilde{x},\widetilde{y}]\in L_{(0)}$ and
$$
[\widetilde{x},\widetilde{y}]\,\equiv\,\widetilde{[x,y]}\quad\ \Big(
\mathrm{mod}\ \,Q(\alpha)+Q(\beta)\Big).
$$
\end{lemm}
\begin{proof}
Set $\Delta:=\{\alpha\}\cup (\beta+\mathbb{F}_p\alpha)$.
Proposition~\ref{linear} says that $L(\delta)=H\oplus Q(\delta)$ for
any $\delta\in\Delta$. In conjunction with
Proposition~\ref{Q-alpha}, this gives
\begin{eqnarray}\label{q-delta}
\big(L_p(\alpha,\beta)^{(1)}\big)(\delta)\,=\,\big(H\cap
L_p(\alpha,\beta)^{(1)}\big)\oplus
Q(\delta)\qquad\quad\,(\forall\,\delta\in\Delta).
\end{eqnarray}
Recall that $\Phi\colon\,L_p(\alpha,\beta)^{(1)}
\stackrel{\sim}{\longrightarrow}\,\mathcal{M}$ is a Lie algebra
isomorphism taking $H\cap L_p(\alpha,\beta)^{(1)}$  onto
$\mathfrak{c}_{\mathcal{M}}(\mathfrak{t})$ and
$\big(L_p(\alpha,\beta)^{(1)}\big)_{(0)}$ onto $\mathcal{M}_{(0)}$.
Therefore,
\begin{eqnarray}\label{dim}
\dim H\cap L_p(\alpha,\beta)^{(1)}=\,5,\qquad
\dim\big(L_p(\alpha,\beta)^{(1)}\big)_{(0)}=\,120.
\end{eqnarray}
Combining (\ref{dim}) and (\ref{q-delta}) we now deduce that for
every $\delta\in\Delta$ the subalgebra $Q(\delta)\,=\,
L_p(\delta)\cap\big(L_p(\alpha,\beta)^{(1)}\big)_{(0)}$ has
codimension $5$ in the $1$-section
$\big(L_p(\alpha,\beta)^{(1)}\big)(\delta)$. Since
$L_p(\alpha,\beta)^{(1)}\cong\,\mathcal{M}$, it follows from
\cite[Lemmas~4.1 \& 4.4]{P94}, for instance, that
$\dim\big(L_p(\alpha,\beta)^{(1)}\big)(\delta)\,=\,25$. Therefore,
$\dim Q(\delta)=20$ for all $\delta\in\Delta$.

For any $\mu\in\Delta$ one has
$$Q(\mu)\cap
\big(\textstyle{\sum}_{\delta\in\Delta\setminus\{\mu\}}\,Q(\delta)\big)
\subset\,Q(\mu)\cap(\textstyle{\sum}_{\delta\in\Delta\setminus\{\mu\}}\,L(\delta)\big)
\subset\,Q(\mu)\cap H=\,(0).$$ This shows that the sum $Q(\alpha)+
\textstyle{\sum}_{j=0}^4\,Q(\beta+j\alpha)$ is direct. But then
$$
\dim\Big(Q(\alpha)\oplus
\textstyle{\bigoplus}_{j=0}^4\,Q(\beta+j\alpha)\Big)=\, 6\cdot
20\,=\,120\,=\,\dim \big(L_p(\alpha,\beta)^{(1)}\big)_{(0)},
$$
implying that
$\big(L_p(\alpha,\beta)^{(1)}\big)_{(0)}=\,Q(\alpha)+\textstyle{\sum}_{j\in\mathbb{F}_p}\,Q(\beta+j\alpha).$
As a consequence,
\begin{eqnarray}
\qquad\qquad\ \big[Q(\alpha),\,Q(\beta)\big]&\subset&
\Big[\big(L_p(\alpha,\beta)^{(1)}\big)_{(0)},\big(L_p(\alpha,\beta)^{(1)}\big)_{(0)}\Big]
\,\subset\,\big(L_p(\alpha,\beta)^{(1)}\big)_{(0)}\\
&=&Q(\alpha)+\,\textstyle{\bigoplus}_{j=0}^4\,Q(\beta+j\alpha)\,\subset\,\,L_{(0)}.\nonumber
\end{eqnarray}
This shows that $[\widetilde{x},\widetilde{y}]\in L_{(0)}$.
Computing modulo $Q(\alpha)+Q(\beta)$ we get
\begin{eqnarray*}
[\widetilde{x},\widetilde{y}]&=&\big([x,y]-l_{\alpha+\beta}([x,y]-
[x,l_\beta(y)]+l_\alpha([x,l_\beta(y)]-[l_\alpha(x),y]+l_\beta([l_\alpha(x),y]\big)
\\
&+&[l_\alpha(x),l_\beta(y)]+\big(l_{\alpha+\beta}([x,y])-
l_\alpha([x,l_\beta(y)])-l_\beta([l_\alpha(x),y])\big)\\
&=&
\widetilde{[x,y]}-\widetilde{[x,l_\beta(y)]}-\widetilde{[l_\alpha(x),y]}+\widetilde{h}\\
&\equiv&\widetilde{[x,y]}+\widetilde{h},
\end{eqnarray*}
where
$\widetilde{h}=l_{\alpha+\beta}([x,y])-l_\alpha([x,l_\beta(y)])-l_\beta([l_\alpha(x),y])+[l_\alpha(x),l_\beta(y)]$.
As $\widetilde{[x,y]}\in L_{(0)}$, it must be that $\widetilde{h}\in
H\cap L_{(0)}=\,H\cap\big(\sum_{\gamma\in\Gamma}\,Q(\gamma)\big).$
Expressing
$\widetilde{h}=\sum_{\gamma\in\Gamma}\,(v_\gamma-l_\gamma(v))$ with
$v_\gamma\in L_\gamma$ we see that $v_\gamma=0$ for all $\gamma$,
whence $l_\gamma(v_\gamma)=0$ and $\widetilde{h}=0$. The result
follows.
\end{proof}
\begin{theo}\label{L(0)2}
$L_{(0)}$ is a proper subalgebra of $L$.
\end{theo}
\begin{proof}
By our earlier remark in this section, we need to show that
$[Q(\alpha),Q(\beta)]\subset L_{(0)}$ for all pairs
$(\alpha,\beta)\in\Gamma^2$ such that $\alpha$ and $\beta$ are
$\mathbb{F}_p$-independent. If $(\alpha,\beta)$ is a Melikian pair,
this follows from Lemma~\ref{L(0)1}.

Take any $\mathbb{F}_p$-independent $\alpha,\beta\in\Gamma$ for
which $(\alpha,\beta)$ is not a Melikian pair. Then
$H^3\cap\ker\alpha\,=\,H^3\cap\ker\beta$; see Lemma~\ref{melpair1}.
Recall that $H^3\cap \ker\alpha=Fh_\alpha$ for some nonzero
$h_\alpha\in H^3$. Put
$\Gamma(\alpha):=\{\gamma\in\Gamma\,|\,\,\gamma(h_\alpha)\ne 0\}$.
Since $H^3\subset T$, the set $\Gamma(\alpha)$ is nonempty. Then it
follows from Schue's lemma~\cite[Prop.~1.3.6(1)]{St04} that
\begin{equation}\label{L(0)3}
L_\beta\,=\,\textstyle{\sum}_{\gamma\in\Gamma(\alpha)}\,[L_\gamma,L_{\beta-\gamma}].
\end{equation}

Let $\gamma$ be an arbitrary root in $\Gamma(\alpha)$. Since
$\alpha(h_\alpha)=\beta(h_\alpha)=0$, it is immediate from
Lemma~\ref{melpair1} that $(\alpha,\gamma)$  and
$(\alpha,\beta-\gamma)$ are Melikian pairs in $\Gamma^2$.

Suppose $(\alpha+\gamma,\beta-\gamma)$ is not a Melikian pair. Then
$(\beta-\gamma)(h_{\alpha+\gamma})=0$ by Lemma~\ref{melpair1}. As
$(\beta-\gamma)(h_\alpha)=-\gamma(h_\alpha)\ne 0$ and $\dim H^3=2$
by Theorem~\ref{sum}(2), this yields $H^3=Fh_\alpha\oplus
Fh_{\alpha+\gamma}$. Also, $(\alpha+\beta)(h_\alpha)=0$ and
$(\alpha+\beta)(h_{\alpha+\gamma})=\big((\alpha+\gamma)+(\beta-\gamma)\big)(h_{\alpha+\gamma})=0$
by our assumption on $(\alpha+\gamma,\beta-\gamma)$. This shows that
$\alpha+\beta$ vanishes on $H^3$ and hence on $(H^3)_p=T$; see
Theorem~\ref{sum}(2). But then $\alpha+\beta=0$, a contradiction.
Thus, $(\alpha+\gamma,\beta-\gamma)$ is a Melikian pair.

If $(\gamma,\alpha+\beta-\gamma)$ is not a Melikian pair, then
$\gamma(h_{\alpha+\beta-\gamma})=0$. As $\gamma\in\Gamma(\alpha)$,
we then have $H^3=Fh_\alpha\oplus Fh_{\alpha+\beta-\gamma}.$ But
then $\alpha+\beta=\gamma+(\alpha+\beta-\gamma)$ vanishes on
$(H^3)_p$, a contradiction. So $(\gamma,\alpha+\beta-\gamma)$ is a
Melikian pair, too.

We now take arbitrary $u\in L_\alpha$ and $v\in L_\beta$. By
(\ref{L(0)3}), there exist
$\gamma_1,\ldots,\gamma_N\in\Gamma(\alpha)$ such that
$v=\sum_{i=1}^N\,[x_i,y_i]$ for some $x_i\in L_{\gamma_i}$ and
$y_i\in L_{\beta-\gamma_i}$, where $1\le i\le N$. Applying
Lemma~\ref{L(0)1} and the preceding remarks we obtain
\begin{eqnarray*}
[\widetilde{u},\widetilde{v}]&\in&
\textstyle{\sum}_{i=1}^N\,[\widetilde{u},[\widetilde{x_i},\widetilde{y_i}]]\,+\,
\sum_{i=1}^N\,[Q(\alpha),Q(\gamma_i)+Q(\beta-\gamma_i)]\\
&\subset&\textstyle{\sum}_{i=1}^N\,
\big([[\widetilde{u},\widetilde{x_i}],\widetilde{y_i}]
+[\widetilde{x_i},[\widetilde{u},\widetilde{y_i}]]\big)+L_{(0)}\\
&\subset&\textstyle{\sum}_{i=1}^N\,\big([Q(\alpha+\gamma_i),Q(\beta-\gamma_i)]\,+\,
[Q(\gamma_i),Q(\alpha+\beta-\gamma_i)]\big)+L_{(0)}\subset\,
L_{(0)}.
\end{eqnarray*}
Consequently, $[Q(\alpha),Q(\beta)]\subset L_{(0)}$ in all cases.
The argument at the end of the proof of Lemma~\ref{L(0)1} shows that
$L_{(0)}\cap H=(0)$. Hence $L_{(0)}$ is a proper subalgebra of $L$.
\end{proof}
Recall the subspace $H_{(-1)}$ from Proposition~\ref{nilp}(2).
According to Proposition~\ref{nilp}(3), $\big[Q(\gamma),
H_{(-1)}]\subset H_{(-1)}+Q(\gamma)\subset H_{(-1)}+L_{(0)}$ for all
$\gamma\in\Gamma$. In view of Theorem~\ref{L(0)3}, this means that
$$
\big[L_{(0)},\,H_{(-1)}+L_{(0)}\big]\,=\,\big[\textstyle{\sum}_{\gamma\in\Gamma}\,\,Q(\gamma),H_{(-1)}+
\textstyle{\sum}_{\delta\in\Gamma}\,Q(\delta)\big] \subset\,
H_{(-1)}+L_{(0)}.
$$
Thus, $L_{(-1)}:=\,H_{(-1)}+L_{(0)}$ is stable under the adjoint
action of the subalgebra $L_{(0)}$.

\medskip

We have finally come to the end of this tale. Let $L'$ denote the
subalgebra of $L$ generated by $L_{(-1)}$. Proposition~\ref{nilp}(2)
shows that $H^3\subset L'$. Then the $p$-envelope of $L'$ in $L_p$
contains $(H^3)_p=T$; see Theorem~\ref{sum}(2). As a consequence,
$L'$ is $T$-stable. Let $\gamma$ be any root in $\Gamma$. Then
$[T,x-l_\gamma(x)]\subset L'$ for all $x\in L_\gamma$, implying
$L_\gamma\subset L'$. As this holds for all $\gamma\in\Gamma$ and
$L$ is simple, we deduce that $L'=L$.

It follows from Theorem~\ref{L(0)3} that $L_{(-1)}\supsetneq
L_{(0)}$. We now consider the standard filtration of $L$ associated
with the pair $\big(L_{(-1)},L_{(0)}\big)$ (it is defined
recursively by setting $L_{(i)}:=\,\{x\in
L_{(i-1)}\,|\,\,[x,L_{(i-1)}]\subset L_{(i-1)}\}$ and
$L_{(-i)}:=\,\big[L_{(-1)},L_{(-i+1)}\big]+L_{(-i+1)}$ for all
$i>0$). Since $L$ is simple and finite-dimensional, this filtration
is exhaustive and separating. Let $G=\bigoplus_{i\in\,\Z}\,G_i$
denote the associated graded Lie algebra, where
$G_i\,=\,\mathrm{gr}_i(L)\,=\,L_{(i)}/L_{(i+1)}$.

Since $L_{(-1)}=H_{(-1)}+L_{(0)}$, we have that
$L_{(-i)}=\,L_{(0)}\,+\,\sum_{j=1}^i\,(H_{(-1)})^j\,$ for all $i>0$.
Since $(H_{(-1)})^3\subset H^3\subset\mathfrak{z}(H)$ by
Theorem~\ref{sum}(2), this shows that $L_{(-4)}=\,L_{(-3)}$, i.e.
$G_{-4}=(0)$. As $\dim H_{(-1)}=2$, we obtain by the same token that
$\dim G_{-2}\le 1$ and $\dim G_{-3}\le 2$.

Let $(\alpha,\beta)$ be any Melikian pair in $\Gamma^2$. By our
remarks in the proof of Lemma~\ref{L(0)1},
$\big(L_p(\alpha,\beta)^{(1)}\big)\cap L_{(0)}\,=\,
\big(L_p(\alpha,\beta)^{(1)}\big)_{(0)}$, while from the explicit
description of $\Theta(H_{(-1)})$ in the proof of
Proposition~\ref{nilp} and Proposition~\ref{iso} we see that
\begin{equation}\label{filt}
H_{(-1)}+\big(L_p(\alpha,\beta)^{(1)}\big)_{(0)}=\,
\big(L_p(\alpha,\beta)^{(1)}\big)_{(-1)}. \end{equation} In
particular, $H_{(-1)}\subset L_p(\alpha,\beta)^{(1)}$. It follows
that the filtration of $L_p(\alpha,\beta)^{(1)}\cong\,\mathcal{M}$
induced by that of $L$ has the property that
$$
L_{(i)}\,=\,\big(L_p(\alpha,\beta)^{(1)}\cap
L_{(i)}\big)\,+\,L_{(i-1)},\qquad i=-1,-2,-3.
$$
In view of (\ref{filt}), this entails that $\dim G_{-1}=\dim
G_{-3}=2$ and $\dim G_{_2}=1.$

As $\dim G_{-1}=2$, and $G_0$ acts faithfully on $G_{-1}$, we have
an embedding $G_0\subset \mathfrak{gl}(2)$. As
$\big(L_p(\alpha,\beta)^{(1)}\big)_{(0)}$ acts on
$\big(L_p(\alpha,\beta)^{(1)}\big)_{(-1)}/\big(L_p(\alpha,\beta)^{(1)}\big)_{(0)}$
as $\mathfrak{gl}(2)$, it follows from (\ref{filt}) that
$\big(L_{(0)}\cap L_p(\alpha,\beta)^{(1)}\big)/\big(L_{(1)}\cap
L_p(\alpha,\beta)^{(1)}\big) \cong\,\mathfrak{gl}(2)$. As a
consequence, $G_0\cong\mathfrak{gl}(2)$. Finally, (\ref{filt})
yields that $L_p(\alpha,\beta)^{(1)}\cap L_{(4)}\ne (0)$, giving
$G_4\ne (0)$.

Applying \cite[Thm.~5.4.1]{St04} we now obtain that the graded Lie
algebra $G$ is isomorphic to a Melikian algebra $\mathcal{M}(m,n)$
regarded with its natural grading. By a result of Kuznetsov
\cite{Ku}, any filtered deformation of the naturally graded Lie
algebra $\mathcal{M}(m,n)$ is isomorphic to $\mathcal{M}(m,n)$; see
\cite[Thm.~6.7.3]{St04}. Thus, $L\cong \mathcal{M}(m,n)$, completing
the proof of Theorem~\ref{1.2}.
\begin{cor}
Let $L$ be a finite-dimensional simple Lie algebra of Cartan type
over an algebraically closed field of characteristic $p>3$ and let
$T$ be a torus of maximal dimension in $L_p\subset\Der L$. Then the
centralizer of $T$ in $L_p$ acts triangulably on $L$.
\end{cor}
\begin{proof}
This is an immediate consequence of \cite[Thm.~A]{PS4} and
Theorem~\ref{1.2}.
\end{proof}

\end{document}